\documentclass[a4 paper, 11pt,reqno]{amsart}

\usepackage{amssymb}
\usepackage{amscd}
\usepackage{amsfonts}
\usepackage{mathrsfs}
\usepackage{setspace}
\usepackage{version}
\usepackage{mathtools}
\usepackage{dsfont}

\usepackage{amsmath,amssymb,amscd,verbatim,bbm,graphicx,enumerate,tikz}

\definecolor{crimson}{rgb}{0.86, 0.08, 0.24}
\definecolor{bleudefrance}{rgb}{0.19, 0.5, 0.91}
\usepackage[pdftex,colorlinks,linkcolor=crimson,citecolor=bleudefrance,filecolor=black]{hyperref}

\usepackage{float}
\usepackage{tikz-cd}

\usepackage{tikz}
\usetikzlibrary{calc}
\usetikzlibrary{shadings,intersections}
\usetikzlibrary{decorations.text}
\usepackage{pgfplots,wrapfig}
\usepackage{lipsum}

\usepackage{xcolor}

\pgfplotsset{compat=1.17}

\theoremstyle{plain}
\numberwithin{equation}{section} 
\newtheorem{theorem}{Theorem}[section]
\newtheorem{proposition}[theorem]{Proposition}
\newtheorem{lemma}[theorem]{Lemma}

\newtheorem{corollary}[theorem]{Corollary}
\newtheorem{definition}[theorem]{Definition}

\newtheorem{claim}[theorem]{Claim}
\theoremstyle{remark}
\newtheorem{remark}[theorem]{Remark}

\newtheorem{example}[theorem]{Example}
\usepackage{soul}

\renewcommand{\leq}{\leqslant}
\renewcommand{\geq}{\geqslant}
\newsavebox{\proofbox}
\savebox{\proofbox}{\begin{picture}(7,7)  \put(0,0){\framebox(7,7){}}\end{picture}}

\newcommand\E{\mathbb{E}}
\newcommand\Z{\mathbb{Z}}
\newcommand\R{\mathbb{R}}

\newcommand\N{\mathbb{N}}

\newcommand\Gr{\operatorname{Gr}}
\newcommand\SL{\operatorname{SL}}

\newcommand\GL{\operatorname{GL}}

\newcommand\PGL{\operatorname{PGL}}

\newcommand\Mat{\operatorname{Mat}}

\newcommand\Isom{\operatorname{Isom}}
\newcommand\Lip{\operatorname{Lip}}

\renewcommand\P{\mathbb{P}}

\newcommand\id{\operatorname{id}}

\newcommand\acts{\operatorname{\curvearrowright}}

\usepackage{accents}
\newlength{\dhatheight}

\title[Counting and Boundary Limit theorems for  Gromov-hyperbolic groups]{Counting and Boundary Limit theorems for representations of Gromov-hyperbolic groups}

\author{Stephen Cantrell}
\address{Department of mathematics, University of Chicago, Chicago, Illinois 60637, USA}
\email{scantrell@uchicago.edu}
\thanks{}

\author{Cagri Sert} 
\address{Institut f\"{u}r Mathematik, Universit\"{a}t Z\"{u}rich, Winterthurerstrasse 190, 8057 Z\"{u}rich, Switzerland}
\email{cagri.sert@math.uzh.ch}
\thanks{C.S. is supported by the grant SNF Ambizione 193481.}

\begin{document}



\maketitle

\begin{center}
\textit{Dedicated to the memory of \'Emile Le Page}
\end{center}

\begin{abstract}
Given a Gromov-hyperbolic group $G$ endowed with a finite symmetric generating set, we study the statistics of counting measures on the spheres of the associated Cayley graph under linear representations of $G$. More generally, we obtain a weak law of large numbers for subadditive functions, echoing the classical Fekete lemma. For strongly irreducible and proximal representations, we prove a counting central limit theorem with a Berry--Esseen type error rate and exponential large deviation estimates. Moreover, in the same setting, we show convergence of interpolated normalized matrix norms along geodesic rays to Brownian motion and a functional law of iterated logarithm, paralleling the analogous results in the theory of random matrix products. Our counting large deviation estimates provide a positive answer to a question of Kaimanovich--Kapovich--Schupp.  In most cases, our counting limit theorems will be obtained from stronger almost sure limit laws for Patterson--Sullivan measures on the boundary of the group.
\end{abstract}

{
\hypersetup{linkcolor=black}
\tableofcontents
}

\section{Introduction}

Let $\mu$ be a probability measure on $G=\GL_d(\R)$ and $(X_i)_{i \in \N}$ be a sequence of independent $G$-valued random variables with distribution $\mu$. Let $Y_n$ denote the $n^{th}$-step of the random product $X_n \ldots X_1$. 
The theory of random matrix products is concerned with studying the asymptotic behaviour of  $Y_n$, for example, by investigating limit theorems (law of large numbers, central limit theorem, large deviations, etc.) for numerical quantities associated to matrices such as the operator norm $\|Y_n\|$ or spectral radius. The most intricate part of the theory is when the probability measure $\mu$ is finitely or countably supported say inside a countable group $\Gamma<G$. In that case, one has to deal with the possible singular behaviour of the countable subgroup $\Gamma$ inside the ambient group $\GL_d(\R)$. After pioneering works of Furstenberg, Kesten  \cite{furstenberg.non.commuting,furstenberg-kesten} and several others, significant progress was made by Le Page \cite{lepage} in early '80s however many open questions still persist.

The theory of random matrix products provides a way to express asymptotic behaviour of large elements of $\Gamma$ in $\GL_d(\R)$. Indeed, for a finitely supported probability measure $\mu$ as above, the probabilistic description of the asymptotic behaviour of $Y_n$ is a problem of \textit{symbolic counting}, i.e.~ counting with certain multiplicities. A related but different way to study the asymptotic behaviour of elements of $\Gamma$, perhaps more directly related to group $\Gamma$ itself rather than its symbolic representation, would be to study statistics of asymptotics of \textit{actual elements} of $\Gamma$. However, due to disparate algebraico-combinatorial structure of different countable groups $\Gamma<\GL_d(\R)$ such a general description is notoriously harder to obtain. Accordingly, such counting asymptotics results are much less developed compared to the theory of random matrix products.

In this article, we will be interested in describing \textit{counting asymptotics} and \textit{boundary limit laws} for representations of Gromov-hyperbolic groups. These include virtually free groups, cocompact isometry groups of negatively curved geodesic spaces, groups with small cancellation property etc. From another perspective, in some probabilistic models (e.g.~ random groups), most finitely presented groups are Gromov-hyperbolic.
We shall prove four main limit theorems:\\
\indent \textbullet ${}$ \textit{Law of large numbers for subadditive functions:} this is of more general nature compared to the following results, it holds for any real-valued subadditive function on $\Gamma$.\\
\indent The next results concern matrix representations of Gromov-hyperbolic groups, they hold under the standard (strong) irreducibility and proximality assumptions of random matrix products theory:\\
\indent \textbullet ${}$ \textit{Exponential large deviation estimates for counting:} this one refines the aforementioned law of large numbers in the setting of matrix representations and provides a positive answer to a question raised by Kaimanovich--Kapovich--Schupp \cite[Problem 9.3]{kaimanovich-kapovich-schupp}. Apart from representations, we also prove counting large deviation estimates for isometric actions on Gromov-hyperbolic spaces.\\[1pt]
\indent \textbullet ${}$ \textit{Counting central limit theorem with Berry--Esseen type error term}.\\[1pt]
\indent \textbullet ${}$ \textit{Convergence of normalized interpolations along geodesic rays under a Patterson--Sullivan measure to the standard Brownian motion}: this one is of a different nature, it pertains to a measure on the boundary rather than counting. In fact, the first three points above also have corresponding boundary analogues which, beyond interest in themselves, serve as a tools to prove them.

Somewhat ironically, the key mechanism that will allow us to obtain these deterministic counting asymptotics is the inherent dynamical or probabilistic structure of the Gromov-hyperbolic groups. Indeed, as realised by Cannon \cite{cannon} and Gromov \cite{gromov}, the geodesics on such a group can be coded by a finite state automaton. This makes it possible to approach the deterministic data of these groups by (a collection of) well-behaved stochastic processes, namely Markov chains. For example, for the last three results mentioned above, it enables us to employ probabilistic results of \textit{Markovian random matrix products} (mainly due to Bougerol \cite{bougerol.survey,bougerol.thm.limite, bougerol.comparaison} and Guivarc'h \cite{guivarch.markov}; we also develop some of them further) to the deterministic counting results. This transfer, however, requires handling some difficulties which we manage to do by, among others, elaborating on techniques developed by Calegari--Fujiwara \cite{calegari-fujiwara} (generally) and Gekhtman--Taylor--Tiozzo \cite{GTTCLT} (for the central limit theorem). The deterministic nature of our results, in particular the fact that we do not induce randomness using an external source (like a subshift of finite type \cite{calegari-fujiwara, pollicott.sharp.2011}
) is of particular interest.
We shall comment more on each of our results and on the past works below, let us now continue by stating our theorems and remarks more precisely.

\bigskip

Let $\Gamma$ be a finitely generated group and $S$ a  generating set for $\Gamma$ -- all considered generating sets will be assumed to be finite and symmetric. The choice of $S$ makes $\Gamma$ into a metric space by considering the associated length function on $\Gamma$, namely $|g|_S=\min \{n \in \N \, | \, s_1\ldots s_n=g, \, s_i \in S\}$ and for $g,h \in \Gamma$ setting the (left) metric to be $d_S(g,h)=|g^{-1}h|_S$. Recall that for $\Delta \geq 0$, by a $\Delta$-hyperbolic metric space $(M,d)$, we understand a metric space such that for every $x,y,z,o \in M$, we have $(x,y)_o \geq (x,z)_o \wedge (z,y)_o -\Delta$, where $(\cdot,\cdot)_{\cdot}$ is the Gromov product given by $(x,y)_o=\frac{1}{2}(d(x,o)+d(y,o)-d(x,y))$.
The group $\Gamma$ is said to be \textit{Gromov-hyperbolic} if there exists a real constant $\Delta \geq 0$ and a generating set $S$ such that the associated metric space is $\Delta$-hyperbolic. Given a generating set $S \subseteq \Gamma$, we write $S_n$ for the sphere of radius $n$ for the associated metric, namely $S_n:=\{g \in \Gamma: |g|_S = n\}$. Finally a Gromov-hyperbolic group $\Gamma$ is said to be \textit{non-elementary} if it is not virtually cyclic, i.e.~ does not contain a cyclic subgroup of finite index.

\subsection{Convergence of subadditive spherical averages}
A real-valued function $\varphi$ on a group $\Gamma$ is called \textit{subadditive}, if for every $g,h \in \Gamma$, we have $\varphi(gh) \leq \varphi(g) + \varphi(h)$. The following is our first result.  

\begin{theorem}[Weak law of large numbers for subadditive spherical averages]\label{thm:lln}
Let $\Gamma$ be a non-elementary Gromov-hyperbolic group endowed with a generating set $S$ and $\varphi : \Gamma \to \mathbb{R}$ is subadditive function on $\Gamma$. Then, there exists $\Lambda \geq 0$ such that
for any $\epsilon >0$,
\[
\lim_{n\to\infty} \frac{1}{\#S_n} \#\left\{ g \in S_n : \left| \frac{\varphi(g)}{n} - \Lambda \right|> \epsilon\right\} =0.
\]
In particular,
\[
\lim_{n\to\infty} \frac{1}{n}
\sum_{|g|_S = n} \frac{1}{\#S_n}\varphi(g) = \Lambda.
\]
\end{theorem}

Unlike our other results below where we will specialize to linear representations, the generality of the previous result goes far beyond; we comment on subadditive functions of different nature in Remark \ref{rk.subad.functions.intro} below. We note that the first statement above is precisely a weak law of large numbers whereas the second one corresponds to convergence in expectation (for a strong law, see Theorem \ref{thm:rayconvergence}). Finally, notice the curious analogy with the classical Fekete lemma which matches this convergence in expectation when $\Gamma=\N$ and $S=\{1\}$.

\begin{remark}[Examples of subadditive functions]\label{rk.subad.functions.intro}
Two large classes of subadditive functions contain the following.\\[2pt]
\indent 1. (Semi-norms on groups) Let $H$ be any group endowed with a semi-norm $|\cdot|$ and $\rho:\Gamma \to H$ a homomorphism (cf.~ \cite{kaimanovich-kapovich-schupp}.). The function $\varphi(\gamma):= |\rho(\gamma)|$ is clearly a subadditive function on $\Gamma$ and this construction encompasses many examples:\\[1pt]
\indent 1.a. Already in the case $H=\Gamma$, $\rho=\id$ and $|\cdot|$ any length function on $\Gamma$, the previous theorem applied to $\varphi(\cdot)=|\cdot|$ yields an asymptotic ratio $\Lambda$ between $|\cdot|_S$ and $|\cdot|$. Note that $\Lambda>0$ if, for example, $|\cdot|$ comes from a finite generating set. More generally, let $(X,d)$ be a metric space, $o \in X$ and $\Gamma \acts X$ by isometries. Then $\varphi(\gamma)=d(\gamma \cdot o, o)$ is a subadditive function.\\[1pt]
\indent 1.b. Let $H=\GL_d(\R)$, $\|\cdot\|$ an operator norm on the algebra $\Mat_d(\R)$ of matrices and $\rho:\Gamma \to H$ a representation. Then, $\varphi(\gamma):=\log \|\rho(\gamma)\|$ is an example of a subadditive function.\\[2pt]
\indent 2. (Quasi-morphisms) Another class of examples comes from the observation that Theorem \ref{thm:lln} remains valid for any function $\varphi'$ on $\Gamma$ such that $|\varphi-\varphi'|$ is bounded. In view of this, the previous result applies to any quasi-morphism (see Remark \S \ref{sec.lln}). For those, it is not hard to see that $\Lambda=0$.
\end{remark}

\begin{remark}[Possible extensions]
Using different methods that rely on the topological flow introduced by Mineyev \cite{mineyev} and studied by Tanaka in \cite{tanaka.topflow}, it might be possible to prove that Theorem \ref{thm:lln} holds when we count with respect to some other hyperbolic metrics that are not necessarily word metrics (see  \cite[Theorem 3.12]{cantrelltanaka}). We have decided not to present the proof of this result as it is not clear how to obtain more refined counting limit laws below in this more general setting. 
\end{remark}

We note that Theorem \ref{thm:lln} and its almost sure version (Theorem \ref{thm:rayconvergence} below) generalise several previous works. For example, \cite[Theorem 7.3]{countinglox} and \cite[Theorem 7.4]{tanaka.topflow} follow from the particular case where $\varphi$ is a displacement function associated to an isometric group action with additional requirements. It also generalises (without error term) \cite[Theorem 1.1]{cantrell.TAMS}. See also \S \ref{subsub.intro.as.lln}.



It would be interesting to characterise when the constant $\Lambda$ appearing in Theorem \ref{thm:lln} is strictly positive. For a subadditive function $\varphi$ coming from a semi-norm (1.~ of Remark \ref{rk.subad.functions.intro}) one can typically say more, see Proposition \ref{prop.positivity.in.section}. We will also see a characterization below in the case of strongly irreducible representations.

The rest of our counting results (except Theorem \ref{thm.gromov.counting.ld}) concern finite dimensional representations $\Gamma \to \GL_d(\R)$ of Gromov-hyperbolic groups and we now specialize to this setting.

\subsection{Counting limit theorems for representations}

Recall that a representation  $\rho: \Gamma \to \GL_d(\mathbb{R})$ is said to be \textit{strongly irreducible} if there does not exist a finite collection of proper non-trivial subspaces of $\mathbb{R}^d$ whose union is invariant under the action of $\rho(\Gamma)$. It is said to be \textit{proximal} if there exists a sequence of elements $g_n \in \rho(\Gamma)$ such that $\frac{g_n}{\|g_n\|}$ converges to a rank-one linear transformation.

\subsubsection{Positivity of average growth rate} In what follows, whenever a representation  $\rho:\Gamma \to \GL_d(\R)$ of a Gromov-hyperbolic group $\Gamma$ (equipped with a generating set) is understood, $\Lambda$ denotes the average growth rate given by applying Theorem \ref{thm:lln} to $\varphi(g) = \log\|\rho(g)\|$. Clearly, $\Lambda$ does not depend on the choice of the operator norm. The following result gives a characterization of when $\Lambda$ is positive.

\begin{proposition}\label{prop.positivity}
Let $\Gamma$ be a non-elementary Gromov-hyperbolic group, $S$ a generating set for $\Gamma$ and $\rho:\Gamma \to \GL_d(\R)$ a strongly irreducible representation. Then the constant $\Lambda \ge 0$ is strictly positive if and only if $\rho(\Gamma)$ is not relatively compact in $\PGL_d(\R)$.
\end{proposition}

This result is ultimately a consequence of Furstenberg's result \cite{furstenberg.non.commuting} on positivity of the top Lyapunov exponent for independent and identically distributed (iid) random products. However, for this statement, we additionally (need to) exploit the symmetry of the generating set since positivity of top Lyapunov exponent may fail for random products in $\GL_d(\R)$. 

Combined with Theorem \ref{thm:lln}, this result already implies that if such a $\Gamma$ is Zariski-dense in a real semisimple linear Lie group $G$, the word metric $d_S$ and any left-$G$-invariant Riemannian metric $d_G$ are Lipschitz equivalent when restricted to a large (i.e.~ full asymptotic density in the spheres $S_n$) subset of $\Gamma$. We discuss this more in the large deviation part \S \ref{subsub.large.dev.intro} below and in detail in \S \ref{subsec.LMR}.

\subsubsection{Exponential large deviation estimates}\label{subsub.large.dev.intro} Establishing the next result was one of the earlier motivations of our work. In \cite{kaimanovich-kapovich-schupp} Kaimanovich--Kapovich--Schupp  asked whether exponential large deviation estimates can be obtained for free groups equipped with certain generating sets. The following therefore provides a class of such examples with considerably less restrictions on the underlying group and generating set (see \S \ref{sec.last}).

\begin{theorem} \label{thm:ldt}
Let $\Gamma$ be a non-elementary Gromov-hyperbolic group, $S$ a generating set  and $\rho:\Gamma \to \GL_d(\R)$ a strongly irreducible and proximal representation. Then for every $\epsilon > 0$,
\[
\limsup_{n\to\infty} \frac{1}{n} \log \left( \frac{1}{\#S_n} \#\left\{ g \in S_n : \left| \frac{\log\|\rho(g)\|}{n} - \Lambda \right| > \epsilon \right\}\right) < 0.
\]
Here $\Lambda >0$ is the constant obtained from applying Theorem \ref{thm:lln}.
\end{theorem}

This result is analogous to a result of Le Page \cite{lepage} (see \cite[Theorem 6.2]{bougerol.book}) for iid random matrix products. We note that a multi-dimensional version, a consequence which pertains to the exponential concentration of the multi-dimensional Cartan projection around a Lyapunov vector (in the spirit of Benoist--Quint \cite[Theorem 13.17.(iii)]{bq.book}), follows immediately from this result. As we discuss further in \S \ref{subsec.unique.max}, this establishes the uniqueness of maximum of the growth indicator function considered in \cite{sert.thesis}. 

Furthermore, as discussed in \S \ref{subsec.LMR}, it follows from this result and positivity of $\Lambda$ that, when $\Gamma$ is a Zariski-dense subgroup of a real linear semisimple Lie group $G$, the word-metric $d_S$ on $\Gamma$ and and left-$G$-invariant Riemannian metric $d_G$ on $G$ coming from a Killing form are Lipschitz equivalent when restricted to an $S$-exponentially generic subset of $\Gamma$ (see Corollary \ref{corol.LMR})\footnote{Notice that in general even if $\Gamma$ is a (non-uniform) lattice in $G$, one cannot hope to have this Lipschitz equivalence on whole of $\Gamma$. Such a global Lipschitz equivalence holds for higher-rank irreducible lattices \cite{LMR1,LMR2} which are of course not Gromov-hyperbolic.}.
It may be tempting to try to prove this result using a random walk approach. 
However, to do this, one would need to construct a probability measure $\mu$ on $S$ for which we have the equality $h_\mu =\ell_\mu \log \lambda$ in the fundamental inequality $h_\mu \leq \ell_\mu \log \lambda$ of Guivarc'h (here, $h_\mu$ is the asymptotic (Avez) entropy of $\mu$, $\ell_\mu$ is its drift and $\lambda$ exponential growth rate of $S$-spheres in $\Gamma$, see e.g.~\cite{gmm}). The reason for this is that a probability measure $\mu$ with $h_\mu<\ell_\mu \log \lambda$ will only see an exponentially small part of the spheres of $S$. On the other hand, as shown in \cite[Theorem 1.3]{gmm}, the equality case $h_\mu=\ell_\mu \log\lambda$ is very rigid and forces the ambient group to be virtually free.

Regarding its proof, Theorem \ref{thm:ldt} will be deduced from an almost-sure version of it (with respect to geodesic rays following Patterson--Sullivan measure class on boundary) which we will discuss below (Theorem \ref{thm:ldtboundary}). 

\bigskip

The following result establishes exponential counting large deviation estimates in another setting, that of isometric actions on Gromov-hyperbolic spaces. This setting has recently attracted much attention both from probabilistic \cite{aoun-sert.gromov, baik-choi-kim, bq.clt.hyperbolic, BMSS, gouezel.21, maher-tiozzo} and counting \cite{calegari-fujiwara, cantrell.TAMS, cantrell.georays, cantrelltanaka, GTTCLT,countinglox,yang} perspectives. To state our result, recall that the action of a group $\Gamma$ on a Gromov-hyperbolic space $H$ by isometries is said to be \textit{non-elementary} if it there exists $\gamma_1,\gamma_2 \in \Gamma$ acting as loxodromic elements (see \S \ref{subsec.ld.isometry}) with disjoint pairs of fixed points on the Gromov boundary of $H$. 

\begin{theorem}\label{thm.gromov.counting.ld}
Let $\Gamma$ be a Gromov-hyperbolic group, $S$ a generating set of $\Gamma$ and $(H,d)$ a geodesic Gromov-hyperbolic space and $o \in H$ a basepoint. Suppose that $\Gamma$ acts on $H$ by isometries and that the action is non-elementary. Then, there exists a constant $\Lambda>0$ such that for every $\epsilon>0$, we have
\[
\limsup_{n\to\infty} \frac{1}{n} \log \left( \frac{1}{\#S_n} \#\left\{ g \in S_n : \left| \frac{d(g \cdot o,o)}{n} - \Lambda \right| > \epsilon \right\}\right) < 0.
\]
\end{theorem}

This result is analogous to the main result of the recent work \cite{BMSS} (see also \cite{gouezel.21}) in the setting of iid random walks on Gromov-hyperbolic spaces. As for Theorem \ref{thm:ldt}, we will deduce Theorem \ref{thm.gromov.counting.ld} from a corresponding boundary limit theorem (Theorem \ref{thm.gromov.ray.ld}) for Patterson--Sullivan measures. To prove the latter, we crucially make use of the large deviation estimates that we develop from the work of Benoist--Quint \cite{bq.clt}, for cocycles over random products of group elements in Markovian dependence (these tools also serve us in the Berry--Essen estimate as explained above). We defer the statement of Theorem \ref{thm.gromov.ray.ld} to Section \ref{section:ldt}.

\begin{remark}
Recently Cantrell and Tanaka \cite[Theorem 4.23]{cantrelltanaka} proved a global large deviation principle that implies Theorem \ref{thm.gromov.counting.ld} when $\Gamma$ acts on $H=\Gamma$ by multiplication and $d$ is a left-invariant hyperbolic metric that is quasi-isometric to a word metric. 
\end{remark}

\subsubsection{Central limit theorem with Berry--Esseen type error term} Equipped with a law of large numbers, we now state the first refined limit theorem for counting statistics in representations:

\begin{theorem}\label{thm:clt}
Let $\Gamma$ be a non-elementary Gromov-hyperbolic group, $S$ a generating set  and $\rho:\Gamma \to \GL_d(\R)$ a strongly irreducible and proximal representation. Fix an operator norm $\|\cdot\|$ on $\Mat_d(\R)$. Then, there exists a constant $C>0$  such that for every $t \in \R$ 
\[
\left| \frac{1}{\#S_n} \#\left\{ g \in S_n : \frac{\log\|\rho(g)\| - n\Lambda}{\sqrt{n}} \le t \right\}- \frac{1}{\sqrt{2\pi} \sigma} \int_{-\infty}^t e^{-s^2/2\sigma^2} \ ds\right| \leq  \frac{C \log n}{\sqrt{n}}
\]
where $\Lambda >0$ is as in Theorem \ref{thm:lln} and $\sigma^2 > 0$ are strictly positive constants.
\end{theorem}
In the sequel, whenever a strongly irreducible and proximal representation is fixed, $\sigma^2$ will denote the variance in the above CLT.

The proof of this result requires several ingredients. We first prove a Berry--Esseen central limit theorem for the norm of Markovian random matrix products (Theorem \ref{thm.bougerol.berryess}) based on the analogous result of Bougerol \cite{bougerol.thm.limite} for the norm cocycle. To do this, we use an idea due to Xiao--Grama--Liu from their recent work \cite{XGL}.  The core of the argument is based on large deviation estimates from Benoist--Quint \cite{bq.book} that we develop (Theorem \ref{thm.ld.matrix}) for the Markovian setting by elaborating on other work due to Benoist--Quint \cite{bq.clt} which concerns large deviation estimates for cocycles. Equipped with these results as well as techniques developed by Calegari--Fujiwara \cite{calegari-fujiwara}, we employ a quantitative version of an argument from recent work of Gekhtman--Taylor--Tiozzo \cite{GTTCLT} to carry out our proof.

\subsection{Boundary limit theorems for representations}
As previously mentioned, limit theorems with respect to Patterson--Sullivan measures on the boundary will play a key role in our work: on the one hand, we will prove new results for them (such as Theorem \ref{thm:boundarywiener} on convergence to the Brownian motion), on the other hand, they will be used to prove the counting law of large numbers (Theorem \ref{thm:lln}) and large deviation theorems (Theorems \ref{thm:ldt} and \ref{thm.gromov.counting.ld}). More specifically, we will describe the growth rate of subadditive functions (and the log-norm function for linear representations) along Patterson--Sullivan typical geodesic rays. We achieve this by comparing Markov measures on a Cannon coding with Patterson--Sullivan measures on the boundary of our considered group. Along with techniques from ergodic theory and geometric group theory, this will allow us to translate results concerning Markovian random products to 
asymptotic behaviour along Patterson--Sullivan typical geodesic rays.

In the statements of our boundary limit theorems (and throughout this work), we will consider the boundary $\partial \Gamma$ of $\Gamma$ equipped with generating set $S$ to be the collection of $|\cdot|_S$ geodesic rays up to the usual bounded distance equivalence. For $\xi \in \partial \Gamma$, we use the notation $\xi_n \to \xi$ to indicate that $(\xi_n)_{n \in \N}$ is a geodesic ray in the class of $\xi$ (see \S \ref{sec.hyp}).

\subsubsection{Law of large numbers for the Patterson--Sullivan measure class}\label{subsub.intro.as.lln}

Here is the strong law underlying Theorem \ref{thm:lln}:

\begin{theorem}[Strong law of large numbers] \label{thm:rayconvergence}
Let $\Gamma$ be a non-elementary Gromov-hyperbolic group endowed with a generating set $S$ and $\varphi : \Gamma \to \mathbb{R}$ is subadditive function on $\Gamma$. Let $\nu$ be a probability measure on $\partial \Gamma$ in the Patterson--Sullivan measure class. Then, there exists a constant $\Lambda \geq 0$ such that 
\[
\lim_{n\to\infty} \frac{\varphi(\xi_n)}{n} = \Lambda
\]
for $\nu$-almost every $\xi \in \partial \Gamma$ and every representative $\xi_n \to \xi$.
\end{theorem}

This result generalizes \cite[Theorem A.1]{kaimanovich-kapovich-schupp}. The reason we call it a strong law is that, roughly speaking, the Patterson--Sullivan measures can be viewed as the law of a process for which the uniform counting measures correspond to finite time distributions. This is also the spirit of the deduction of Theorem \ref{thm:lln} from the previous result.

\subsubsection{Large deviations for Patterson--Sullivan measures}
The quantitative analogue of Theorem \ref{thm:rayconvergence} for linear representations is the following result.

\begin{theorem}\label{thm:ldtboundary}
Let $\Gamma$ be a non-elementary Gromov-hyperbolic group, $S$ a generating set and $\rho:\Gamma \to \GL_d(\R)$ a strongly irreducible and proximal representation. Let $\nu$ be a Patterson--Sullivan measure  on  $\partial \Gamma$ for the $S$ word metric and $\Lambda \geq 0$ be the constant from Theorem \ref{thm:rayconvergence}. Then, for any $\epsilon >0$,
\[
\limsup_{n\to\infty} \frac{1}{n} \log \nu\left( \xi \in \partial \Gamma: \text{for all $\xi_m \to \xi$ with $\xi_0=\id$, }\left| \frac{\log\|\rho(\xi_n)\|}{n} - \Lambda \right| > \epsilon \right) <0.
\]
\end{theorem}

Here, when we say the $\nu$ is a Patterson--Sullivan measure, we mean that it is constructed as a weak limit as in \eqref{PSdef} or \eqref{eq.usual.ps} below. We note that any two measures obtained in this way are mutually absolutely continuous and their densities are bounded away from 0 (and infinity).

The proof makes use of Bougerol's results \cite{bougerol.thm.limite} which are translated to the group theoretic setting using techniques due to Calegari--Fujiwara \cite{calegari-fujiwara} and extensions of these techniques due to Cantrell \cite{cantrell.georays}. The scheme of proof, somewhat common to the next Theorem \ref{thm:boundarywiener}, is expounded in \S \ref{subsec.outline} below.

\subsubsection{Convergence to the Wiener measure and law of iterated logarithm}
We now turn to our last result which is an invariance principle and functional law of iterated logarithm with respect to Patterson--Sullivan measures. We first need some notation. Suppose $\Gamma$ is a Gromov-hyperbolic group endowed with a generating set $S$ and that $\rho: \Gamma \to \GL_d(\R)$ is a strongly irreducible, proximal representation.  Let $C([0,1])$ denote the continuous real valued functions on $[0,1]$ equipped  with the Borel $\sigma$-algebra for the topology of uniform convergence. We define a sequence of random variables $(S_n)_{n \in \N}$ on $\partial \Gamma$ taking values in $C([0,1])$ as follows. For each $\xi \in \partial \Gamma$, integer $n \geq 1$ and $t \in [0,1]$, we define $S_n\xi(t)$ to be
\begin{equation}\label{eq.defn.snxi}
\min_{\xi_m \to \xi} \frac{1}{(n\sigma^2)^{1/2}} \left( \log \|\rho(\xi_{\lfloor tn \rfloor})\|-nt\Lambda + (nt-\lfloor nt \rfloor) (\log \|\rho(\xi_{\lfloor tn \rfloor +1})\| - \log\|\rho(\xi_{\lfloor tn \rfloor})\|) \right)
\end{equation}
where $\Lambda$ and $\sigma^2 >0$ are the mean and variance from Theorem \ref{thm:clt}. The reason we consider the minimum over the set of representatives is only practical, it allows to define the random variables $S_n$ on $\partial \Gamma$; replacing $\min$ with $\max$ will not alter the asymptotic behaviour (and hence the results to follow) since any two representatives of a boundary point $\xi \in \partial \Gamma$ stay at bounded distance. 
We denote by $\mathcal{W}$ the Wiener measure on $C([0,1])$. Recall that this is the distribution of the standard Brownian motion $B(\cdot) \in C([0,1])$ which is characterized \cite{karatzas-shreve} by  $B(0) \overset{a.s.}{=}0$, and for every $p \in \N$ and reals $0=t_0<t_1<\ldots<t_p$, the real-valued random variables $B(t_1),B(t_2)-B(t_1),\ldots,B(t_p)-B(t_{p-1})$ are independent and distributed with the Gaussian distribution, respectively, $\mathcal{N}(0,t_{i}-t_{i-1})$.

We will prove the convergence to Wiener measure with respect to the Patterson--Sullivan measure obtained as the weak limit
\begin{equation}\label{PSdef}
\nu = \lim_{n\to\infty} \frac{ \sum_{|g|_S \leq n} \lambda^{-|g|_S} \delta_g}{\sum_{|g|_S \leq n} \lambda^{-|g|_S}}.  
\end{equation}
Here $\lambda \in (1,\infty)$ denotes the exponential growth rate of the cardinality of $S_n = \{g \in \Gamma: |g|_S = n\}$. The fact that the limit \eqref{PSdef} exists will be explained in \S \ref{subsec.hyp.gp}. We prove the following result.

\begin{theorem} \label{thm:boundarywiener}
Let $\Gamma$ be a non-elementary Gromov-hyperbolic group, $S$ a generating set  and $\rho:\Gamma \to \GL_d(\R)$ a strongly irreducible and proximal representation. Let $\nu$ be the Patterson--Sullivan measure defined in \eqref{PSdef}. Then,
\begin{enumerate}[\hspace{0.78cm}1.]
    \item  under $\nu$, the sequence $(S_{n})_{n \in \N}$ of $C([0,1])$-valued random variables converges in distribution to $\mathcal{W}$; and,
\item for $\nu$-almost every $\xi \in \partial \Gamma$, the set of limit points of the sequence $\left(\frac{S_{n}\xi}{2\log \log n}\right)_{n \in \N}$ of elements of $C([0,1])$ is equal to the following compact subset of $C([0,1])$:
$$
\left\{f \in C([0,1]): f \; \text{is absolutely continuous}, f(0)=0, \int_{0}^1f'(t)^2dt \leq 1\right\}.
$$
\end{enumerate}
\end{theorem}

Two immediate corollaries of this result are the classical central limit theorem and law of iterated logarithm with respect to the Patterson--Sullivan measure $\nu$ (Corollary \ref{corol.boundary.CLT.LIL}).

\subsection{Outline of the arguments}\label{subsec.outline}
We briefly outline the over arching argument used to prove Theorem \ref{thm:boundarywiener} which is also valid to some extent for the proof of Theorem \ref{thm:ldtboundary} (see below for other limit theorems).
\begin{enumerate}
    \item We begin by introducing multiple Markov chains based on the Cannon coding for our group $\Gamma$ and generating set $S$.
    \item We formulate and, in some cases, further develop Bougerol's results  \cite{bougerol.survey,bougerol.thm.limite} for random matrix products in Markovian dependence.
    \item Using work of Goldsheid--Margulis \cite{goldsheid-margulis} and Gou\"{e}zel--Math\'{e}us--Maucourant \cite{gmm}, we show that our assumptions on the representations (strongly irreducible and proximal) allow us to apply the results of (ii) to the Markovian products introduced in (i).
    \item We then use an argument of Calegari--Fujiwara \cite{calegari-fujiwara} to show that the means $\Lambda$ and variances $\sigma^2$ coming from the limit theorems obtained in (iii) of different Markovian products introduced in (i) coincide (and equal the limiting average $\Lambda$ obtained from Theorem \ref{thm:lln}).
    \item We compare the stationary distributions on the Markov chains to a Patterson--Sullivan measures on the boundary of the group. With some additional work, this allows us to prove that along geodesic rays in the Cayley graph of $(\Gamma, S)$, the log-norm function satisfies the corresponding limit theorem with respect to Patterson--Sullivan measure on the Gromov boundary $\partial \Gamma$.
\end{enumerate}

For counting limit theorems (Theorems \ref{thm:lln}, \ref{thm:ldt}, \ref{thm.gromov.counting.ld}), we ultimately use the boundary results obtained in (v) above together with regularity estimates on Patterson--Sullivan measures to get counting limit theorems.

As discussed above, our tactic for proving counting CLT with error term (Theorem \ref{thm:clt}) is a little bit different  (without passing by a boundary limit theorem to optimize the Berry--Esseen error term), see a more detailed description in \S \ref{section:clt}.


\subsection{Previous works}
Here we briefly comment on some previous related works in the literature.

In \cite{kaimanovich-kapovich-schupp} (see also \cite{kapovich-schupp-shpilrain}) Kaimanovich--Kapovich--Schupp study generic asymptotic behaviour (in the sense of Theorems \ref{thm:lln} and \ref{thm:rayconvergence}) of elements in countable groups $\Gamma$ where genericity is understood with respect to the uniform counting measure in a free group $\mathrm{F}$ (with a free generating set) when $\Gamma$ is seen as a quotient of $\mathrm{F}$. This point of view lies in between the two extremes, namely the symbolic counting point of view of probability theory (i.e.~ group invariant random walks on groups) and our deterministic counting viewpoint. In vague terms, our approach agree with that of Kapovich--Kaimanovich--Schupp when the underlying group is a free group and generating set is free\footnote{accordingly, our Theorem \ref{thm:rayconvergence} generalizes \cite[Theorem A.1]{kaimanovich-kapovich-schupp}
.} and these two together agree with the probabilistic (iid random walks) approach when the underlying algebraic object is a free semigroup.


Coming back to counting asymptotics, there has recently been significant interest in counting limit theorems on hyperbolic groups, see for example \cite{calegari-fujiwara,cantrell.TAMS,cantrelltanaka,choi, GTTCLT,countinglox,horsham.thesis, horshamsharp,pollicott.sharp.TAMS, yang}. In some of these works, techniques from thermodynamic formalism (tracing back to \cite{bowen,sinai}, see also \cite{lalley}) are used. Although these techniques are powerful and allow for stronger results to be obtained, they usually require strong assumptions on the studied potentials. In this work, our assumptions are too weak to allow us to apply techniques from thermodynamic formalism. In others, including ours, ideas from Markov chain or random walk theory (in a sense initiated in this context by \cite{calegari-fujiwara}) are used instead of thermodynamic techniques. For example, in \cite{GTTCLT} the authors deduce a (qualitative) counting CLT for displacement functions on Gromov-hyperbolic spaces\footnote{See \S \ref{sec.last} for a consequence of our large deviation results for the displacement function on symmetric spaces of non-compact type.} from a CLT for centerable cocycles. Using ideas of Benoist--Quint \cite{bq.clt} relying on solving a cohomological equation, it might be possible to do so in our setting as well. However, our approach relying instead on results of Bougerol coming from analytic perturbation theory, yields more quantitative results (such as the Berry--Esseen bounds). It was indeed one of our goals to get quantitative results as it seems particularly in line with the spirit of counting problems. 


Finally, we also mention that in the upcoming work \cite{CCDS} with I. Cipriano and R. Dougall, we obtain more precise limit laws for the both the spectral radius and norm potentials under the assumption that our representation is Anosov (or dominated). In this setting we will be able to exploit ideas from thermodynamic formalism.


\subsubsection{Further directions}
In this work, we restricted our attention to Gromov-hyperbolic groups. Our approach relies on the existence of a nice combinatorial structure (Cannon coding) and stochastic results on this structure. Various generalisations of the notion of the Canon coding have been studied, both from a combinatorial perspective (\cite{cannon.etal}\footnote{See already in \cite{kaimanovich-kapovich-schupp} some considerations towards this direction, but the results therein does not readily yield counting estimates.}) and geometric perspective (\cite{GTTCLT}). It seems possible to find examples of non-hyperbolic groups equipped with certain specific (in some cases abstract) generating sets (see \cite{GTTCLT}) for which our counting results will hold. In some settings (e.g.~ relatively hyperbolic groups) it may also be possible to find analogues of our boundary limit theorems. It would be interesting to characterize the widest class of groups (equipped with any generating set) for which our results hold. 

\subsection*{Acknowledgements}

The authors thank Emmanuel Breuillard and Ryokichi Tanaka for useful discussions and suggestions.

\section{Hyperbolic groups and automatic structures}\label{sec.hyp}

\subsection{Gromov-hyperbolic groups}\label{subsec.hyp.gp}
Let $\Gamma$ be a finitely generated group and $S$ a (finite, symmetric) generating set. As in the introduction, for $g,h \in \Gamma$, $|g|_S$ denotes the word length of $g$ with respect to $S$ and $d_S(g,h)=|g^{-1}h|_S$ defines a left-invariant metric on $\Gamma$. The Gromov product of $g,h \in \Gamma$ is defined as $\langle g, h \rangle = \frac{1}{2} (|g|_S + |h|_S - |g^{-1}h|_S)$. The group $\Gamma$ is said to be Gromov-hyperbolic if $(\Gamma,d_S)$ is a Gromov-hyperbolic metric space. We recall that a metric space $(H,d)$ is said to be Gromov-hyperbolic if there exists $\Delta>0$ such that for every $o,x,y,z \in H$, $$(x,z)_o \geq \min\{(x,y)_o,(y,z)_o\}-\Delta,$$ where $(x,y)_o:=\frac{1}{2}(d(x,o)+d(y,o)-d(x,y))$
denotes the Gromov-product.
Although the constant $\Delta>0$ may depend on the generating set $S$, Gromov-hyperbolicity of $\Gamma$ does not depend on $S$.

Fix a Gromov-hyperbolic group $\Gamma$ and a generating set $S$. A geodesic ray is a sequence of elements $\xi_n \in \Gamma$ such that $|\xi_n^{-1} \xi_m|_S = m-n$ for each $m,n \in \N$ with $m \ge n$. The Gromov boundary $\partial \Gamma$ is the set of equivalence classes of geodesic rays where two rays $\xi$ and $\xi'$ are equivalent if $\sup_{n\ge 1} |\xi_n^{-1} \xi_n'|_S$ is finite. Since the action of $\Gamma$ on itself by left-multiplication is by isometries (with respect to $d_S$), the natural action on the set of geodesic rays factors through this equivalence relation and defines an action of $\Gamma$ on its Gromov boundary $\partial \Gamma$. It is well-known that the set $\Gamma \cup \partial \Gamma$ carries a compact metrizable topology extending the (discrete) topology of $\Gamma$ such that $\Gamma$ is open and dense in $\Gamma \cup \partial \Gamma$ and the $\Gamma$-action is by homeomorphisms.

We can extend the Gromov product to $\partial \Gamma \times \Gamma$ by setting
\[
\langle \xi, g \rangle = \sup\{ \liminf_{n\to\infty} \langle \xi'_n, g \rangle : \xi'_n \to\xi \}
\]
where the supremum is taken over geodesic rays $\xi'_n$ with $\xi'_n \to \xi$; the latter notation denotes the fact that $\xi'_n$ is a geodesic ray in the equivalence class corresponding to $\xi$. Using this extended Gromov product, for $R>0$ and $g \in \Gamma$, we define the $R$-shadow based at $g$ to be the following subset of $\partial \Gamma$:
\[
O(g,R) = \{ \xi \in \partial \Gamma: \langle \xi, g \rangle > |g|_S - R\} 
\]

To prove our law of large numbers (Theorem \ref{thm:lln}), a key ingredient will be the study of the growth rate of subadditive functions along typical geodesic rays in $\Gamma$. In particular, we will be interested in the behaviour of our functions along Patterson--Sullivan typical geodesic rays. Recall that a Patterson--Sullivan measure for the length function $|\cdot|_S$ on $\Gamma$ is obtained as a weak limit of the following sequence of measures on the compact $\Gamma \cup \partial \Gamma$
\begin{equation}\label{eq.usual.ps}
    \frac{\sum_{g\in \Gamma} \lambda^{-s |g|_S} \delta_{g}}{\sum_{g\in \Gamma} \lambda^{-s |g|_S}}
\end{equation}
as $s \searrow 1$ where $\lambda>1$ is the exponential growth rate of the cardinality of $S_n = \{g \in \Gamma: |g|_S = n\}.$ Alternatively, we can obtain a Patterson--Sullivan measure as the weak limit of the sequence
\begin{equation}\label{eq.def.ps}
 \frac{\sum_{|g|_S \le n} \lambda^{-|g|_S
}\delta_{g}}{\sum_{|g|_S \le n} \lambda^{-|g|_S
}}
\end{equation}
as $n\to\infty$ (See Section 4 of \cite{calegari-fujiwara}). Any measure $\nu$ constructed using either of the above two methods yields a Radon measure supported on $\partial \Gamma$ such that the $\Gamma$ action preserves its measure class and is ergodic. In this setting ergodic means that $\Gamma$-invariant subsets of $\partial \Gamma$ have either full or null $\nu$-measure. An important property exhibited by Patterson--Sullivan measures is the so-called quasiconformal property: for each $R > 0$ sufficiently large, there exists a constant $C>1$ depending only on $R$ and the hyperbolicity constant of $\Gamma$ such that
\begin{equation} \label{eq:shadows}
C^{-1} \lambda^{-|g|_S} \le \nu(O(g,R)) \le C \lambda^{-|g|_S}
\end{equation}
for all $g \in \Gamma$.

Before we move on to discuss the strongly Markov structure of hyperbolic groups, we record a basic property of subadditive functions which we will use implicitly throughout our work. Recall that the left and right word metrics associated to a generating set $S$ on $\Gamma$ are
\[
d_S(g,h) = |g^{-1}h|_S \ \ \text{ and } \ \ d_R(g,h) := |gh^{-1}|_S.
\]
We will repeatedly (and sometimes implicitly) use the fact that subadditive functions $\varphi: \Gamma \to \mathbb{R}$ are Lipschitz in these metrics metrics as noted in the next result.
\begin{lemma}\label{lem:Lipschitz}
Fix a finite symmetric generating set $S$ for $\Gamma$ and let $\varphi: \Gamma \to \mathbb{R}$ be subadditive. Then, $\varphi$ is Lipschitz in the left and right word metrics.
\end{lemma}
\begin{proof}
It suffices to show that there exist a constant $C>0$ such that
\[
|\varphi(g) - \varphi(sg)| \le C \ \ \text{ and } \ \  |\varphi(g) - \varphi(gs) | \le C
\]
for all $g \in \Gamma$ and $s \in S$. It follows easily from the definition of subadditivity that $C = \max_{s\in S} |\varphi(s)|$ works.
\end{proof}

\subsection{Markov structure of Gromov-hyperbolic groups}
It was realized by Cannon \cite{cannon} that certain Kleinian groups enjoy a strong coding property: the elements of metric spheres in the Cayley graph can be bijectively represented by admissible words of corresponding length in a finite automaton (which we will refer to as strongly Markov property, see Definition \ref{def.strong.markov}). It was indicated by Gromov \cite{gromov} and proved by Coornaert--Delzant--Papadopoulos \cite{cdp} and Ghys--de la Harpe \cite{gh} that general Gromov-hyperbolic groups with arbitrary finite generating sets enjoy the strongly Markov property. We now discuss this crucial property which will allow us to associate a subshift of finite type to a Gromov-hyperbolic group equipped with a generating set.
\begin{definition}\label{def.strong.markov}
A group $\Gamma$ is strongly Markov if given any generating set $S$ for $\Gamma$, there exists a finite directed graph $\mathcal{G}$ with vertex set $V$ and directed edge set $E \subset V \times V$ that exhibit the following properties:
\begin{enumerate}
\item $V$ contains a vertex $\ast$ such that $(x,\ast)$ does not belong to $E$ for any $x \in V$,
\item there exists a labelling $\lambda :E \to S$ such that the map sending a path (starting at $\ast$) with concurrent edges $(\ast,x_1),(x_1,x_2),\ldots , (x_{n-1},x_n)$ to the group element $\lambda(\ast,x_1) \lambda(x_1,x_2) \ldots \lambda(x_{n-1},x_n),$ is a bijection,
\item the above bijection preserves word length; if $|g|=n$, then the finite path corresponding to $g$ has length $n$.
\end{enumerate}
\end{definition}
\noindent To simplify notation later on, we augment the above strongly Markov structure by introducing an additional vertex labelled $0$ to $V$. We also add directed edges from every vertex $x \in V $ to $0$ and define $\lambda(x,0) = \id$ (the identity in $\Gamma$) for every $x \in V$. We will assume that every strongly Markov structure has been augmented in this way and will abuse notation by labelling the augmented structure, its edge and vertex set by $\mathcal{G}$, $V$ and $E$ respectively. This directed graph $\mathcal{G}$ allows us to introduce a subshift of finite type as we now explain.

\subsubsection{Shift spaces}\label{subsub.shifts}
Let $A$ be a $k \times k$ matrix consisting of zeros and ones. We use the notation $A_{i,j}$ to denote the $(i,j)$th entry of $A$. The subshift of finite type associated to $A$ is the space
$$\Sigma_A = \{(x_n)_{n=0}^{\infty} : x_n \in \{1,2,...,k\}, A_{x_n,x_{n+1}}=1, n \in \mathbb{Z}_{\ge 0}\}.$$
Given $x$ in $\Sigma_A$ we write $x_n$ for the $n$th coordinate of $x$. The shift map $\sigma: \Sigma_A \rightarrow \Sigma_A$ sends $x$ to $y= \sigma(x)$ where $y_n=x_{n+1}$ for all $n \in \mathbb{Z}_{\ge 0}$. 

The mixing properties of $(\Sigma_A, \sigma)$ are determined by the structure of the matrix $A$.
\begin{definition}
We say that a $k \times k$ zero-one matrix $A$ is irreducible if for every $(i,j)$ ($i,j \in \{1,2,...,k\}$), there exists $n\in \mathbb{N}$ such that $(A^n)_{i,j} >0$. We say that $A$ is aperiodic if there exists $n \in \mathbb{N}$ such that $(A^n)_{i,j} > 0$ for all $i,j$. 
\end{definition}

It is a standard fact that if $A$ is irreducible then $(\Sigma_A, \sigma)$ is (topologically) transitive, and if $A$ is aperiodic then $(\Sigma_A, \sigma)$ is mixing. Further, if $A$ is irreducible then there exists a natural number $p\ge1$ known as the period of $A$ such that the alphabet $\{1,\ldots,k\}$ of $A$ is partitioned into $p$ disjoint subsets $A_i$ and $\Sigma_A$ has a cyclic decomposition
\[
\Sigma_A = \bigsqcup_{k=0}^{i-1} \Sigma_{A}(i),
\]
where $\Sigma_A(i)$ is the subset of $\Sigma_A$ starting with elements from $A_i$. The shift map $\sigma:\Sigma_A \to \Sigma_A$ sends $\Sigma_{A}(i)$ to $\Sigma_{A}(i+1)$ where $i, i+1$ are taken modulo $p$ and for each $i=0, \ldots, p-1$ the subshifts $(\Sigma_{A}(i), \sigma^p)$ are mixing.

\subsubsection{Shift space associated to a Markov structure}\label{subsub.shift.space.ofmarkov}
Suppose now that $\mathcal{G}$ is a strongly Markov structure associated to a Gromov-hyperbolic group $\Gamma$ and generating set $S$. We can describe $\mathcal{G}$ using a zero-one matrix $A$: we label the vertices of $\mathcal{G}$ by $0,\ast,1,2,\ldots, k \in \N$ (where $0$ and $\ast$ are distinguished vertices described above) and set $A_{i,j} = 1$ if and only if there is a directed edge from vertex $i$ to vertex $j$ and otherwise we set $A_{i,j}=0$. We can then construct a subshift of finite type $\Sigma_A$ as described in the previous paragraph. We will write $A'$ for the matrix obtained from $A$ by discarding the row and column corresponding to the vertex $0$ and $A''$ for the one where we also discard the vertex $\ast$. 
We will write $\Sigma_A^0$ for the collection of sequences in $\Sigma_A$ that contains an occurrence of $0$. Note that, by construction, if a sequence $(x_n)_{n=0}^\infty$ has $x_k = 0$ for some $k$ then $x_l = 0$ for all $l \ge k$. We will use the notation $(x_0, \ldots, x_{n-1}, \dot{0})$ to express sequences that start with the vertices $x_0, x_1, \ldots, x_{n-1}$ and then end with infinitely many zeros. Note that $\Sigma_A^0$ is dense in $\Sigma_A$ when $\Sigma_A$ is endowed with the restriction of the product topology on $V^\N$. We define a map $i : \Gamma \to \Sigma_A^0$ by $
i(g) = (\ast, x_1, \ldots, x_{n}, \dot{0})
$
where $(\ast, x_1, \ldots, x_{n}, \dot{0})$ is the unique sequence belonging to $\Sigma_A$ such that $g= \lambda(\ast,x_1) \lambda(x_1,x_2) \cdots \lambda(x_{n-1},x_{n})$ (and $|g|_S = n$).

For certain hyperbolic groups and generating sets (i.e.~ for a free group equipped with a free generating set) one can find a strongly Markov structure $\mathcal{G}$ such that the corresponding matrix $A''$ is aperiodic. However, for general hyperbolic groups and generating sets it is not known whether it is always possible to find a Markov structure such that the matrix $A''$ is aperiodic or even irreducible.
After relabelling (i.e.~ permuting) the columns and rows of $A''$, we may assume that $A''$ has the form 
$$A'' = \begin{pmatrix} 
B_{1,1} & 0 & \dots & 0  \\
B_{2,1} & B_{2,2} & \dots & 0\\
\vdots & \vdots & \ddots & \vdots\\
B_{m',1} & B_{m',2} & \dots & B_{m',m'}
\end{pmatrix} ,$$

\noindent where the matrices $B_{i,i}$ are irreducible. The matrices $B_{i,i}$ are known as the irreducible components of $A''$ and the corresponding vertex sets in $\mathcal{G}$ are the irreducible components of $\mathcal{G}$.
By property $(3)$ in Definition \ref{def.strong.markov}, it is easy to see that the spectral radius of each $B_{i}$ is bounded above by the growth rate $\lambda$ of the group $\Gamma$. Moreover, by the same token, there must be at least one component that has $\lambda$ as an eigenvalue. We call an irreducible component \textit{maximal} if the corresponding matrix $B_i$ has spectral radius $\lambda$. We relabel the irreducible components so that maximal components correspond to $B_{i,i}$ for $i=1,\ldots,m$, which we will denote as $B_i$. An important property of $\mathcal{G}$ is that the maximal components of $\mathcal{G}$ are disjoint. That is, there does not exist a path in $\mathcal{G}$ from one maximal component to another.
This is a consequence of a result of Coornaert \cite{coor} which asserts that for a non-elementary hyperbolic group (and any generating set $S$) the growth of $\#S_n$ is purely exponential, i.e.~ for $\Gamma$, $S$ as above there exist $C>1$ and $\lambda >1$ such that for all $n \in \mathbb{Z}_{\ge 0}$.
\begin{equation}\label{eq.pure.growth}
C^{-1} \lambda^n \le \#S_n \le C \lambda^n.
\end{equation}


\section{Markovian random matrix products}\label{sec.markov}

This section is mostly independent of the rest of the paper and it is devoted to limit theorems for norms of Markovian random matrix products (which will be important ingredients of our counting results): simplicity of Lyapunov exponents, invariance principle, functional law of iterated logarithm (LIL), large deviation estimates and Berry--Esseen bounds.\\[3pt]
\textbullet ${}$ \textit{Simplicity of Lyapunov exponents} (\S \ref{subsec.simplicity}): We will briefly recall the work of Bougerol \cite{bougerol.comparaison} (see also Virtser \cite{virtser} and Royer \cite{royer}) generalizing previous work of Guivarc'h \cite{guivarch.markov} and ultimately the key result of Furstenberg \cite{furstenberg.non.commuting} on positivity of the top Lyapunov exponent.\\[3pt]
\textbullet ${}$ \textit{Invariance principle and LIL} (\S \ref{subsec.Donsker.LIL.markov}): we will recall the work of Bougerol \cite{bougerol.survey,bougerol.thm.limite}  generalizing corresponding results in the iid setting due to Le Page \cite{lepage}.\\[3pt]
\textbullet ${}$ \textit{Large deviation estimates} (\S \ref{subsec.ld.markov.matrix} and \S \ref{subsec.ld.isometry}): We will prove large deviation estimates both for Markovian matrix products and Markovian random walks on Gromov-hyperbolic spaces. The former result is contained in Bougerol's work \cite{bougerol.thm.limite}, however we will give a different proof using an approach of Benoist--Quint \cite{bq.clt}. We have two reasons for giving a different proof: the first one is that the tools developed for this proof will be used in the large deviation ingredient of the Berry--Esseen estimate (which we could not directly obtain from Bougerol's work), the second one is that this approach is more general and gives also the corresponding results for Markovian random walks on Gromov-hyperbolic spaces. The latter will be used later to give another setting providing a positive answer to a question of Kaimanovich--Kapovich--Schupp \cite{kaimanovich-kapovich-schupp} (see \S \ref{sec.last}).\\[3pt]
\textbullet ${}$ \textit{Berry--Esseen estimates} (\S \ref{subsec.berry.esseen.markov}):
We will prove Berry--Esseen estimates for matrix norms $\log \|M_n\|$ using the corresponding estimates of Bougerol \cite{bougerol.thm.limite} for $\log \|M_nv\|$ by adapting the approach of Xiao--Grama--Liu \cite{XGL} and using our large deviation estimates.\\ 

Before proceeding, we mention that we will restrict ourselves to Markovian random matrix products over countable state Markov chains. The general state space cases are typically treated using similar ideas but with heavier machinery (see Guivarc'h \cite{guivarch.markov} and Bougerol \cite{bougerol.comparaison,bougerol.thm.limite} for nice expositions). Although we will only need to apply these results to the finite state space case, we include the countable setting since it does not introduce any serious additional difficulties and since we believe that this generality could be useful for works in contexts close to ours (e.g.~ for extensions of our counting results).

\subsection{Basic definitions} We start by setting our notation, brief recalls and defining Markovian random walks on groups and Markovian random matrix products.

\subsubsection{Reminders on the theory Markov chains on countable state spaces}\label{subsub.markov.chains}

Let $E$ be a countable set and $P$ a probability transition kernel on $E$. 
By (standard) abuse of notation, let $P$ also denote the associated Markov operator and its dual: given a real-valued function $f$ on $E$, $Pf(x)=\int f(y)P(x,dy)$ whenever the integral makes sense.
We shall write $\mu P$ for the action of $P$ on probability measures on $E$. 
Given a probability measure $\nu$ on $E$, the distribution of the associated Markov chain on $E^\N$ is denoted by $\mathbb{P}_\nu$. We will usually denote the sequence of coordinate functions by $z_n$ for $n=0,1,\ldots$. We say that the probability kernel $P$ is irreducible if for every $x,y \in E$, there exists $n \in \N$ such that $P^n(x,y)>0$. For an irreducible kernel $P$, its period is defined to be $\gcd \{n \in \N  :  P^n(x,x)>0\}$ for some (equivalently all) $x \in E$. An irreducible kernel is said to be aperiodic if its period is one. In general, if the period is $p \in \N$, there exists a partition $E_1,\ldots, E_p$ of the the state space $E$ such that for every $x \in E_i$, $P(x,E_{i+1})=1$ $(i \mod p)$. If $P$ is irreducible and has period $p$, the $P^p$ defines an irreducible aperiodic kernel on $E_i$ for every $i=1,\ldots,p$. 

A probability measure $\pi$ on $E$ is called $P$-stationary if it satisfies $\pi P=\pi$. 
An irreducible transition kernel $P$ is said to be \textit{positively recurrent} if it admits a stationary probability measure $\pi$, in which case this probability measure is unique. If $P$ has period $p \in \N$, we have $\pi=\frac{1}{p}\sum_{i=1}^p \pi_{|E_i}$ and $\pi_{|E_i} P=\pi_{|E_{i+1}}$ $(i \mod p)$, where $p \pi_{|E_i}$ is the unique stationary probability measure of the irreducible aperiodic kernel $P^p$ on $E_i$. 
The Markov chain $(z_n)$ is said to be \textit{uniformly geometrically ergodic} if there exist a $P$-stationary probability measure $\pi$ on $E$ and constants $C>0$ and $\rho \in (0,1)$ such that for every $z \in E$ and $n \in \N$, we have $\|P^n(z,\cdot)-\pi(\cdot)\|_{TV} \leq C \rho^n$, where $\|\cdot\|_{TV}$ denotes the total variation or equivalently the $\ell^1$-norm (note that this condition forces $P$ to be irreducible and aperiodic). This is automatically satisfied if $P$ satisfies the Doeblin condition (i.e.~ there exist $n \in \N$, $z \in E$ and $\delta>0$ such that for every $y \in E$, $P^n(y,z)\geq \delta$) and in particular if $E$ is finite and $P$ is aperiodic.

\subsubsection{Markovian matrix products associated to a Markov chain}

Let $E$ be a countable state space, $P$ a transition kernel on $E$ and $\Gamma$ a group. Given a map $X:E \to \Gamma$, the associated \textit{Markovian random walk} on $\Gamma$ is defined as the process $M_n=X(z_{n})\cdots X(z_1)$. We will often write $X_n=X(z_n)$ for the $n^{th}$-step of the associated Markovian random walk. When $\Gamma \leq \GL_d(\R)$, we will mostly refer to it as a Markovian random matrix product. In this section, we will always require that the transition kernel $P$ be irreducible and positive recurrent. For Markovian random matrix products, we will always ask that the map $X$ has the following integrability condition with respect to the stationary probability measure $\pi$ of $P$: $\E_{\pi}[\log N(X_1)]<\infty$, where for a matrix $g \in \GL_d(\R)$ and a choice of norm $\|\cdot\|$ on $\R^d$, we write $N(g)=\max \{\log \|g\|, \log \|g^{-1}\|\}$. Note that the integrability condition does not depend on the choice of norm.

We will briefly refer to all this data as a Markovian random walk or Markovian (random matrix) product and denote it by $(M_n)$.

\subsection{Simplicity of Lyapunov exponents}\label{subsec.simplicity}

Given a Markovian product $(M_n)$, it follows from the Furstenberg--Kesten theorem (or subadditive ergodic theorem) that for every $k=1,\ldots,d$ there exist constants $\lambda_1 \geq \ldots \geq \lambda_d$ such that  $\P_\pi$-a.s.~ we have
\begin{equation}\label{eq.kingman.matrix}
    \frac{1}{n} \log \|\wedge^k M_n\| \underset{n \to \infty}{\longrightarrow} \sum_{i=1}^k \lambda_i.
\end{equation}
These numbers are called the \textit{Lyapunov exponents} of the Markovian product $(M_n)$. Clearly, they do not depend on the choice of the norm on $\Mat(\wedge^k \R^d)$.

We will now see a result characterizing certain situations where these exponents are distinct from each other. We first need some definitions.

We say that a subset $T$ of $\GL_d(\R)$ is \textit{$r$-proximal} with $r \in \{1,\ldots,d-1\}$ if there exists a sequence $(g_n)$ of elements in $T$ such that $\frac{g_n}{\|g_n\|}$ converges in $\Mat_d(\R)$ to a linear transformation of rank at most $r$. Sometimes, we shall simply write proximal instead of 1-proximal.

Given a Markovian product $(M_n)$, for $x_0 \in E$, let $T_{x_0}:= \{M \in \GL_d(\R) : \exists n \in \N, x_1,\ldots,x_n \in E \; \text{such that} \; P(x_i,x_{i+1})>0 \; \text{and} \; M=X(x_n)\ldots X(x_1) \},$
where the indices $i$ in the condition $P(x_i,x_{i+1})>0$ ranges from $0$ to $n-1$. Moreover, for $x \in E$, let
$T_{x_0}(x):=\{M \in \GL_d(\R) : \exists n \in \N, x_1,\ldots,x_n=x \in E \; \text{such that} \; P(x_i,x_{i+1})>0 \; \text{and} \; M=X(x_n)\ldots X(x_1) \}$. Note that for every $x \in E$, $T_x(x)$ is a semigroup in $\GL_d(\R)$ contained in the set $T_x$.

The Markovian product $(M_n)$ is said to be \textit{$r$-contracting} if there exists $x \in E$ such that $T_x$ is $r$-proximal.
We say that a Markovian product $(M_n)$ is \textit{irreducible} if for any $r \in \{1,\ldots,d-1\}$ there does not exist a map $V:E \to \Gr_r(\R^d)$ (where $\Gr_r$ denotes the Grassmanian of $r$-dimensional subspaces) such that for every $x_0 \in E$ and $n \in \N$, $\P_{x_0}$-a.s. $M_n V(x_0)=V(x_n)$.
Finally, we say that a Markovian product $(M_n)$ is \textit{strongly irreducible} if for any $r \in \{1,\ldots,d-1\}$ there does not exist a finite number of maps $V_i:E \to \Gr_r(\R^d)$ (say, $i=1,\ldots,t$) such that for $x \in E$, denoting $W(x)=\cup_{i=1}^t V_i(x)$, we have, for every $x_0 \in E$ and $n \in \N$, $P_{x_0}$-a.s. $M_n W(x_0)=W(x_n)$.



The following particular case of a result of Bougerol \cite[Th\'{e}or\`{e}me 1.6]{bougerol.comparaison} gives a characterization of the so-called simplicity of Lyapunov spectrum in our setting.  
\begin{theorem}[Simplicity of Lyapunov spectrum, Guivarc'h \cite{guivarch.markov} and Bougerol \cite{bougerol.comparaison}]\label{thm.bougerol.simplicity}
Let $(M_n)$ be a strongly irreducible Markovian product. Then, for $r=1,\ldots,d-1$, we have
$\lambda_1>\lambda_{r+1}$ if and only if $M_n$ is $r$-contracting. 
\end{theorem}

This result will be a crucial ingredient for the upcoming limit theorems (Theorems \ref{thm.bougerol.wiener.lil} and \ref{thm.bougerol.berryess}). It will also be used in the proof of positivity of $\Lambda$ in Proposition \ref{prop.positivity} (however, this can alternatively be deduced from Proposition \ref{prop.positivity.in.section} relying directly on the earlier positivity result of Furstenberg).

\subsection{Invariance principle and functional law of iterated logarithm}\label{subsec.Donsker.LIL.markov}
Here we briefly discuss two limit theorems due to Bougerol \cite{bougerol.thm.limite}: the first one is an analogue of the classical invariance principle due to Donsker which is a generalization of the central limit theorem. The second one is the analogue of Strassen's functional law of iterated logarithm (LIL) generalizing the Hartman--Wintner LIL. For these results (and others to follow), we will need further assumptions on the Markovian product $(M_n)$ that we now discuss.

Following Bougerol \cite[\S 3]{bougerol.thm.limite} (see also \cite{bougerol.survey}), we shall say that a Markovian random matrix product $(M_n)$ satisfies\\[5pt]
\textbf{Condition $(A_1)$:} If the Markov chain $(z_n)$ is uniformly geometrically ergodic; and,\\[4pt]
\textbf{Condition $(A_2)$:} If there exist positive constants $a,B$ such that $\mathbb{E}_x[e^{aN(M_1)}] \leq B$ for every $x \in E$.

Notice that both are automatically satisfied if $E$ is finite and $(z_n)$ is aperiodic.

\begin{theorem}[Convergence to the Wiener process and LIL, Bougerol \cite{bougerol.thm.limite}]\label{thm.bougerol.wiener.lil}
Let $(M_n)$ be a $1$-contracting irreducible Markovian random matrix product satisfying condition $(A_1)$ and $(A_2)$. For $\sigma_0>0$ and $t \in [0,1]$ and $n \in \N$, let $S_{n}(t)$ denote $C([0,1])$-valued random variable defined by 
\begin{equation}\label{eq.defn.sn}
S_{n}(t)=\frac{1}{(n\sigma_0^2)^{1/2}} \left( \log \|M_{\lfloor tn \rfloor}\|-nt\lambda_1 + (nt-\lfloor nt \rfloor) (\log \|M_{\lfloor tn \rfloor +1}\| - \log\|M_{\lfloor tn \rfloor}\|) \right)    
\end{equation}
Then, there exists a constant $\sigma_0=\sigma>0$ such that for every $x \in E$ 
\begin{enumerate}[\hspace{0.78cm}1.]
\item under $\P_x$, the sequence $(S_{n})_{n \in \N}$ of $C([0,1])$-valued random variables converges in distribution to $\mathcal{W}$; and,
\item for $\P_x$-a.e.~ $\omega$, the set of limit points of the sequence $\left(\frac{(S_{n}(t))(\omega)}{2\log \log n}\right)_{n \in \N}$ of elements of $C([0,1])$ is equal to the following compact subset of $C([0,1])$:
$$
\left\{f \in C([0,1]): f \; \text{is absolutely continuous}, f(0)=0, \int_{0}^1f'(t)^2dt \leq 1\right\}.
$$
\end{enumerate}
\end{theorem}

We indicate how to deduce this version from Bougerol's original statement which concerns $\log\|M_n v\|$ for a non-zero vector $v \in \R^d$.

\begin{proof}
In view of Theorem \ref{thm.bougerol.simplicity}, our assumptions on the Markovian product $(M_n)$, namely, $1$-contracting and irreducible, imply that the condition \cite[(A3)]{bougerol.thm.limite} is satisfied (see \cite[Definition 2.7]{bougerol.thm.limite} and thereafter). Therefore \cite[Th\'{e}or\`{e}me 4.5]{bougerol.thm.limite} implies both statements when $\log \|M_n\|$ is replaced by $\log \|M_n v\|$ for some non-zero $v \in R^d$. Note that positivity of the variance follows from \cite[Proposition 4.9]{bougerol.thm.limite}. The statements for  $\log \|M_n\|$ then follow from \cite[Proposition 2.8]{bougerol.thm.limite}: the second conclusion directly follows and the first one follows by appealing to a standard fact, see e.g.~ \cite[Problem 4.16]{karatzas-shreve}.
\end{proof}


\subsection{Large deviation estimates for Markovian random matrix products}\label{subsec.ld.markov.matrix}
In this part, we prove the following theorem by using some ideas that we adapt from the work of Benoist--Quint \cite{bq.clt}. The developed tools will also serve as an ingredient in the the proof of Berry--Esseen estimates.

\begin{theorem}[Markovian random matrix products]\label{thm.ld.matrix}
Let $(M_n)$ be a strongly irreducible and $1$-contracting random matrix product satisfying $(A_1)$ and $(A_2)$. Let $\|\cdot\|$ be a fixed norm on $\R^d$. Then, for every $\epsilon>0$, there exist $\alpha>0$ and $C>0$ such that for every $x \in E$ and $n \in \N$, and non-zero $v \in V$, we have
$$\mathbb{P}_x( |\log \|M_n v\|-n\lambda_1| \geq n\epsilon) \leq C e^{-\alpha n} \qquad \text{and} \qquad \mathbb{P}_x( |\log \|M_n\|-n\lambda_1| \geq n\epsilon) \leq C e^{-\alpha n}.$$
\end{theorem}

This result is not new; it follows from Bougerol's \cite[Th\'{e}or\`{e}me 4.3]{bougerol.thm.limite}. However, we give a different proof. The tools developed for this proof, beyond their aforementioned utility in the Berry--Esseen estimate, will also allow us to prove Theorem \ref{thm.ld.gromov} in the next part. The rest of \S \ref{subsec.ld.markov.matrix} is devoted to its proof.

\subsubsection{Large deviations in Breiman's LLN}
Following Benoist--Quint \cite{bq.clt}, we adopt a slighly more general setting. Let $C$ be a compact metrizable space and $E$ a Polish space. We say that a Markov-Feller transition kernel $Q$ on $Y=E \times C$ covers a Markovian transition kernel $P$ on $E$, if the following diagram commutes

\begin{equation}\label{eq.covering.diag}
 \begin{tikzcd}
E \times C \arrow{r}{Q} \arrow[swap]{d}{\pi_1} & \mathcal{P}(E\times C) \arrow{d}{\pi_1{}_{\ast}} \\%
E \arrow{r}{P}& \mathcal{P}(E)
\end{tikzcd}
\end{equation}
Here, $\mathcal{P}(E)$ (resp.~ $\mathcal{P}(E \times C)$) denotes the set of probability measures on $E$ (resp.~ on $E \times C$), $\pi_1:E \times C \to E$ is the projection map and $\pi_1{}_{\ast}$ is the induced push-forward map. 

Given a bounded continuous function $\varphi:Y \to \R$, we set 
\[
\ell_\varphi^+=\sup_{\eta} \int \varphi \ d\eta \qquad \text{and} \qquad \ell_\varphi^-=\inf_{\eta} \int \varphi \ d\eta
\]
where the supremum and infimum are taken over $Q$-invariant probability measures on $Y$. 

Let us say that a Markov--Feller kernel $P$ on $E$ is uniformly positive recurrent if for every $\epsilon>0$, there exists a compact set $K \subseteq E$ and $N \in \N$ such that for every $x \in E$ and $n \geq N$, we have $\frac{1}{n}\sum_{j=1}^n (\delta_x P^j) (K)>1-\epsilon$. 
The following result is a more general version of \cite[Proposition 3.1]{bq.clt} that one can derive from its proof with a small variation explained below.

\begin{proposition}[Benoist-Quint]\label{prop.bq.3.1}
Let $Q$ be a Markov--Feller transition kernel on $Y=E\times C$ covering a transition kernel $P$ on $E$. Suppose that $P$ is uniformly positive recurrent. Then, for every bounded continuous function $\varphi:Y \to \R$ and $\epsilon>0$, there exists $C_0>0$ and $\alpha_0>0$ such that for every $y \in Y$ and $n \in \N$, we have
$$
\mathbb{Q}_y \left\{(y_0,\ldots) \in Y^\N : \frac{1}{n}\sum_{k=1}^n \varphi(y_k) \in [\ell_\varphi^- - \epsilon,\ell_\varphi^+ + \epsilon]\right\} >1-C_0 e^{-\alpha_0 n}.
$$
\end{proposition}

\begin{proof}
The proof of \cite[Proposition 3.1]{bq.clt} goes through: the uniform convergence \cite[(3.2)]{bq.clt} is the only point that needs care in our non-compact case and it follows from the uniform positive recurrence assumption we imposed on the transition kernel $P$ on $E$. Indeed, suppose that the convergence $\max(\ell_\varphi^+, \frac{1}{n}\sum_{k=1}^nQ^k\varphi) \to \ell_\varphi^+$ is not uniform. Then, one finds a sequence $y_n$ of points and $\epsilon_0>0$ such that for every $n \in \N$, $\frac{1}{n}\sum_{k=1}^nQ^k\varphi(y_n) \geq \ell_\varphi^+ +\epsilon_0$. By the uniform positive recurrence property of $P$, any limit point $\eta$ of $\frac{1}{n}\sum_{k=1}^n \delta_{y_n}Q^k$ projects to $\pi$ (the unique stationary measure for $P$ on $E$) and since $C$ is compact, $\eta$ gives full mass to $Y$. Hence it is a $Q$-invariant probability measure on $Y$ satisfying $\int \varphi \ d\eta \geq \ell_\varphi^+ + \epsilon_0$, a contradiction. 
\end{proof}


We now prove a large deviation result for cocycles associated to group actions (cf.~ \cite[Proposition 3.2]{bq.clt}). Let $\Gamma$ be a locally compact second countable group acting continuously on $C$. Let $X:E \to \Gamma$ be a continuous map and $P$ be a Markov-Feller transition kernel on $E$. We consider the Markov--Feller transition kernel on $Y$ defined as follows: for Borel subsets $A \subset E$ and $B \subset C$ and $y=(x,c)$, we set
\begin{equation}\label{eq.covering.op}
Q(y,A \times B):= P(x,A)  \ 1_{B}(X(x) \cdot c).
\end{equation}
By construction the kernel $Q$ covers the transition kernel $P$ in the sense of \eqref{eq.covering.diag}. For $y_0=(x_0,c) \in Y$, we will denote by $\mathbb{Q}_{y_0}$ the probability measure on $Y^\N$ determined by the kernel $Q$ and the initial distribution $\delta_{y_0}$. Note that $\mathbb{Q}_{y_0}$ is the push-forward of $\mathbb{P}_{x_0}$ by 
\begin{equation}\label{eq.Q.by.P}
\begin{aligned}
 E^\N & \to Y^\N\\
 (x_i) & \mapsto ((x_0,c),(x_1,X(x_0)c),(x_2,X(x_1)X(x_0) c),\ldots,y_n,\ldots)
 \end{aligned}
\end{equation}
where $y_n=(x_n,X(x_{n-1})\ldots X(x_0) c)$.

A continuous cocycle $\sigma:\Gamma \times C \to \R$ is said to have uniform exponential moment if there exists $\alpha_2>0$ and $C_2>0$ such that for every $x \in E$, $\mathbb{E}_x[\sup_{c \in C}e^{\alpha_2\sigma(X(z_1),c)}] \leq C_2$.

\begin{proposition}\label{prop.bq.3.2}
Under the assumptions of Proposition \ref{prop.bq.3.1}, given a continuous cocycle $\sigma: \Gamma \times C \to \R$ with uniform exponential moment for every $\epsilon>0$, there exist $C>0$ and $\alpha>0$ such that for every $x \in E$, $c \in C$, and $n \in \N$, we have 
$$
\P_x\left\{(x_0,\ldots) : \frac{1}{n} \sum_{k=1}^n \sigma(X(x_k),X(x_{k-1})\ldots X(x_0) c) \in [\ell^- -\epsilon,\ell^+ +\epsilon] \right\} \geq 1-Ce^{-\alpha n},
$$
where
$\ell^+=\sup_{\eta} \int \sigma(X(x),c) \ d\eta(x,c),$
with the supremum taken over Borel probability measures on $Y$ that are $Q$-invariant (and the lower-average $\ell^-$ is defined similarly with $\inf$ instead of $\sup$).
\end{proposition}

\begin{proof}
We will write the sum $\sum_{k=1}^n \sigma(X(x_k),X(x_{k-1})\ldots X(x_0) c)$ as a sum of two quantities for which we have exponential concentration bounds -- one of them thanks to Proposition \ref{prop.bq.3.1} and the other thanks to Azuma type concentration bounds for sums of martingale differences (see e.g.~ \cite[Theorem 1.1]{liu-watbled}).

To do this, let $\xi:Y \to \R$ be defined for $y=(x,c)$ as $\xi(y)=\int \sigma(X(z),X(x)c) \ dP_x(z)$. Note that $\xi$ is continuous (since $P$ is Markov--Feller and $\sigma$ is continuous) and bounded (thanks to the uniform exponential moment hypothesis). Furthermore, having fixed $c \in C$, let $\phi_n$ be the sequence of functions defined on $E^\N$ by
$$
\phi_n((x_i))=\sigma(X(x_{n}),X(x_{n-1})\ldots X(x_0)c)-\int \sigma(X(z),X(x_{n-1})\ldots X(x_0)c)) \ dP_{x_{n-1}}(z).
$$

We then have 
\begin{equation}\label{eq.sum.xi.phi}
\sum_{k=1}^n \sigma(X(x_k),X(x_{k-1})\ldots X(x_0) c)=\sum_{k=1}^n\phi_k((x_i)) + \sum_{k=0}^{n-1} \xi(y_k),
\end{equation}
where we recall that $y_k=(x_k,X(x_{k-1})\ldots X(x_0)c)$ and $y_0=(x_0,c)$.

One now readily checks that for every $(x,c) \in Y$, under $\P_x$, $\phi_k$ is a martingale difference sequence with respect to the canonical filtration $\mathcal{F}_n$ on $E^\N$. Indeed, for $P_{x_0}$-a.e.~ $(x_i)$, we have that $\E_x(\phi_n | \mathcal{F}_{n-1})((x_i))$ is equal to
\begin{equation}
\begin{aligned}
\E_{x_{n-1}}[\sigma(X(z_{n}),X(x_{n-1})\ldots X(x_0)c)]-\int \sigma(X(z),X(x_{n-1})\ldots X(x_0)c) \ dP_{x_{n-1}}(z)=0
\end{aligned}
\end{equation}

Thanks to the uniform exponential moment assumption, we can apply \cite[Theorem 1.1]{liu-watbled} and deduce that for every $\epsilon>0$, there exists $C_1>0$ and $\alpha_1>0$ such that for every $(x,c) \in Y$ and $n \in N$, we have
\begin{equation}\label{eq.exp.est.phi}
    \P_x\left\{(x_i) : \frac{1}{n} \sum_{k=1}^n \phi_k((x_i)) \geq \epsilon\right\} \leq C_1 e^{-\alpha_1n}.
\end{equation}

On the other hand, by Proposition \ref{prop.bq.3.1} applied to the function $\xi$ and thanks to the relation \eqref{eq.Q.by.P}, we obtain that for every $\epsilon>0$, there exists $C_0>0$ and $\alpha_0>0$ such that for every $(x,c) \in Y$ and $n \in \N$, we have
\begin{equation}\label{eq.exp.est.xi}
\P_x\left\{(x_i): \frac{1}{n}\sum_{k=0}^{n-1} \xi(x_k,X(x_{k-1})\ldots X(x_0)c) \in [\ell^-_\xi -\epsilon,\ell^+_\xi+\epsilon]\right\} \geq 1-C_0 e^{-n\alpha_0},
\end{equation}
where $\ell_\xi^+=\sup_{\eta}\int \int \sigma(X(z),X(x)c) \ dP_{x}(z)d\eta(x,c)=\sup_\eta \int \sigma (X(x),c) \ d\eta(x,c)$ since $\xi=Q\sigma$ and $\eta$ is $Q$-stationary. In view of \eqref{eq.sum.xi.phi}, the result now follows by \eqref{eq.exp.est.phi} and \eqref{eq.exp.est.xi}.
\end{proof}


We are now ready to prove Theorem \ref{thm.ld.matrix}.

\begin{proof}[Proof of Theorem \ref{thm.ld.matrix}]
We start by proving the first inequality. We will basically show that we can apply Proposition \ref{prop.bq.3.2} and that in that result, we have $\ell^-=\ell^+=\lambda_1$.

Let $C$ be the $(d-1)$-dimensional real projective space $\mathbf{P}(\R^d)$ endowed with the usual action of $\GL_d(\R)$. We take $\Gamma=\GL_d(\R)$ and $X$ as the map $E \to \GL_d(\R)$ in the data of the Markovian product $(M_n)$. Let $Q$ be the Markov--Feller transition kernel on $Y:=E \times C$ constructed as in \eqref{eq.covering.op} covering the kernel $P$ on $E$. Moreover, let $\sigma:\Gamma \times \mathbf{P}(\R^d) \to \R$ be the continuous cocycle given by $\sigma(g,[v])=\log \frac{\|gv\|}{\|v\|}$ where $v$ is any non-zero element in $\R^d$ and $[v]$ denotes its projection to $\mathbf{P}(\R^d)$. Thanks to condition $(A_1)$, the kernel $P$ is uniformly positively recurrent and thanks to condition $(A_2)$, $\sigma$ has a uniform exponential moment. Therefore, Proposition \ref{prop.bq.3.2} yields that for every $\epsilon>0$, there exist $C>0$ and $\alpha>0$ such that for every $x \in E$, non-zero $v \in \R^d$ and $n \in \N$, we have
\begin{equation}\label{eq.apply.3.2.matrix}
\mathbb{P}_x\left( \frac{1}{n}\log \frac{\|M_n v\|}{\|v\|}  \notin [\ell^- -\epsilon,\ell^+ +\epsilon]\right) \leq C e^{-\alpha n}.
\end{equation}
We will now see that there exists a unique $Q$-stationary probability measure on $Y$ and deduce that $\ell^-=\ell^+$. To this end, let $\eta$ be a $Q$-stationary probability measure on $E \times C$. Since $Q$ covers the kernel $P$, the projection of $\eta$ to $E$ is a $P$-stationary probability measure, which is therefore equal to $\pi$ (because $\pi$ is the unique $P$-stationary probability measure on $E$). Hence we can write $\eta=\int \delta_x \otimes \nu_x \ d\pi(x)$, where $\nu_x$ is a probability measure on $C$. By \cite[Lemma 3.4]{bq.3}, $\nu_x$ is the stationary probability measure for the Markov operator $Q_x$ induced by $Q$ on $\{x\} \times C$. But since $x$ is a single state in the state space $E$, $Q_x$ is the Markov operator induced by a probability measure $\mu_x$ on $\Gamma$ (the renewal measure, see \cite[\S 3.1]{prohaska-sert}), i.e.~ $Q_x((x,c),A\times B)=\delta_{x}(A)\cdot(\mu_x \ast \delta_{c}(B))$ for $A \subseteq E$ and $B \subseteq C$. By construction, the semigroup generated by the renewal measure $\mu_x$ is precisely the semigroup $T_x(x)$ (\cite[page 15]{prohaska-sert}). It follows from \cite[Th\'{e}or\`{e}me 5.3.(ii)]{bougerol.comparaison} that for every $x \in E$, the semigroup $T_x(x)$ in $\GL_d(\R)$ is strongly irreducible and proximal. Therefore, by \cite[\S III Theorem 3.1]{bougerol.book}, $\nu_x$ is the unique stationary probability measure of $\mu_x$ and hence it does not depend on $\eta$. This shows that $Q$ has a unique stationary probability measure $\eta$ on $Y$. In particular $\ell^-=\ell^+=:\ell$ in \eqref{eq.apply.3.2.matrix}. Choosing a basis $v_1,\ldots,v_d$ with unit vectors, applying \eqref{eq.apply.3.2.matrix} with each $v_i$, one gets by Borel--Cantelli that for every $x \in E$, $\P_x$ a.s.  
$$
\sup_{i=1,\ldots,d} \frac{1}{n} \log \|M_n v_i\| \underset{n \to \infty}{\longrightarrow} \ell.
$$
Since the supremum over a basis with unit vectors defines a norm on $\Mat_d(\R)$ comparable to an operator norm, one gets by the subadditive ergodic theorem (see \eqref{eq.kingman.matrix}) that $\ell=\lambda_1$.
\end{proof}

\subsection{Large deviation estimates for Markovian random walks on Gromov-hyperbolic spaces}\label{subsec.ld.isometry}

We introduce some basic definitions from metric geometry to state Theorem \ref{thm.ld.gromov} below. 

Let $(H,d)$ be a Gromov-hyperbolic metric space and $o \in H$ a basepoint. Given $x \in H$, let $h_x \in \Lip_{o}^1(H)$ denote the function defined by $h_x(y)=d(x,y)-d(x,o)$, where $\Lip_{o}^1(H)$ is the set of 1-Lipschitz functions on $H$ vanishing at $o$, endowed with the pointwise topology. By taking the closure in $\Lip_{o}^1(H)$, we get a compactification $\overline{H}^h$ of $H$, called the \textit{horofunction compactification}. The compact $\overline{H}^h$ is metrizable if $H$ is separable. In our case, since we will consider (Markovian) random walks on a countable group acting on $H$, we can and will without loss of generality suppose that $H$ is separable (see \cite[Remark 4]{gruber-sisto-tessera}).
The map $x \mapsto h_x$ is injective on $H$ and we usually identify $H$ with its image in $\overline{H}^h$. 
The \textit{Busemann cocycle} $\sigma: \Isom(H) \times \overline{H}^h \to \mathbb{R}$ is defined by $\sigma(\gamma,h)=h(\gamma^{-1}o)$. Note that for $o \in H \subseteq \overline{H}^h$, $\sigma(\gamma,o)=d(\gamma o,o)=:\kappa(\gamma)$ is the displacement functional. Recall finally that an element $\gamma \in \Isom(H)$ is called loxodromic if it has precisely two fixed points on the Gromov boundary $\partial H$ of $H$.

We say that a Markovian random walk $(M_n)$ on $\Isom(H)$ is \textit{non-elementary} if there does not exist a finite number of maps $V_i$ (say, $i=1,\ldots,t$) from $E$ to the Gromov boundary $\partial H$ such that for $x \in E$, denoting $W(x)=\cup_{i=1}^t V_i(x)$, we have, for every $x_0 \in E$ and $n \in \N$, $P_{x_0}$-a.s. $M_n W(x_0)=W(x_n)$. Moreover, we say that the Markovian random walk $(M_n)$ satisfies \textit{Condition $(A_2')$} if there exist positive constants $a$ and $B$ such that $\mathbb{E}_x[e^{a\kappa(M_1)}] \leq B$ for every $x \in E$. 

\begin{theorem}[Markovian random walks on Gromov-hyperbolic spaces]\label{thm.ld.gromov}
Let $H$ be a geodesic Gromov-hyperbolic space and $(M_n)$ be a non-elementary Markovian random walk on $\Isom(H)$ satisfying $(A_1)$ and $(A_2')$. Then, there exists a constant $\ell_\Lambda \geq 0$ such that for every $\epsilon>0$, there exists $\alpha>0$ and $C>0$ such that for every $x \in E$, $\xi \in \overline{H}^h$ and $n \in \N$, we have
$$
\P_x(|\sigma(M_n,\xi)-n\ell_\Lambda| \geq n\epsilon)\leq Ce^{-\alpha n}.
$$
\end{theorem}
The constant $\ell_\Lambda$ is called the \textit{drift of the Markovian random walk} $(M_n)$. We note that specializing to $\xi=o \in H$, the above statement boils down to large deviation estimates for the displacement function $\kappa(\cdot)$.

This result generalizes the assertion concerning the uniqueness of zero in the recent work \cite{BMSS} in the iid setting (let us note that even more recently, Gou\"{e}zel \cite{gouezel.21} managed to get rid of the exponential moment assumption in the same setting). In the iid case, the dependence of $\alpha$ on $\epsilon$ has been specified and quantitative estimates have been recently obtained when $H$ is proper (see \cite{aoun-sert.gromov,aoun-mathieu-sert}). Finally, see also the recent work of Goldsborough--Sisto \cite{gold-sisto} for another perspective on Markovian random products of isometries.

\begin{proof}
We aim to apply Proposition \ref{prop.bq.3.2}. To this end, let $C=\overline{H}^h$ and $\Gamma$ be the countable group generated by the image of the map $X:E \to \Isom(H)$ in the data of the Markovian random walk $(M_n)$. Recall that the group of isometries $\Isom(H)$ acts on $\overline{H}^h$ by homeomorphisms given, for $\gamma \in \Isom(H)$,  $h \in \overline{H}^h$ and $y \in M$, by $(\gamma\cdot h)(y)=h(\gamma^{-1}y)-h(\gamma^{-1}o)$ and the Busemann cocycle is a continuous cocycle over this action. Let $Q$ be the Markov--Feller transition kernel on $Y:=E \times C$ constructed as in \eqref{eq.covering.op} covering the kernel $P$ on $E$. Thanks to condition $(A_1)$, the kernel $P$ is uniformly positively recurrent and thanks to condition $(A_2)$, $\sigma$ has a uniform exponential moment. Therefore, Proposition \ref{prop.bq.3.2} implies that for every $\epsilon>0$, there exist $C>0$ and $\alpha>0$ such that for every $x \in E$,  $\xi \in \overline{H}^h$ and $n \in \N$, we have
\begin{equation*}
\mathbb{P}_x\left( \frac{1}{n}\sigma(M_n,\xi)  \notin [\ell^- -\epsilon,\ell^+ +\epsilon]\right) \leq C e^{-\alpha n}.
\end{equation*}
We now let $\ell_\Lambda$ be the constant given by the subadditive ergodic theorem as the $\P_\pi$ almost sure limit of $\frac{1}{n} \kappa(M_n)$ as $n \to \infty$. It remains to show that $\ell^+=\ell^-=\ell_\Lambda$. If $\ell_\Lambda=0$ this equality is easy to see, so we suppose that $\ell_\Lambda>0$. Here, a different argument is needed compared to the corresponding part in the proof of Theorem \ref{thm.ld.matrix}, since unlike therein, in the current setting we do not know whether there is only one $Q$-stationary probability measure on $Y$. Let $\eta$ be a $Q$-stationary probability measure on $Y$. By the Chacon--Ornstein ergodic theorem,
\[
\eta\left\{(x,c) \in Y : \frac{1}{n} \sum_{k=0}^{n-1} Q^k \sigma(x,c) \underset{n \to \infty}{\longrightarrow} \int \sigma \ d \eta \right\}=1.
\]
By specializing to such $(x_0,c)$, using the fact that for every $y=(x',c')$, $\mathbb{Q}_y$ is the pushforward of $\mathbb{P}_{x'}$ by the map \eqref{eq.Q.by.P}, we get that
\begin{equation}\label{eq.chacon.nu}
\lim_{n \to \infty}\frac{1}{n}\E_{x_0}[\sigma(X(z_{n-1})\ldots X(z_1) X(x_0),c)]=\int \sigma d\eta.    
\end{equation}
For $n \geq 1$, let $\tau_x(n)$ denote the random variable, defined on $E^\N$, which is given by $n^{th}$-return time to $x$. Thanks to condition $A_1$, $\tau_x(1)$ (equivalently $\tau_x(n)$ for every $n \in \N$) has a finite exponential moment. For $x \in E$, let $\mu_x$ be the (renewal) probability measure defined as $\mu_x(g)=\mathbb{P}_x(X(z_{\tau_x(1)-1})\ldots X(z_1)X(x)=g)$. It is easy to see that $\mu_x$ has a finite exponential moment, i.e.~ $\int e^{\beta\kappa(g)} d\mu_x(g)<\infty$ for some $\beta>0$. The support of $\mu_x$ is the subsemigroup $X(x)^{-1} T_x(x) X(x)$ of $\Gamma$ defined in the same way as in \S \ref{subsec.simplicity}. Since the Markovian product $(M_n)$ is non-elementary and has positive drift, the semigroup $T_x(x)$ is clearly unbounded. Moreover, it follows from the same argument as in the proof of \cite[Th\'{e}or\`{e}me 5.3.(ii)]{bougerol.comparaison} that $T_x(x)$ does not stabilize a finite collection of points in the Gromov boundary $\partial H$. Now, \cite[Proposition 3.1]{CCMT} implies that the group generated by $T_x(x)$ contains two independent loxodromics and then \cite[Theorem 6.2.3 and Proposition 6.2.14]{das-simmons-urbanski} imply that the semigroup $T_x(x)$ is non-elementary (i.e.~ contains two independent loxodromics).

For the rest, on the one hand, it is not hard deduce from \eqref{eq.chacon.nu} that 
\begin{equation}\label{eq.chacon.nu.return}
    \frac{1}{n}\E_{\mu_{x_0}}[\sigma(g_n \ldots g_1 ,c)]\to E[\tau_{x_0}(1)] \int \sigma \ d\eta
\end{equation}
and since $\mu_{x_0}$ is non-elementary and has a finite exponential moment, it follows from \cite[Lemma 3.9]{aoun-mathieu-sert} and \eqref{eq.chacon.nu.return} that 
\begin{equation}\label{eq.chacon.to.kappa}
    \frac{1}{n}E_{\mu_{x_0}}[\kappa(g_n \ldots g_1)] \to E[\tau_{x_0}(1)]\int \sigma \ d\eta.
\end{equation}
On the other hand, the left-hand-side of \eqref{eq.chacon.nu.return} converges to $E[\tau_{x_0}(1)] \ell_\Lambda$. This shows that $\int \sigma d\eta=\ell_\Lambda$. Since $\eta$ is an arbitrary $Q$-stationary probability measure, this shows that $\ell^+=\ell^-=\ell_\Lambda$, completing the proof.
\end{proof}

\begin{remark}\label{rk.positive.gromov}
We remark that in Theorem \ref{thm.ld.gromov}, we cannot exclude the possibility that $\ell_\Lambda=0$. However, a handy characterization of when $\ell_\Lambda>0$ follows from the previous proof. Indeed, let $x \in E$ and $T_x(x)$ be the semigroup above. As in the proof above, since the Markovian product $(M_n)$ is non-elementary, the semigroup $T_x(x)$ does not stabilize a finite collection of points in the Gromov-boundary $\partial H$. Moreover, as shown above if $T_x(x)$ is unbounded, $T_x(x)$ contains two independent loxodromics. The linear escape result in \cite{maher-tiozzo} implies that the drift of $\mu_x$ is positive and hence from the proof above, we get $\ell_\Lambda>0$.  On the other hand, it is easy to see that $\ell_\Lambda=0$ if for some (equivalently all) $x \in E$, the semigroup $T_x(x)$ is bounded.
\end{remark}

\subsection{Central limit theorem with Berry--Esseen type estimates}\label{subsec.berry.esseen.markov}

Specializing to $t=1$ in \eqref{eq.defn.sn}, 1.~ of Theorem \ref{thm.bougerol.wiener.lil} says that the central limit theorem holds: for every $x \in E$ and $a \in \R$, we have
\begin{equation}\label{eq.gives.clt}
\P_x \left(\log \|M_n\|-n\lambda_1 \leq a\sqrt{n} \right)   \underset{n \to \infty}{\longrightarrow}  \frac{1}{\sigma \sqrt{2\pi}} \int_{-\infty}^a e^{-\frac{s^2}{2\sigma^2}} \ ds.
\end{equation}

In the following result, we give the Berry--Esseen type bound for the convergence rate in \eqref{eq.gives.clt}. Our main interest in the Berry--Esseen bound is that it will be used to obtain a \textit{quantitative} counting central limit theorem on spheres of the Cayley graph of a Gromov-hyperbolic group. Unlike Theorem \ref{thm.bougerol.wiener.lil}, it is not simple to deduce the Berry--Esseen type bound for $\log \|M_n\|$ from that of $\log \|M_n v\|$ --- the latter was proven by Bougerol \cite{bougerol.thm.limite}. Indeed, even in the iid case, although the Berry--Esseen bound for $\log \|M_n v\|$ has been known since the work of Le Page \cite{lepage}, the bounds for the matrix norm were only recently studied \cite{cuny.grad, cuny.jan, XGL}. Below, we give a version of these results for the Markovian case adapting the approach of Xiao--Grama--Liu \cite{XGL} and using our large deviation estimates (replacing the large deviation ingredient of \cite{XGL} from \cite{bq.book} in the iid case).

\begin{theorem}[Berry--Esseen bound in CLT]\label{thm.bougerol.berryess}
Let $(M_n)$ be a strongly irreducible and $1$-contracting Markovian product satisfying $(A_1)$ and $(A_2)$. Let $\|\cdot\|$ be a fixed norm on $\R^d$. Then, there exists a constant $D>0$ such that for every $x \in E$, $a \in \R$, and $n \in \N$, we have
$$
\left| \P_x \left(\log \|M_n\|-n\lambda_1 \leq a\sqrt{n} \right)  - \frac{1}{\sigma \sqrt{2\pi}} \int_{-\infty}^a e^{-\frac{s^2}{2\sigma^2}} \ ds \right| \leq \frac{D \log n}{\sqrt{n}}.
$$
\end{theorem}


Regarding the central limit theorem, we signal that in view of the more recent progress of Benoist--Quint \cite{bq.clt,bq.clt.hyperbolic} (see also \cite{cuny.grad} for Berry--Esseen estimates) optimizing the moment hypothesis in the central limit theorem for the iid case (respectively, improving the Berry--Esseen estimates), it is probable that \eqref{eq.gives.clt} and some Berry--Esseen type estimates hold under a polynomial moment hypothesis (we do not pursue these directions).

To prove the Berry--Esseen estimate in Theorem \ref{thm.bougerol.berryess}, we will need some further results on large deviation estimates which are given in the next two lemmas.

\begin{lemma}\label{lemma.lambda.gap.ld}
Under the assumptions of Theorem \ref{thm.bougerol.berryess}, for every $\epsilon>0$, we have
$$
\limsup_{n \to\infty} \frac{1}{n} \log \P_x\left(\frac{1}{n} \log \frac{\|M_n\|^2}{\|\wedge^2 M_n\|} \leq \lambda_1-\lambda_2 - \varepsilon \right)<0,
$$
uniformly in $x \in E$.
\end{lemma}

\begin{proof}
Since the top Lyapunov exponent of the Markovian product $(\wedge^2 M_n)$ is $\lambda_1+\lambda_2$ and $(\wedge^2 M_n)$ satisfies $(A_1)$ and $(A_2)$, by Theorem \ref{thm.ld.matrix}, it suffices to show that for every $\epsilon>0$, we have the following uniformly in $x \in E$:
\begin{equation}\label{eq.to.prove.wedge.ld}
\limsup_{n \to\infty} \frac{1}{n} \log \P_x\left(\frac{1}{n} \log \|\wedge^2 M_n\| \geq \lambda_1+\lambda_2 +\varepsilon \right)<0.
\end{equation}

To prove this, we will apply Proposition \ref{prop.bq.3.2}. Let $C=\mathbb{P}(\wedge^2 \R^d)$ and $G=\GL_d(\R)$ and $X$ as the map $E \to \GL_d(\R)$ in the data of the Markovian product $(\wedge^2 M_n)$, let $Q$ be the Markov--Feller transition kernel on $Y:=E \times C$ constructed as in \eqref{eq.covering.op} covering the kernel $P$ on $E$. Moreover, let $\sigma:G \times \mathbf{P}(\wedge^2\R^d) \to \R$ be the continuous cocycle given by $\sigma(g,[v\wedge w])=\log \frac{\|\wedge^2 g(v \wedge w)\|}{\|v \wedge w\|}$ where $v\wedge w$ is a line in $\wedge^2 \R^d$ and $[v \wedge w]$ denotes its projection to $\mathbf{P}(\wedge^2 \R^d)$. Thanks to condition $(A_1)$, the kernel $P$ is uniformly positively recurrent and thanks to condition $(A_2)$, $\sigma$ has a uniform exponential moment. Therefore, we can apply Proposition \ref{prop.bq.3.2} and deduce that for every $\epsilon>0$, there exists $C>0$ and $\alpha>0$ such that for every $x \in E$, non-zero $v \wedge w \in \wedge^2 \R^d$ and $n \in \N$, we have
\begin{equation*}
\mathbb{P}_x\left( \frac{1}{n}\log \frac{\|\wedge^2 M_n (v \wedge w)\|}{\|v \wedge w\|}  \notin [\ell^- -\epsilon,\ell^+ +\epsilon]\right) \leq C e^{-\alpha n}.
\end{equation*}
By choosing a basis of $\wedge^2 \R^d$ as in the proof of Theorem \ref{thm.ld.matrix}, we only need to show that $\ell^+ \leq \lambda_1+\lambda_2$. Let $\delta>0$ be given and $\eta$ be a $Q$-stationary and ergodic probability measure on $Y$ with $\int \sigma \ d\eta \geq \ell^+ -\delta$. By the Chacon--Ornstein ergodic theorem, we have 
$$
\eta\left\{(x,c) \in Y : \frac{1}{n} \sum_{k=0}^{n-1} Q^k \sigma(x,c) \underset{n \to \infty}{\longrightarrow} \int \sigma d \eta \right\}=1.
$$
By specializing to such $(x_0,c)$, using the fact that for every $y=(x',c')$, $\mathbb{Q}_y$ is the pushforward of $\mathbb{P}_x$ by the map  \eqref{eq.Q.by.P}, we get that 
\begin{align*}
    \ell^+-\delta \leq \int \sigma d\eta &= \lim_{n \to \infty}\frac{1}{n}\E_{x_0}[\sigma(X(x_{n-1})\ldots X(x_0),c)]\\
    &\leq \lim_{n \to \infty} \frac{1}{n} \E_{x_0}[\log \|\wedge^2 X(x_{n-1}) \ldots X(x_0)\| ].
\end{align*}
It remains to observe that the last term  is bounded above by $\lambda_1+\lambda_2$. In fact we claim that it is equal to $\lambda_1+\lambda_2$. Indeed, by subadditive ergodic theorem, $\P_\pi$-a.s.~ $\frac{1}{n}\log \|\wedge^2 X(x_{n-1})\ldots X(x_0))\| \to \lambda_1+\lambda_2$. Since $\P_{x_0}$ is absolutely continuous with respect to $\P_\pi$, this convergence also holds true $\P_{x_0}$-a.s. Since the sequence $\frac{1}{n}\log \|\wedge^2 X(x_{n-1}) \ldots X(x_0)\|$ is uniformly integrable (thanks to condition $(A_2)$), the result follows.
\end{proof}

The next lemma is the Markovian version of \cite[Lemma 17.8]{bq.book} which was used to deduce a local limit theorem for the norms $\log \|M_n\|$ from a local limit theorem for vector norms $\log\|M_n v\|$ in the iid setting. Thanks to our above large deviation estimates, Benoist--Quint's proof applies in our setting as we indicate below.
\begin{lemma}\label{lemma.cartan.iwasawa}
Under the assumptions of Theorem \ref{thm.bougerol.berryess}, for every $\epsilon>0$, there exists $\ell_0 \in \N$, and $c>0$ such that for every $n \geq \ell \geq \ell_0$ and non-zero $v \in R^d$ and $x \in E$, we have
$$
\P_x \left( \left|\log \|M_n\| -\log \frac{\|M_n v\|}{\|M_\ell v\|} -\log \|M_\ell\| \right| > e^{-\epsilon \ell} \right)<e^{-c\ell}.
$$
\end{lemma}
\begin{proof}
We can apply the proof of the result \cite[Lemma 17.8 and (17.9)]{bq.book} which makes use of various large deviation estimates all of which are established in our more general setting. Namely, the ingredient \cite[Proposition 14.3]{bq.book} is similarly obtained in our setting using Theorem \ref{thm.ld.matrix} and Lemma \ref{lemma.lambda.gap.ld} which we proved for this purpose. Once equipped with this ingredient, Benoist--Quint only use the iid version of Lemma \ref{lemma.lambda.gap.ld} (see \cite[(17.10)]{bq.book}) and the linear algebraic lemma \cite[Lemma 14.2]{bq.book} and hence this part of the proof goes through in our setting as well. We omit the details in order not to burden the exposition with many more linear algebraic notions that will only be used in a repetitive proof.
\end{proof}
Equipped with the previous two lemmas, we can give the proof of Theorem \ref{thm.bougerol.berryess} adapting the approach of \cite{XGL}.
\begin{proof}
Let $F$ denote the cumulative distribution function of the Gaussian $\mathcal{N}(0,1)$, and for $n \geq 1$, $x \in E$ and $t \in \R$, set
$$
I_{x,n}(t)= \mathbb{P}_x\left(\frac{\log \|M_n\|-n\lambda_1}{\sigma \sqrt{n}}\leq t\right).
$$
Since for any $v \in \R^d$ of unit norm, we have $\|M_nv\| \leq \|M_n\|$, for any $v \in \R^d$ with $\|v\|=1$, by \cite[Th\'{e}or\`{e}me 4.1]{bougerol.thm.limite}, we have
$$
I_{x,n}(t) \leq \mathbb{P}_x\left(\frac{\log \|M_nv\|-n\lambda_1}{\sigma \sqrt{n}}\leq t\right) \leq F(t) + \frac{C}{\sqrt{n}}.
$$

The non-trivial bound is therefore the lower bound for $I_{x,n}(t)$ which we now turn to. By Lemma \ref{lemma.cartan.iwasawa}, for any $\epsilon>0$, there exist $\ell_0 \in \N$ and $c>0$ such that for every $n \geq \ell \geq \ell_0$, we have
\begin{equation}\label{eq.berry1}
\begin{aligned}
I_{x,n}(t) & \geq \mathbb{P}_x\left(\frac{\log \|M_n\|-n\lambda_1}{\sigma \sqrt{n}}\leq t \; \; \;  \text{and} \; \; \; \left|\log \|M_n\| -\log \frac{\|M_n v\|}{\|M_\ell v\|} -\log \|M_\ell\| \right| \leq e^{-\epsilon \ell}\right)\\
& \geq \mathbb{P}_x\left( \frac{\log \|M_n v\|-\log \|M_\ell v\|+\log \|M_\ell \| -n \lambda_1+e^{-\epsilon \ell}}{\sigma \sqrt{n}}\leq t\right) - e^{c \ell}.
\end{aligned}
\end{equation}

By Theorem \ref{thm.ld.matrix}, for every $\lambda'>\lambda_1$, there exists $c'>0$ such that for every $\ell \geq 1$, for every $x \in \E$, we have
$$
\mathbb{P}_x(\log \|M_\ell\| \geq \ell \lambda') \leq e^{-c' \ell}.
$$
Therefore, it follows from \eqref{eq.berry1} that for every $x\in E$ and $n \geq \ell \geq \ell_0$, we have
\begin{equation}\label{eq.berry2}
I_{x,n}(t) \geq \mathbb{P}_x\left( \frac{\log \|M_n v\|-\log \|M_\ell v\|+\ell \lambda' -n \lambda_1+e^{-\epsilon \ell}}{\sigma \sqrt{n}}\leq t\right) -e^{c' \ell} - e^{c \ell}.
\end{equation}
Notice now that we can rewrite $\log \|M_n v\| - \log \|M_\ell v\|$ as $\log \|X_n \ldots X_{\ell+1}\widetilde{v}\|$, where $\widetilde{v}=\frac{M_\ell v}{\|M_\ell v\|}$. To exploit this cocycle property, for $\ell \in \N$, let $\mathcal{F}_\ell$ denote the $\sigma$-algebra generated by the first $\ell$ steps $z_1,\ldots,z_\ell$ of the Markov chain on $E$. Conditioning on the first $\ell$-steps, we have
\begin{equation}\label{eq.berry3}
\begin{aligned}
       & \mathbb{P}_x\left( \frac{\log \|M_n v\|-\log \|M_\ell v\|+\ell \lambda' -n \lambda_1+e^{-\epsilon \ell}}{\sigma \sqrt{n}}\leq t\right)\\ & =\mathbb{E}_x\left( \mathbb{P}_x\left(\frac{\log \|M_n v\|-\log \|M_\ell v\|+\ell \lambda' -n \lambda_1+e^{-\epsilon \ell}}{\sigma \sqrt{n}}\leq t | \mathcal{F}_\ell\right)\right)\\
       & \geq \mathbb{E}_x\left(\inf_{v \in \R^d, \|v\|=1} \mathbb{P}_{z_\ell}\left( \frac{\log \|M_{n-\ell}v\|+\ell \lambda' -n \lambda_1+e^{-\epsilon \ell}}{\sigma \sqrt{n}}\leq t\right)\right)\\
       &=\mathbb{E}_x\left(\inf_{v \in \R^d, \|v\|=1} \mathbb{P}_{z_\ell}\left( \frac{\log \|M_{n-\ell}v\|-(n-\ell) \lambda_1 }{\sigma \sqrt{n-\ell}} \leq T_n \right)\right),
\end{aligned}
\end{equation}
where $T_n$ is the random variable  
$$
T_n=\frac{\sqrt{n}-\sqrt{n-\ell}}{\sigma \sqrt{n(n-\ell)}} \log \|M_{n-\ell}v\|-\frac{\ell \lambda'-n\lambda_1+e^{-\epsilon \ell}}{\sigma \sqrt{n}}+t-\frac{\lambda_1 \sqrt{n-\ell}}{\sigma}.
$$
By Theorem \ref{thm.ld.matrix}, up to possibly reducing $c'>0$, we have that for every $n \geq \ell \geq 1$, for every $v \in \R^d$ with $\|v\|=1$ and $y \in E$,
$$
\mathbb{P}_{y}(\log \|M_{n-\ell}v\|>\lambda'(n-\ell))\leq e^{-c' (n-\ell)}
$$
so that by \eqref{eq.berry3}, we have
\begin{equation}\label{eq.berry4}
    \begin{aligned}
&\mathbb{P}_x\left( \frac{\log \|M_n v\|-\log \|M_\ell v\|+\ell \lambda' -n \lambda_1+e^{-\epsilon \ell}}{\sigma \sqrt{n}}\leq t\right)\\ & \geq \mathbb{E}_x\left(\inf_{v \in \R^d, \|v\|=1} \mathbb{P}_{z_\ell}\left( \frac{\log \|M_{n-\ell}v\|-(n-\ell) \lambda_1 }{\sigma \sqrt{n-\ell}} \leq t_n\right)\right)-e^{c'(n-\ell)},
\end{aligned}
\end{equation}
where $t_n$ is the constant
$$
t_n=\frac{\sqrt{n}-\sqrt{n-\ell}}{\sigma \sqrt{n(n-\ell)}} \lambda'(n-\ell)-\frac{\ell \lambda'-n\lambda_1+e^{-\epsilon \ell}}{\sigma \sqrt{n}}+t-\frac{\lambda_1 \sqrt{n-\ell}}{\sigma}.
$$
Now applying once more \cite[Th\'{e}or\`{e}me 4.1]{bougerol.thm.limite} to \eqref{eq.berry4} and combining it with \eqref{eq.berry2}, we get that for every $n \geq \ell \geq \ell_0$ and $x \in E$
$$
I_{x,n}(t) \geq  F(t_n) - \frac{C}{\sqrt{n-\ell}}-e^{-c'(n-\ell)}-2e^{-c''\ell}
$$
with $0<c'':=\min \{c,c'\}$. Now using the expression of $t_n$ above, one gets that for any $r>0$ fixed, letting $\ell=\lfloor r \log n \rfloor$, we have $|t-t_n| \leq \frac{D_r \log n}{\sqrt{n}}$ for some $D_r \in (0,\infty)$ and every $n \in \N$ and $t \in \R$. Using this and the fact that $F$ is the cumulative distribution function of the standard Gaussian $\mathcal{N}(0,1)$, one deduces by elementary calculus that choosing $\ell=\lfloor \frac{1}{c''}\log n\rfloor$, we have that there exists $D \in (0,\infty)$ such that for every $n \in \N$, $x \in E$ and $t \in \R$, we have
$$
I_{x,n}(t) \geq F(t)-\frac{D \log n}{\sqrt{n}},
$$
as required.
\end{proof}

\subsection{Finite state versions without condition $(A_1)$}\label{subsec.finite.markov}
In our applications, we will need to deal with Markovian products associated to Markov chains on finite state spaces which are irreducible but not necessarily aperiodic. Such chains never satisfy the uniform recurrence condition $(A_1)$. However, it is not hard to deduce versions of above limit theorems for such finite state chains by considering the Markovian products along periodic times $(pn)_{n \in \N}$, where $p \in \N$ denotes the period of the Markov chain. The goal of this part is to briefly record these versions of the above limit theorems for later use.

Let $E$ be a finite state space and $P$ an irreducible Markovian transition kernel on $E$. We denote by $p \in \N$ the period of $P$ and, for $i=0,\ldots,p-1$, by $E_i$ the periodic components of $E$. Let a map $X:E \to \Gamma$ to a group $\Gamma$ be given and $(M_n)$ be the associated Markovian product (recall that $M_n=X(z_n) \ldots X(z_1)$, where $(z_n)_{n \geq 0}$ denotes the Markov chain on $E$). The Markovian product $(M_n)$ does not necessarily satisfy condition $(A_1)$; we will associate some auxiliary Markovian products that will satisfy it. To do this, for $i=0,\ldots,p-1$, let $\widehat{E}^i$ be the set of length $p$-paths based at $E_i$, i.e.~
$$
\widehat{E}^i:=\{(x_1,\ldots,x_p) : P(x_j,x_{j+1})>0 \; \text{for} \; j=1,\ldots, p-1, \; \text{and}\; x_1 \in E_{i+1}\},
$$
where $j$'s are considered modulo $p$. We introduce a Markovian transition kernel $\widehat{P}^i$ on $\widehat{E}^i$ by setting, for $(x_1,\ldots,x_p)$ and $(y_1,\ldots,y_p)$ in $\widehat{E}^i$, 
$$
\widehat{P}^i((x_1,\ldots,x_p),(y_1,\ldots,y_p))=P(x_p,y_1) \prod_{j=1}^{p-1} P(y_j,y_{j+1}).
$$
It is easily checked that $\widehat{P}^i$ defines an irreducible and aperiodic Markovian kernel. 
Now, consider the map 
\begin{equation*}
\begin{aligned}
\widehat{X}:\widehat{E}^i &\to \Gamma,\\
\widehat{X}(x_1,\ldots,x_p)& \mapsto X(x_p) \ldots X(x_1).
\end{aligned}
\end{equation*}
We construct a Markovian product $(\widehat{M}^i_n)$ for each $i=0,\ldots,p-1$ in the usual way. Since $\widehat{P}^i$ is aperiodic and $\widehat{E}^i$ is finite, the Markovian product $(\widehat{M}^i_n)$ automatically satisfies conditions $(A_1)$ and $(A_2)$ (and $(A_2')$ in the setting of Theorem \ref{thm.ld.gromov}).
Moreover, for every $i=0,\ldots,p-1$ and $(x_1,\ldots,x_p) \in \widehat{E}^i$, under $\widehat{P}^i_{(x_1,\ldots,x_p)}$, the distribution of $(\widehat{M}_n^i)$ is the same as the distribution of $(M_{pn})$ unde allow to use the products $(\widehat{M}^i_n)$ to control the product $(M_n)$

We therefore aim to apply the above limit theorems to the Markovian products $(\widehat{M}^i_n)$ and deduce the corresponding limit theorems for the products $(M_{pn})_{n \in \N}$, and then use the fact that the operator norm $\log \|\cdot\|$ and the displacement $\kappa(\cdot)$ (in the setting of Theorem \ref{thm.ld.gromov}) is subadditive to deduce the same limit theorems for the Markovian product $(M_n)$ along all times $n \in \N$. To this end, we also need to relate $1$-contracting and strong irreducibility assumptions on $(M_n)$ and $(\widehat{M}_n^i)$.

\begin{lemma}\label{lemma.mn.to.hatmn}
The Markovian product $(M_n)$ is 1-contracting/strong irreducible/non-elementary if any only if the Markovian product $(\widehat{M}_n^i)$ is, respectively, 1-contracting/ strong irreducible/non-elementary for some (equivalently all) $i=0,\ldots,p-1$.
\end{lemma}
The proof is elementary, we briefly indicate the argument.
\begin{proof}
Suppose $(M_n)$ is $1$-contracting, then there exists $x \in E$ and a sequence $g_n \in T_x$ such that $g_n/\|g_n\|$ converges to a rank one linear transformation. Writing $g_n$ as a product of elements $X(y)$ for $y \in E$ and discarding the last elements to make the length divisible by the period $p$, we find a finite set $F$ and for each $n \in \N$ an element $h_n \in F$ such that, if necessary passing to a subsequence, $h_n^{-1}g_n$ belongs to $T_{(x,x_1,\ldots,x_{p-1})}$ for some $x_1,\ldots,x_{p-1} \in E$. Up to further passing to a subsequence, $h_n$ stabilizes and $h_n^{-1}g_n/\|h_n^{-1}g_n\|$ converges to a rank one transformation. The converse implication (and the statement that some $i$ is equivalent to all $i$) is clear.

Suppose now that for some $i=0,\ldots,p-1$,  $(\widehat{M}_n^i)$ is not strongly irreducible. Then for every $(x_1,\ldots,x_p) \in \widehat{E}_i$, there exists a union of a  finite collection of proper subspaces $W(x_1,\ldots,x_p)$ such that $\widehat{P}^i_{(x_1,\ldots,x_p)_0}$ a.s.~ $\widehat{M}_n^i W((x_1,\ldots,x_p)_0)=W((x_1,\ldots,x_p)_n)$. Using this, first, one verifies that $W(x_1,\ldots,x_p)$ only depends on $x_p$. We set $W(x_p):=W(x_1,\ldots,x_p)$ for any $(x_1,\ldots,x_p) \in \widehat{E}^i$ which is hence well-defined. Second, one checks that for any $x_k \in \widehat{E}^{i+k}$ ($i+k$ considered modulo $p$), the union of subspaces  $W(x_k):=X(x_{k})\ldots X(x_1)W(x_p)$ where $P(x_p,x_1)>0$ and $P(x_i,x_{i+1})>0$ for every $i=1,\ldots,k-1$, is well-defined (i.e.~ does not depend on the path $(x_1,\ldots,x_{p-1}$). Finally, one verifies that for every $x_0 \in E$, $P_{x_0}$ a.s.~ $M_nW(x_0)=W(x_n)$, i.e.~ $(M_n)$ is not strongly irreducible. The other implications are clear and the statement about non-elementariness is proven in the same way as strong irreducibility.
\end{proof}

Combining the constructions above and the previous lemma, one readily deduces the following from Theorems \ref{thm.bougerol.wiener.lil} and \ref{thm.ld.matrix}. 
\begin{theorem}\label{thm.irred.finite.markov}
Let $E$ be a finite set, $P$ an irreducible Markovian transition kernel on $E$, $X:E \to \GL_d(\R)$ a map, and $(M_n)$ the associated Markovian product on $\Gamma$. Suppose that $(M_n)$ is strongly irreducible and $1$-contracting. Then there exist positive constants $\Lambda$ and $\sigma$ such that for every $x \in E$\\[2pt]
\indent 
1. the sequence of $C([0,1])$-valued random variables defined by
\begin{equation}\label{eq.defn.sn2}
S_{n}(t)=\frac{1}{(n\sigma^2)^{1/2}} \left( \log \|M_{\lfloor tn \rfloor}\|-nt\Lambda + (nt-\lfloor nt \rfloor) (\log \|M_{\lfloor tn \rfloor +1}\| - \log\|M_{\lfloor tn \rfloor}\|) \right)
\end{equation}
converges to the Wiener measure $\mathcal{W}$ as $n \to \infty$;\\[2pt]
\indent 2. for $\P_x$-a.e.~ $\omega$, the set of limit points of the sequence $\left(\frac{(S_{n}(t))(\omega)}{2\log \log n}\right)_{n \in \N}$ is the compact set given in 2.~ of Theorem \ref{thm.bougerol.wiener.lil}; and,\\[2pt]
\indent 3. for every $\epsilon>0$, 
$$
\mathbb{P}_x( |\log \|M_n\|-n\Lambda| \geq n\epsilon) \leq C e^{-\alpha n}.$$
\end{theorem}

\begin{remark}\label{rk.markov.hyperbolic.case}
Similarly, using Theorem \ref{thm.ld.gromov} one obtains the following statement for a non-elementary Markovian product on $\Isom(H)$ associated to a finite irreducible Markov chain: there exists $\ell_\Lambda>0$ such that for every $\epsilon>0$ and $x \in E$, we have
$$
\limsup_{n \to \infty} \frac{1}{n} \log \mathbb{P}_x(|d(M_n \cdot o,o)-n\ell_\Lambda| \geq n\epsilon)<0.$$
The proof is very similar to the proof of the previous result and it is omitted.
\end{remark}

\begin{proof}[Proof of Theorem \ref{thm.irred.finite.markov}] 
1. Thanks to Lemma \ref{lemma.mn.to.hatmn}, we can apply Theorem \ref{thm.bougerol.wiener.lil} to each $(\widehat{M}_n^i)$ and get that for every $i=0,\ldots,p-1$, there exist constants $\sigma_{i}>0$ and $\Lambda_{i} \in \R$ such that for every $x=(x_1,\ldots,x_p) \in \widehat{E}_i$, we have that under $\P_x$,
\begin{equation}\label{eq.brown.hat1}
\frac{1}{(n\sigma_{i}^2)^{1/2}} \left( \log \|\widehat{M}^{i}_{\lfloor tn \rfloor}\|-nt\Lambda_i + (nt-\lfloor nt \rfloor) (\log \|\widehat{M}^{i}_{\lfloor tn \rfloor +1}\| - \log\|\widehat{M}^{i}_{\lfloor tn \rfloor}\|) \right) \underset{n \to \infty}{\overset{\mathcal{L}}{\longrightarrow}} \mathcal{W}.
\end{equation}
Recall that the distribution (denoted $\mathcal{L}_{\P_{(x_1,\ldots,x_p)}}(\widehat{M}^{i}_n)$) of $\widehat{M}^{i}_n$ under $\P_{(x_1,\ldots,x_p)}$ is equal to that of $M_{np}$ under $\P_{x_p}$ for each $n \in \N$ (i.e.~ $\mathcal{L}_{\P_{x_p}}(M_{np})$). 
Therefore, specializing to $t=1$ in the previous displayed equation, this implies that for $x \in E_i$, under $\P_x$, the sequence $\frac{1}{(n\sigma_i)^2}(\log \|M_{np}\|-n\Lambda_i)$ converges in distribution to the Gaussian $\mathcal{N}(0,1)$ as $n \to \infty$. Thanks to the inequality 
\begin{equation}\label{eq.simple.norm.ineq}
|\log \|gh\|-\log\|g\|| \leq \max\{\log \|h\|,\log \|h^{-1}\|\}    
\end{equation}
we get that under $\P_x$, $\frac{1}{(n\sigma_i)^2}(\log \|gM_{np}\|-n\Lambda_i) \underset{n \to \infty}{\overset{\mathcal{L}}{\longrightarrow}} \mathcal{N}(0,1)$ for any fixed $g \in \GL_d(\R)$. Now using the fact that $\mathcal{L}_{\P_x}(M_{np+1})=\sum_{y \in E_{i+1}} P(x,y)\mathcal{L}_{\P_y}(X(y)M_{np})$, we easily deduce that $\Lambda_i=\Lambda_0$ and $\sigma_i=\sigma_0$ for every $i=0,\ldots,p-1$. Therefore, \eqref{eq.brown.hat1} together with the equality $\mathcal{L}_{\P_{(x_1,\ldots,x_p)}}(\widehat{M}^{i}_n)=\mathcal{L}_{\P_{x_p}}(M_{np})$ implies that for every $x \in E$, under $\P_x$
$$
\frac{1}{(n\sigma_{0}^2)^{1/2}} \left( \log \|M_{p\lfloor tn \rfloor}\|-nt\Lambda_0 + (nt-\lfloor nt \rfloor) (\log \|M_{p\lfloor tn \rfloor +p}\| - \log\|M_{p\lfloor tn \rfloor}\|) \right) \underset{n \to \infty}{\overset{\mathcal{L}}{\longrightarrow}} \mathcal{W}
$$
Once more using \eqref{eq.simple.norm.ineq} together with the fact that the state space is finite, one gets that for every $x \in E$, under $P_x$, $\|\widehat{S}_n-S_{np}\|_\infty \to 0$ in probability (as $n \to \infty$), where $S_{n}$ is defined in \eqref{eq.defn.sn2} with $\sigma:=\sigma_0/\sqrt{p}$ and $\Lambda:=\Lambda_0/p$. This implies (see e.g.~ \cite[Problem 4.16]{karatzas-shreve}) that for every $x \in E$, under $\P_x$
\begin{equation}\label{eq.brown.hat2}
S_{np} \underset{n \to \infty}{\overset{\mathcal{L}}{\longrightarrow}} \mathcal{W}.
\end{equation}
Once more using the inequality \eqref{eq.simple.norm.ineq} and the fact that $\sigma>0$, we observe that for every $k \in \N$, and $x \in E$, under $\P_x$, $\|S_{np}-S_{np+k}\| \to 0$ in probability and hence \eqref{eq.brown.hat2} implies that $S_{n} \underset{n \to \infty}{\overset{\mathcal{L}}{\longrightarrow}} \mathcal{W}$ as required.\\
The proofs of 2.~ and 3.~ are proven using the same ideas and are omitted to avoid repetition.
\end{proof}

\subsection{Markov measures for Gromov-hyperbolic groups}
We summarize here a construction of a Markov chain on the strongly Markov structure of a Gromov-hyperbolic group and its connection with the Patterson--Sullivan measure (both due to Calegari--Fujiwara \cite{calegari-fujiwara} in this setting). We also include some further related observations from Cantrell \cite{cantrell.georays}; other more specific ones will be included/proven in later sections where they are needed. 

\subsubsection{Patterson--Sullivan measures seen in the strongly Markov structure}\label{subsub.ps.vs.parry}

We keep the notation from \S \ref{sec.hyp}: let $\Gamma$ be a non-elementary Gromov-hyperbolic group endowed with a generating set $S$. Fix a strongly Markov structure $\mathcal{G}$. Let $\lambda>1$ be the exponential growth rate of $\Gamma$ with respect to $S$. Denote by $\nu$ the Patterson--Sullivan probability measure (see \cite[Definition 4.14]{calegari-fujiwara}) obtained as the limit of the sequence of probabilities $\nu_n$ on $\Gamma \cup \partial \Gamma$, where
\begin{equation}\label{eq.def.nun.hat}
\nu_n:=\frac{ \sum_{|g|_S \leq n} \lambda^{-|g|_S} \delta_g}{\sum_{|g|_S \leq n} \lambda^{-|g|_S}}.    
\end{equation}

 Let $Y=[\ast]$ be the (cylinder) set of sequences $(x_n)$ in $\Sigma_{A}$ that starts with the symbol $\ast$, i.e.~ $x_0=\ast$. Let $Y_n \subseteq Y \cap \Sigma_A^0$ be the subset of $Y$ consisting of sequences $(x_m)$ such that $x_m=0$ for every $m \geq n+1$. In view of Definition \ref{def.strong.markov}, the set $Y_n$ is in bijection with the ball of radius $n$ and hence the measures $\nu_n$ on $\Gamma$ defined in \eqref{eq.def.nun.hat} can be considered as measures on $Y_n$ --- we denote them by $\widehat{\nu}_n$. We note that this definition varies slightly from the one given in Section 4 of \cite{calegari-fujiwara}. Specifically our $\widehat{\nu}_n$ measures are normalised to be probability measures unlike in \cite{calegari-fujiwara}. Passing to the limit $\nu$ on $\Gamma \cup \partial \Gamma$, one gets a limiting measure $\widehat{\nu}=\lim_{n \to \infty} \widehat{\nu}_n$ supported on $Y$ and giving zero measure to each $Y_n$. The fact that the limit exists follows from a direct calculation: the $\widehat{\nu}_n$ measure of each cylinder set (which are open and closed sets the collection of which generates the algebra on $Y_\infty$) converges to a finite limit.  The fact that $\widehat{\nu}$ assigns zero measure to each $Y_n$ corresponds to the fact that $\nu$ on $\Gamma \cup \partial \Gamma$ is supported on the compact $\partial \Gamma$ (i.e.~ gives zero mass to $\Gamma$). Alternatively, denoting $Y_\infty:=Y \setminus (\cup_{n \geq 0} Y_n)$ it is easy to see that there is a Borel map $\Psi:  Y_\infty \to \partial \Gamma$ which takes an infinite path not ending with $0$'s to the equivalence class of the corresponding (infinite) geodesic ray in $\partial \Gamma$ and which pushes $\widehat{\nu}$ forward to $\nu$ (see \cite[\S 3.5]{calegari.notes} for a similar description and more details). Simple topological observations show that $\Psi$ is continuous,  surjective and finite-to-one (\cite[Lemma 3.5.1]{calegari.notes}). It follows that $\nu = \lim_{n\to\infty} \Psi_\ast \widehat{\nu}_n$ and so $\nu$ is obtained as the weak limit of the sequence $\nu_n$ defined in (\ref{eq.def.nun.hat}).

\subsubsection{Parry measure of the strongly Markov structure}\label{subsub.parry}
It is well-known since the work of Shannon \cite{shannon} and Parry \cite{parry} that given an irreducible subshift of finite type $\Sigma_B$ (i.e.~ a subshift associated to an irreducible matrix $B$ consisting of zero's and one's, see \S \ref{subsub.shifts}) there is a unique $\sigma$-invariant probability measure $\mu$ on $\Sigma_B$ for which the corresponding measure theoretic entropy $h_\mu(\sigma)$ is maximal among all $\sigma$-invariant (Borel) probability measures, i.e.~ $h_\mu(\sigma) = \sup_m h_m(\sigma)$ where the supremum is over all $\sigma$-invariant probability measures on $\Sigma_B$. Moreover this measure is a Markov measure in the sense that denoting by $E$ the alphabet of $\Sigma_B$, the matrix $B$ gives rise to a transition kernel $P$ and a probability $\mu_\bullet$ on $E$ such that $\mu_\bullet$ is $P$-stationary (see \S \ref{subsub.markov.chains}) and $\mu=\P_{\mu_\bullet}$. It is an ergodic probability measure (with respect to the shift transformation $\sigma$). We call this measure $\mu=\P_{\mu_\bullet}$ the measure of maximal entropy (also called the Parry measure of $\Sigma_B$). 

Even though the shift space $\Sigma_A$ associated to the strongly Markov structure  $\mathcal{G}$ of a Gromov-hyperbolic group $\Gamma$ (endowed with a generating set $S$) is not irreducible, one can run a Parry-like construction \cite[\S 4.2]{calegari-fujiwara} to obtain a shift-invariant Markovian probability measure $\mu$ on $\Sigma_A$ with the properties discussed below. We do not include the simple construction to avoid repetition. 

A key property of the Parry-like measure $\mu$ is that by \cite[Lemma 4.19]{calegari-fujiwara} (more precisely by \cite[Proposition 4.6]{cantrell.georays}), it is closely related to the measure $\widehat{\nu}$ on $Y \subseteq \Sigma_A$ constructed above using the Patterson-Sullivan measure $\nu$: we have
\begin{equation}\label{eq.parry.to.ps}
    \lim_{n \to \infty} \frac{1}{n} \sum_{k=1}^n \sigma^k_\ast \widehat{\nu}=\mu,
\end{equation}
where the convergence holds (in fact with a speed estimate) in total variation distance.

On the other hand, by \cite[Proposition 4.2]{cantrell.georays} the Parry-like measure $\mu$ is nothing but a linear combination of the Parry measures of maximal irreducible components of $\Sigma_A$: for each maximal component $(B_j)_{j=1,\ldots,m}$, there exists $\alpha_j>0$ such that $\sum_{j=1}^m \alpha_j=1$ and  
\begin{equation}\label{eq.parry.lin.comb}
    \mu=\sum_{j=1}^m \alpha_j \mu_{j},
\end{equation}
where $\mu_{j}$ is the Parry measure of the maximal component $B_j$ of $\mathcal{G}$.

As pointed out in \cite[\S 4.3]{calegari-fujiwara}, one can be more precise about the relation between $\widehat{\nu}$ and $\mu$ (than the mere relation \eqref{eq.parry.to.ps}). We record the following statement from \cite{cantrell.georays} which is an instance of this more precise relation and which will be useful later on. 

\begin{lemma}\cite[Lemma 4.5]{cantrell.georays} \label{lem:alpha}
For each $v \in V$ with $\mu[v]>0$ and $k \in \mathbb{Z}_{\ge 0}$ there exists $\alpha_v^k \ge0$ such that
$$\sigma_\ast^k\widehat{\nu}|_{[v]} = \alpha_v^k \mu|_{[v]}.$$
There exists a length $k$ path from $\ast$ to $v$ if and only if $\alpha_v^k >0$. \qed
\end{lemma}


\section{Law of large numbers for subadditive functions}\label{sec.lln}
In this section we prove Theorem \ref{thm:lln} which is
a general (weak) law of large numbers for subadditive functions on hyperbolic groups. A key ingredient will be Theorem \ref{thm:rayconvergence} which asserts a the existence of a common growth rate along almost every geodesic with respect to the Patterson--Sullivan measure. It can be seen as a strong law of large numbers for subadditive functions with respect to Patterson--Sullivan measure. Since different constructions of Patterson--Sullivan measures yield measures in the same measure class Theorem \ref{thm:rayconvergence} above does not depend on which construction we choose to work with.

We now deduce Theorem \ref{thm:lln} from Theorem \ref{thm:rayconvergence} which will be proven subsequently.
\begin{proof}[Proof of Theorem \ref{thm:lln}]
Let $\Lambda \in \R$ be the constant given by Theorem \ref{thm:rayconvergence}. For each $\epsilon >0$ define 
\[
A_\epsilon = \left\{ g \in \Gamma: \left| \frac{\varphi(g)}{|g|_S} - \Lambda \right| > \epsilon \right\}.
\]
We need to show that for each $\epsilon > 0$ the density of $A_\epsilon$ on $S_n$ vanishes as $n\to\infty$. Fix a Patterson--Sullivan measure, i.e.~ the one given by the limit of \eqref{eq.def.ps}. Note that for any fixed sufficiently large $R>0$, there exist positive constants $C_1$ and $C_2$ such that for every $n \in \N$, we have
\begin{equation}\label{eq.key.to.lln}
\frac{\#(S_n \cap A_\epsilon)}{\#S_n} \le C_1 \sum_{x \in S_n \cap A_\epsilon} \nu(O(x,R)) \le C_2 \  \nu\left(\bigcup_{x \in S_n \cap A_\epsilon} O(x,R) \right).
\end{equation}
The first inequality follows from \eqref{eq:shadows} and \eqref{eq.pure.growth} whilst the second follows from the fact that, due to hyperbolicity, $O(x,R)$ for $x \in S_n$ covers $\partial \Gamma$ up to uniformly bounded multiplicity. If $\xi$ belongs to $O(x,R)$ for some $n \in \N$ and $x \in S_n \cap A_\epsilon$, then, since the function $\varphi$ is Lipschitz with respect to $d_S$, it follows that there exists a constant $C>0$ (depending on the Lipschitz constant, $R$ and the hyperbolicity constant $\Delta$ but not on $n \in \N$) such that for any geodesic representation $(\xi)_{m=0}^\infty$ with $\xi_0 = o$ of $\xi$, we have $|\varphi(\xi_n)-\varphi(x)| \leq C$. In particular, for such $\xi \in \partial \Gamma$, we have
\[
\left|\frac{\varphi(\xi_n)}{n} - \Lambda\right|  \geq \epsilon - \frac{C}{n}.
\]
Therefore, we deduce by Theorem \ref{thm:rayconvergence} that
\[
\nu\left(\bigcup_{x \in S_n \cap A_\epsilon} O(x,R) \right) \underset{n \to \infty}{\longrightarrow} 0.
\]
Plugging this into \eqref{eq.key.to.lln}, the first conclusion of Theorem \ref{thm:lln} follows.

To deduce the second conclusion, note that for every positive $n \in \N$ and $g \in S_n$, by subadditivity of $\varphi$, we have $|\varphi(g)| \leq Dn$, where $D=\max_{g \in S}|\varphi(g)|$. 
We write
\begin{equation*}
\begin{aligned}
\left|\frac{1}{\#S_n} \sum_{g \in S_n} \frac{\varphi(g)}{n}-\Lambda\right| & \leq \left|\frac{1}{\#S_n} \sum_{g \in A_n} \left(\frac{\varphi(g)}{n} - \Lambda \right)\right| + \left|\frac{1}{\#S_n} \sum_{g \in S_n \setminus A_n} \left(\frac{\varphi(g)}{n}- \Lambda \right)\right|\\&
\leq \frac{\#(S\cap A_\epsilon)}{\#S_n}\epsilon +  \frac{\#(S_n \setminus A_\epsilon) }{\#S_n}(D+
\Lambda).
\end{aligned}
\end{equation*}
Since $\frac{\#(S_n \cap A_\epsilon)}{\#S_n} \to 1$ as $n \to \infty$ and $\epsilon>0$ is arbitrary, the second statement of Theorem \ref{thm:lln} follows.

It remains to show that $\Lambda$ is non-negative. To see this note that since $|g|_S = |g^{-1}|_S$ for every $g \in \Gamma$, we have
\begin{align*}
    2 \ \frac{1}{\#S_n} \sum_{g \in S_n} \frac{\varphi(g)}{n} = \frac{1}{\#S_n} \sum_{g \in S_n} \frac{\varphi(g) + \varphi(g^{-1})}{n}   \ge \frac{1}{\#S_n} \sum_{g \in S_n} \frac{\varphi(\id)}{n} = \frac{\varphi(\id)}{n},
\end{align*}
where $\id \in \Gamma$ is the identity element. The result follows by taking the limit as $n\to\infty$.
\end{proof}

\begin{remark}\label{rk.ray.conv.non.neg}
Notice that we proved that the constant $\Lambda$ appearing in Theorems \ref{thm:lln} and \ref{thm:rayconvergence} is the same. One deduces that the constant $\Lambda$ given by Theorem \ref{thm:rayconvergence} is non-negative.
\end{remark}

We now prove Theorem \ref{thm:rayconvergence}. To do so we follow the argument used Cantrell in \cite{cantrell.georays} employing additionally the subadditive ergodic theorem and properties of subadditive functions. The general tactic is to exploit the ergodicity of the Patterson--Sullivan measure to connect the behaviour of different maximal components for the strongly Markov structure: a key idea due to Calegari--Fujiwara \cite{calegari-fujiwara}.


\begin{proof} [Proof of Theorem \ref{thm:rayconvergence}]
Let $\mathcal{G}$ be a strongly Markov structure associated to the tuple $(\Gamma,S)$ and $\Sigma_A$ the shift space defined over symbols corresponding to the vertices of $\mathcal{G}$. Let $B$ be a maximal component (among $B_1,\ldots,B_m$, see \S \ref{subsub.shifts}) of $\mathcal{G}$ and let $(\Sigma_{B}, \sigma)$ denote the subshift defined over this component. Let $\mu_B$ denote the measure of maximal entropy on this subshift. 

For every $n \in \N$, let $f_n: \Sigma_{B} \to \R$ be the function given by $$f_n(x)=\varphi(\lambda(x_0,x_1) \ldots \lambda(x_{n-1},x_n)).$$ Since $\varphi: \Gamma \to \R$ is subadditive, the sequence $(f_n)_{n \in \N}$ constitutes a subadditive cocycle in the sense that $f_{n+m}(x) \leq f_n(x) + f_m(\sigma^n x)$. Since the Parry measure $\mu_B$ is ergodic, Kingman's subadditive ergodic theorem therefore yields that there exists $\Lambda \in \mathbb{R}$ such that
\begin{equation} \label{eq:aeconv}
\lim_{n\to\infty} \frac{f_n(x)}{n} = \Lambda \ \text{ for $\mu_B$-almost every $x \in \Sigma_B$.}
\end{equation}
Now consider the set
\[
E = \left\{ \xi \in \partial \Gamma: \lim_{n\to\infty} \frac{\varphi(\xi_n)}{n} = \Lambda \right\}.
\]
By Lemma \ref{lem:Lipschitz}, this set is well-defined and $\Gamma$-invariant. The conclusion of our proposition is equivalent to the fact that $\nu(E) = 1$, where $\nu$ denotes the Patterson--Sullivan probability measure \eqref{eq.def.ps}. Since the measure (class of) $\nu$ is ergodic with respect to the $\Gamma$-action \cite{coor}, it suffices to show that $\mu(E) > 0$.

Recall from \S \ref{subsub.ps.vs.parry} that we have a surjective, continuous map $\Psi: Y_\infty \to \partial \Gamma$ which pushes $\widehat{\nu}$ forward to $\nu$. Now fix an integer $k$ and vertex $v \in B$ such that there exists a length $k$ path in $\mathcal{G}$ from $\ast$ to a $v$. By combining \eqref{eq.parry.lin.comb} and Lemma \ref{lem:alpha}, one gets that there exists a constant $\alpha > 0$ such that $\sigma^k_\ast \widehat{\nu}|_{[v]} = \alpha \mu_B|_{[v]}.$ In particular, by (\ref{eq:aeconv}) we have that
\[
\sigma_\ast^k\widehat{\nu}|_{[v]} \left\{ x \in \Sigma_B: \lim_{n\to\infty} \frac{f_n(x)}{n} = \Lambda \right\} > 0 .
\]
Using the basic relation $|\phi(gh)-\phi(h)| \leq |\phi(g)|+|\phi(g^{-1})|$ valid for any subadditive function $\phi:\Gamma \to \R$ and every $g,h \in \Gamma$, one readily sees that the convergence $
f_n(x)/n \to \Lambda$ only depends on the tail of the sequence $x$: if $x$ and $y$ are two sequences such that $x_{n_1+k}=y_{n_2+k}$ for some $n_1,n_2 \in \N$ and all $k \in \N$, then the convergence holds either for both of $x,y$ or neither of them. 

This implies that the set 

\[
E_Y:=\left\{x \in Y: \lim_{n\to\infty} \frac{f_n(x)}{n} = \Lambda \right\} \ \ \text{satisfies} \ \ 
\widehat{\nu} (E_Y) > 0 .
\]
Since $\widehat{\nu}$ pushes forward to $\nu$ under $\Psi$ and $\Psi^{-1}(E)=E_Y$, we have that $\nu(E) = \Psi_\ast \widehat{\nu}(E) =\widehat{\nu}(E_Y) > 0$ and the conclusion follows.
\end{proof}

The following is a consequence of the above proof regarding positivity of the constant $\Lambda$ based on the Markov property (see e.g.~  \cite{countinglox, prohaska-sert} for similar uses of this idea in close contexts).

\begin{proposition}\label{prop.positivity.in.section}
Let $H$ be a group endowed with a semi-norm $|\cdot|$, $\rho:\Gamma \to H$ a morphism and $\varphi(\cdot)=|\rho(\cdot)|$ the associated subadditive function on $\Gamma$. Suppose that
every probability measure with finite exponential moment and with support that generates a finite index subgroup of $\Gamma$, has strictly positive $|\cdot|$-drift.
Then the constant $\Lambda$ given by Theorem \ref{thm:lln} is strictly positive.
\end{proposition}
The hypotheses of this proposition are a little awkward but are satisfied in many cases:\\
\textbullet ${}$ (Furstenberg \cite{furstenberg.non.commuting}) $H=\SL_d(\R)$ and the $\rho(\Gamma)$ is a strongly irreducible and non-relatively compact,\\
\textbullet ${}$ (Guivarc'h \cite{guivarch.positive}) the image of $\Gamma$ is non-amenable and has at most exponential $|\cdot|$-growth (see more precisely Kaimanovich--Kapovich--Schupp \cite[Proposition 2.5]{kaimanovich-kapovich-schupp}),\\
\textbullet ${}$ (Maher--Tiozzo \cite{maher-tiozzo}) the group $\Gamma$ acts non-elementarily on a  geodesic Gromov-hyperbolic space.\\[-8pt]

Since we have already used similar arguments involving the renewal measures in \S \ref{sec.markov} 
we will provide a brief proof of the above result.

\begin{proof}
It follows from \eqref{eq:aeconv} that $\Lambda$ is realized as the linear growth rate of $\varphi$ along almost-sure (with respect to $\mu_B$) Markovian trajectories $(\gamma_1,\gamma_2,\ldots)$ on $\Gamma^\N$. Now fix an edge $(v_0,v_1)$ in the maximal component $B$ and consider the induced (renewal) measure $\nu$ on $\Gamma$ obtained by return times the vertices $v_0$ and $v_1$ consecutively (see e.g.~ \cite[Definition 3.4]{prohaska-sert}). Then, by the Markov property, the induced law on $\Gamma^\N$ of this Markovian random walk along the return times to $(v_0,v_1)$ is the Bernoulli law $\nu^{\N}$ on $\Gamma^N$ (see e.g.~ \cite[Lemma 3.5]{prohaska-sert}). In particular, since the state space is finite, $\nu^{\N}$ is absolutely continuous with respect to $\mu_B$. Denoting by $\tau_0$ the expectation of return times (which has a finite exponential moment since the state space is finite and the chain is irreducible), the $|\cdot|$-drift $\Lambda_\nu$ of $\nu$ satisfies $\Lambda_\nu=\tau_0 \Lambda$. Finally, by the same argument that we will see in Lemma \ref{lem:(star)} (relying on \cite[Theorem 4.3]{gmm}), the support of $\nu$ generates a finite index subgroup of $\Gamma$ and hence $\Lambda_\nu>0$ by hypothesis. The result follows.
\end{proof}

We end this section by justifying 2.~ of Remark \ref{rk.subad.functions.intro}: it clear that if the conclusions of Theorems \ref{thm:lln} and \ref{thm:rayconvergence} hold for a function $\varphi$ and $\varphi': \Gamma \to \mathbb{R}$ is such that $|\varphi-\varphi'|$ is bounded, then they also hold for $\varphi'$. In particular, if $\varphi'$ is almost-subadditive in the sense that $\varphi'(gh) \le \varphi'(g) + \varphi'(h) + C$ for some $C>0$ and for all $g,h \in \Gamma$, then the conclusions of Theorems \ref{thm:lln} and \ref{thm:rayconvergence} hold for $\varphi'$: indeed, in this case the function $\varphi' + C$ is subadditive. Since a quasi-morphism (i.e.~ a function $f:\Gamma \to \R$ satisfying $|f(gh)-f(g)-f(h)| \leq C$ for some $C>0$ and for every $g,h \in \Gamma$) on a group $\Gamma$ is almost-subadditive, Theorems \ref{thm:lln} and \ref{thm:rayconvergence} apply. Notice that for a quasi-morphism, the limit $\Lambda$ is necessarily zero. 


\section{Limit theorems for the strongly Markov structure}\label{sec.5}

In \S \ref{subsub.construction.markov} we associate Markovian products to strongly Markov structure $\mathcal{G}$ of a couple $(\Gamma,S)$ and representation $\rho: \Gamma \to \GL_d(\R)$. We then deduce certain properties (1-contracting and irreducible) of the Markovian product from those of $\rho$. In \S \ref{subsec.lim.thm.markov.structure}, we single out the consequence of Theorem \ref{thm.irred.finite.markov} for these chains. In \S \ref{subsec.means.variances} we prove that the Markovian products associated to different maximal components of $\mathcal{G}$ satisfy limit theorems with the same parameters. Finally in \S \ref{subsec.gromov.markov}, we indicate the analogous results for isometries.

\subsection{Markovian products associated to the Markov structure of a Gromov-hyperbolic group
}\label{subsec.loop}


\subsubsection{Construction of Markovian products}\label{subsub.construction.markov}
We will extensively use the notation and terminology introduced in \S \ref{sec.hyp} and \S \ref{sec.markov}. We fix a Gromov-hyperbolic group $\Gamma$, a generating set $S \subset \Gamma$ and an associated strongly Markov structure $\mathcal{G}$. Denote by $(B_j)_{j=1,\ldots,m}$ the maximal components of $\mathcal{G}$. For each maximal component, let $E_j$ denote the set of edges between two vertices of $V_j$, the set of vertices in $B_j$.
The Parry construction discussed in \S \ref{subsub.parry} gives rise to an irreducible Markovian transition kernel $P_{v,j}$ on the state space $V_j$ and the Parry measure denoted $\mu_j$ is the unique shift invariant probability measure of maximal entropy on the associated trajectory space (subshift). We denote its restriction to $V_j$ by $\mu_{v,j}$. This is a $P_{v,j}$-stationary measure on $V_j$ and in accordance with our notation of trajectory measures we have $\mu_j=\P_{\mu_{v,j}}$\footnote{We stress this point to avoid any confusion. In the sequel, we will often use $\mu_j$ instead of $\P_{\mu_{v,j}}$ to simplify the notation.}. To define a Markovian product using the strongly Markov structure, we pass to the associated edge Markov chain: we consider the transition kernel $P_{e,j}$ on $E_j$ defined by $P_{e,j}((v_1,v_2),(v_3,v_4))=P_{v,j}(v_2,v_3) P_{v,j}(v_3,v_4)$. It is also irreducible and has the unique stationary measure $\mu_{e,j}$ given by $\mu_{e,j}(v_1,v_2)=\mu_{v,j}(v_1)P_{v,j}(v_1,v_2)$.

Having fixed a representation $\rho:\Gamma \to \GL_d(\R)$, we consider the map $X$ defined on the state space $E_j$ by the map $\lambda(\cdot,\cdot)$ in the strongly Markov structure (see Definition \ref{def.strong.markov}) and transpose of the representation $\rho$, i.e.~ $X((v_1,v_2)):={}^t\rho(\lambda(v_1,v_2))$. These define the data of our Markovian product that we will denote by $(M_n^j)$ for every $j=1,\ldots,m$. We denote the corresponding Lyapunov exponents by $\Lambda_1(j) \geq \Lambda_2(j) \geq \ldots \geq \Lambda_d(j)$ and will sometimes write $\Lambda(j)=\Lambda_1(j)$ to simplify notation. 

Finally, we write $(M_n)$ for the stationary Markovian product obtained from the Markov chain $(z_n)$ on the state space $\cup_{j=1}^m E_j$ and with transition kernel $P$ defined in the natural way from the $P_{e,j}$'s. Note that in general we will deal with the case $m>1$ so $(z_n)$ is not an ergodic Markov chain with any starting distribution that is a non-trivial linear combination of $\mu_{e,j}$'s.

\subsubsection{Proximality and strong irreducibility of Markovian products}

In the following lemma, we use the notation and constructions of the previous paragraph and show that  proximal and strongly irreducible representations give rise to proximal and strongly irreducible Markovian products. This relies on key ingredients from the works of Goldsheid--Margulis \cite{goldsheid-margulis} and  Gou\"ezel--Math\'{e}us--Maucourant \cite{gmm}.

\begin{lemma}\label{lem:(star)}
Suppose that $\rho:\Gamma \to \GL_d(\R)$ is a proximal and strongly irreducible representation. Then, for each $j=1,\ldots,m$, the Markovian product $(M^j_n)$ is $1$-contracting and strongly irreducible.
\end{lemma}

Recall that a semigroup $\Lambda<\GL_d(\R)$ is irreducible (resp.~ strongly irreducible) if there does not exist a non-trivial proper $\Lambda$-invariant subspace (resp.~ a finite collection of such subspaces whose union is $\Lambda$-invariant). It is not hard to see (see the proof of \cite[Th\'{e}or\`{e}me 5.3 (ii)]{bougerol.comparaison}) for every $j=1,\ldots,m$, $(M_n^j)$ is $1$-contracting or strongly irreducible, if and only if, there exists $x \in E_j$ such that the semigroup $T_x(x)<\GL_d(\R)$ (see \S \ref{subsec.simplicity}) is, respectively, proximal or strongly irreducible.

\begin{proof}
Fix $j \in \{1,\ldots,m\}$ and $x \in E_j$. It suffices to show that the semigroup $T_x(x)$ is proximal and strongly irreducible.

Let us first show that $T_x(x)$ is strongly irreducible. Let $v_1,v_2 \in V_j$ such that $x=(v_1,v_2)$ and denote by $p_j \in \N$ the period of the kernel $P_{e,j}$ on $E_j$. Then by definition  $T_x(x)={}^t\rho(\Gamma_x)$, where
$$
\Gamma_x:=\{\lambda(v_2,v_3)\ldots \lambda (v_{p_j n},v_1) \lambda(v_1,v_2) : n \in \N \; \text{and} \; (v_i,v_{i+1}) \in E_j \; \forall  i=1,\ldots,p_jn \}.
$$
We recall that in the preceding, $\lambda(\cdot,\cdot)$ denotes the labeling map in the definition of strongly Markov structure $\mathcal{G}$ (Definition \ref{def.strong.markov} (ii)). By the property (iii) in Definition \ref{def.strong.markov}, we have that the subset of $\Gamma_x$ consisting of elements of $\Gamma_x$ of $S$-length $p_j n$ is in bijection with the set $\mathcal{G}_{p_jn}(x)$ paths of vertices of length $p_j n$ in $\mathcal{G}$ that are loops around the vertex $v_1$. Since $v_1$ belongs to the maximal component $B_j$, there exists $c>0$ such that $\# \mathcal{G}_{p_jn}(x) \geq c \lambda^{p_j n}$, where $\lambda>1$ is the growth rate of the group $\Gamma$. We therefore have that $\# (\Gamma_x \cap S_{p_j n}) \geq c \lambda^{p_j n}$
and thanks to the purely exponential growth property \eqref{eq.pure.growth} of Gromov-hyperbolic groups, it follows that the upper asymptotic density of the semigroup $\Gamma_x$ over the spheres $S_{p_j n}$ is strictly positive, i.e.
$$
\limsup_{n \to \infty} \frac{\#(S_{np_j} \cap \Gamma_x)}{\#S_{np_j}}>0.
$$
As a consequence, by a result of Gou\"ezel--Math\'{e}us--Maucourant \cite[Theorem 4.3]{gmm}, we get that the subgroup $\Gamma_x^{\pm}$ generated by $\Gamma_x$ has finite index in $\Gamma$. Since $\rho(\Gamma)<\GL_d(\R)$ is strongly irreducible and $\Gamma_x^{\pm}<\Gamma$ is finite index, $\rho(\Gamma_x^{\pm})$ is also strongly irreducible. Since the transpose semigroup ${}^t \Lambda$ of a strongly irreducible semigroup $\Lambda<\GL_d(\R)$ is also strongly irreducible, it follows that the semigroup ${}^t\rho(\Gamma_x)=T_x(x)$ is strongly irreducible, as required.

It remains to show that ${}^t\rho(\Gamma_x)=T_x(x)<\GL_d(\R)$ is a proximal semigroup. It suffices to show again that the transpose semigroup $\rho(\Gamma_x)$ is proximal. By a result of Goldsheid--Margulis \cite{goldsheid-margulis} (for a version we use, see \cite[Lemma 6.23]{bq.book}), it suffices to show that the Zariski-closure $H:=\overline{\rho(\Gamma_x)}^Z<\GL_d(\R)$ is proximal. Recall that the Zariski-closure of a semi-group is a group and so $H=\overline{\rho(\Gamma_x^{\pm})}^Z$. But since $\Gamma_x^{\pm}$ has finite index in $\Gamma$, denoting by $G=\overline{\rho(\Gamma)}^Z$, we have the equality of connected components $G^o=H^o$. Since $\rho(\Gamma)$ is proximal by hypothesis so is $G^o$ and consequently $H$, completing the proof.
\end{proof}


\subsection{Limit theorems for the maximal components of the strongly Markov structure}\label{subsec.lim.thm.markov.structure}

We now put together the construction of Markovian products $(M_n^j)$ in \S \ref{subsub.construction.markov}, Lemma \ref{lem:(star)} and Theorem \ref{thm.irred.finite.markov} to deduce the following limit laws on the maximal components of the strongly Markov structure. For easier referencing, we state them separately.  
\begin{proposition}[Large deviations on maximal components]\label{prop.ld.maximal}
Under the assumptions of Theorem \ref{thm:ldtboundary}, for each $j=1,\ldots,m$ and for every $x \in E_j$,
\begin{equation*}
\limsup_{n \to \infty} \frac{1}{n}\log \P_x(|\log\|M_n^j\|-n \Lambda(j)|>n \varepsilon)<0,
\end{equation*}
where $\Lambda(j)$ is the top Lyapunov exponent of the Markovian product $(M_n^j)$.
\end{proposition}


\begin{proposition}[Convergence to the Wiener measure on maximal components]\label{prop.wiener.maximal}
Under the assumptions of Theorem \ref{thm:boundarywiener}, for every $j=1,\ldots,m$, $\sigma>0$, $\Lambda \in \R$, $t \in [0,1]$ and $n \in \N$, let $S_{n}^j(t)$ denote $C([0,1])$-valued random variable defined by
\begin{equation}\label{eq.affine.path.interpolation}
S_{n}^j(t)=\frac{1}{(n\sigma^2)^{1/2}} \left( \log \|M^j_{\lfloor tn \rfloor}\|-nt\Lambda + (nt-\lfloor nt \rfloor) (\log \|M^j_{\lfloor tn \rfloor +1}\| - \log\|M^j_{\lfloor tn \rfloor}\|) \right).
\end{equation}
Then, there exists $\sigma=\sigma_j>0$ such that for $\Lambda=\Lambda(j) \in \R$ and for every $x \in E_j$, under $\P_x$, the sequence $(S_{n}^j)_{n \in \N}$ of $C([0,1])$-valued random variables converges in distribution to $\mathcal{W}$.
\end{proposition}

Finally, we record the following.

\begin{proposition}[Law of iterated logarithm on maximal components]\label{prop.lil.maximal}
Keep the hypotheses and notation of Proposition \ref{prop.wiener.maximal} and let $\sigma_j>0$ and $\Lambda(j) \in \R$ be the constants given by that result. Then, for every $j=1,\ldots,m$, $x \in E_j$, for $\P_x$-a.e.~ $\omega$, the set of limit points of the sequence $\left(\frac{(S_{n}^j(t))(\omega)}{2\log \log n}\right)_{n \in \N}$ of elements of $C([0,1])$ is equal to the following compact subset of $C([0,1])$:
$$
\left\{f \in C([0,1]): f \; \text{is absolutely continuous}, f(0)=0, \int_{0}^1f'(t)^2dt \leq 1\right\}.
$$
\end{proposition}

\subsection{Comparing means and variances}\label{subsec.means.variances}
To upgrade Propositions \ref{prop.ld.maximal}, \ref{prop.wiener.maximal} and \ref{prop.lil.maximal} to the corresponding limit theorems on the full strongly Markov structure $\mathcal{G}$, we first need to show that the Lyapunov exponents $\Lambda_k(j)$ and variances $\sigma^2_j$ obtained in the previous section all agree (i.e.~ they do not depend on $j=1,\ldots,m$). This result is also needed to prove the positivity of the top Lyapunov exponent. These are the two goals of this paragraph. 

To compare these Lyapunov exponents and variances across maximal components we implement the approach of Calegari--Fujiwara \cite{calegari-fujiwara} (more precisely, its adaptation by Cantrell \cite[Proposition 4.8]{cantrell.TAMS}). The argument crucially relies on the ergodicity of the Patterson--Sullivan measure to compare typical growth rates of appropriately constructed functions along geodesic rays.

\begin{proposition} \label{prop:mandv}
1. There exists constants $\Lambda_1 \geq \Lambda_2 \geq \ldots \geq \Lambda_d$ such that for every $j=1,\ldots,m$ and $i=1,\ldots,d$ we have $\Lambda_i(j)=\Lambda_i$.\\[3pt]
2. There exists a constant $\sigma>0$ such that for every $j = 1, \ldots, m$, we have $\sigma_j^2=\sigma^2$.
\end{proposition}

In the sequel, whenever there is no risk of confusion, we will denote $\Lambda_1$ by $\Lambda$. 

\begin{proof}
1. The proof is similar to that of Theorem \ref{thm:rayconvergence} and hence omitted to avoid repetitive exposition of the same idea.

2. Fix a maximal component $B_j$ and define $S_j \subset \partial \Gamma$ to be Borel measurable subset of $\partial \Gamma$ consisting of boundary elements $\xi$ that have a geodesic representative $(\xi_n)$ such that for each $t \in \mathbb{R}$
\[
\lim_{n\to\infty} \limsup_{m\to\infty} \frac{1}{m} \#\left\{0 \le k\le m: \frac{\log\|\rho(\xi_k^{-1} \xi_{k+n})\| - \Lambda n}{\sqrt{n}} < t \right\}=\frac{1}{\sigma_j\sqrt{2\pi}}\int_{-\infty}^t e^{-\frac{s^2}{2\sigma_j^2}}ds.
\]
The set $S_j$ is well-defined and $\Gamma$-invariant. Given $t \in \mathbb{R}$ and $n \in \mathbb{Z}_{\ge 0}$, we also define the set
\[
F(t,n) = \left\{ x \in \Sigma_{B_j}: \frac{\log\| \rho(\lambda(x_0,x_1) \ldots \lambda(x_{n-1},x_n)) \| - \Lambda n }{\sqrt{n}} < t \right\}.
\]
Here, without loss of generality, we choose the norm $\|\cdot\|$ to be the operator norm induced by the Euclidean norm so that it is invariant under passing to the transpose of a matrix.
Moreover, for $z \in \Sigma_A$ and $m \in \mathbb{Z}_{\ge 0}$, we set
\[
\mu_{(z,m)} = \frac{1}{m} \sum_{k=0}^{m-1} \delta_{\sigma^k z}.
\]
Since the indicator functions $1_{F(t,n)}$ are continuous and $C(\Sigma_A)$ separable, using Birkhoff's ergodic theorem for the shift space $(\Sigma_{B_j},\mu_j)$ where $\mu_j$ is the (ergodic) Parry measure of $\Sigma_{B_j}$, we find a set $\Sigma_{B_j}' \subset \Sigma_{B_j}$ of full $\mu_j$ measure such that for every $z \in \Sigma_{B_j}'$, $t \in \mathbb{R}$ and $n \in \N$, we have 
\begin{equation} \label{eq:indicatorlim}
\lim_{m\to\infty} \int_{\Sigma_{B_j}} 1_{F(t,n)} \ d\mu_{(z,m)} = \mu_j(F(t,n)).
\end{equation}
Notice that, as in the first part of the proof, if $z \in \Sigma_{B_j}$ satisfies the convergence \eqref{eq:indicatorlim}, then any pre-image in $\sigma^{-k}(z)$ (for any $k \ge 1$) also satisfies the same convergence. Hence, thanks to \eqref{eq.parry.lin.comb} and Lemma \ref{lem:alpha}, we can find $k \ge 1$ and a subset $\Sigma_{B_j}^o = \sigma^{-k}(\Sigma_{B_j}') \subset \Sigma_A$ such that $\widehat{\nu}(\Sigma_{B_j}^o) > 0$ and for every $z \in \Sigma_{B_j}^o$, $t \in \mathbb{R}$ and $n \in \N$, we have 
\[
\lim_{m\to\infty} \int 1_{F(t,n)} \ d\mu_{(z,m)} = \mu_j(F(t,n)).
\]
Notice that by construction of the edge-chain in \S \ref{subsub.construction.markov} from the vertex chain, we have
\begin{equation}\label{eq.vertex.to.edge.variance}
\mu_j(F(t,n))=\P_{\mu_{e,j}}\left\{((x_0,x_1), \ldots) \in E_j : \frac{\log\| \rho(\lambda(x_0,x_1) \ldots \lambda(x_{n-1},x_n)) \| - \Lambda n }{\sqrt{n}} < t\right\}.
\end{equation}
Since the operator norm is invariant under transpose, by construction of the Markovian random product $(M_n^j)$, the right-hand-side of \eqref{eq.vertex.to.edge.variance} is equal to $\P_{\mu_{e,j}}\left(\frac{\log \|M_n\|-\Lambda n}{\sqrt{n}}<t\right)$. Using the central limit theorem implied by Proposition \ref{prop.wiener.maximal} (e.g.~ by specializing to $t=1$ in \eqref{eq.affine.path.interpolation} and using the definition of the Wiener measure) then implies that
\[
\lim_{n\to \infty} \lim_{m\to\infty} \int 1_{F(t,n)} \ d\mu_{(z,m)} = \lim_{n\to\infty} \mu_j(F(t,n)) = \frac{1}{\sigma_j\sqrt{2\pi}}\int_{-\infty}^t e^{-\frac{s^2}{2\sigma_j^2}} \ ds
\]
for $t \in \mathbb{R}$. Therefore we deduce that $\Psi(\Sigma_{B_j}^o) \subset S_j$, where $\Psi$ is the function $\Psi:Y_\infty  \to \partial \Gamma$ defined in \S \ref{subsub.ps.vs.parry}. In particular, we have $\nu(S_j) > 0$ and by the ergodicity of $\nu$ this implies that $\nu(S_j) = 1$ and so $\sigma^2_j$ does not depend on $j=1,\ldots,m$, as required.
\end{proof}

A direct consequence of the previous result is the following.

\begin{corollary}\label{corol.limthms.fullautomat}
Propositions \ref{prop.ld.maximal}, \ref{prop.wiener.maximal} and \ref{prop.lil.maximal} hold when $(M_n^j)$ is replaced by $(M_n)$ (see \S \ref{subsub.construction.markov}) and the constants $\Lambda(j)$ and $\sigma_j^2$ are replaced by $\Lambda$ and $\sigma^2$ for each $j=1,\ldots,m$. \qed
\end{corollary}

In Proposition \ref{prop:mandv} we  assumed that our representation $\rho$ is both strongly irreducible and proximal. However, the argument used to prove the first part of this proposition does not require either the strongly irreducible or proximal assumption. We obtain the following which applies to any representation of a Gromov-hyperbolic group into $\GL_d(\R)$.

\begin{lemma}\label{lemma.ps.lyapunovs} Let $\rho: \Gamma \to \GL_d(\R)$ be a representation of a hyperbolic group $\Gamma$ (which is equipped with a generating set). Then the Borel subset $B \subseteq \partial \Gamma$ consisting of elements $\xi$ having a geodesic representative $(\xi_n)$ satisfying 
$$\frac{1}{n}\log\left\|\bigwedge^k \rho(\xi_n)\right\| \underset{n \to \infty}{\longrightarrow} \sum_{i=1}^k \Lambda_i$$
for every $k=1,\ldots,d$, is well-defined, $\Gamma$-invariant and has full $\nu$-mass. \qed
\end{lemma}




We can now characterise the positivity of $\Lambda$ for strongly irreducible representations as claimed in the introduction.

\begin{proof}[Proof of Proposition \ref{prop.positivity}]
To prove the necessity, note that if the image of $\Gamma$ in $\PGL_d(\R)$ is relatively compact, then we can modify the norm $\|\cdot\|$ so that the map $\Gamma \ni \gamma \to \log\|\rho(\gamma)\| \in \R$ is additive. Now using Lemma \ref{lemma.ps.lyapunovs} and Remark \ref{rk.ray.conv.non.neg}, we realize $\Lambda$ as a counting average (as in Theorem \ref{thm:lln}) with respect to an additive function. The symmetry of $S$ readily implies that this counting average is zero.

Let us now show the remaining implication.
Suppose that the image of $\Gamma$ is not-relatively compact in $\PGL_d(\R)$.  Since $\rho(\Gamma)<\GL_d(\R)$ is (strongly) irreducible, it follows that the semigroup $\rho(\Gamma)$ is $r$-proximal for some $r \in \{1,\ldots,d-1\}$ (this is standard, see e.g.~ \cite[Lemma 3.6]{morris-sert}). It then follows by the same argument in Lemma \ref{lem:(star)} that $(M_n^j)$ is $r$-contracting for each $j=1,\ldots,m$. Therefore, since by Proposition \ref{prop:mandv}, the $\Lambda_i$'s are the Lyapunov exponents of the Markovian product $(M_n^1)$ (which satisfies the assumptions of Theorem \ref{thm.bougerol.simplicity})  there exists $r \in \{1,\ldots,d-1\}$ such that $\Lambda_r>\Lambda_{r+1}$. We now relate these Lyapunov exponents with spherical averages to exploit symmetry of the generating set $S$ to get positivity. To this end, we apply Theorem \ref{thm:rayconvergence} for the subadditive functions $\phi_1(\cdot)= \log \|{}^t\rho(\cdot)\|$ and $\phi_{\det}(\cdot)= \log \det(\rho(\cdot))$ and denote the corresponding averages by $\widetilde{\Lambda}_1$ and $\widetilde{\Lambda}_{\det}$, respectively. In view of Lemma \ref{lemma.ps.lyapunovs}, we have $\widetilde{\Lambda}_1=\Lambda_1$ and $\widetilde{\Lambda}_{\det}=\sum_{i=1}^d\Lambda_i$. Now, on the one hand by Remark \ref{rk.ray.conv.non.neg}, $\Lambda_1$ and $\sum_{i=1}^d \Lambda_i$ are non-negative, and on the other hand, we have $\Lambda_1 \geq \ldots \geq \Lambda_r >\Lambda_{r+1} \geq \ldots \geq \Lambda_d$. It follows that $\Lambda=\Lambda_1>0$. 
\end{proof}

\subsection{The case of isometries}\label{subsec.gromov.markov}

Here we briefly indicate how to associate a Markovian product (and the result corresponding to Proposition \ref{prop.ld.maximal} and Corollary \ref{corol.limthms.fullautomat}) in the analogous situation where, instead of a representation $\Gamma \to \GL_d(\R)$, we are given a non-elementary isometric action of $\Gamma$ on a Gromov-hyperbolic space $(H,d)$.

For each maximal component $B_j$ $j=1,\ldots,m$, the underlying Markov chain $(E_j,P_{e,j})$ described in \S \ref{subsub.construction.markov} remains the same. One only modifies the map $X$. We define $X:E_j \to \Isom(H)$ by $X((v_1,v_2))=\lambda(v_1,v_2)^{-1}$. We similarly denote by $(M_n^j)$ the associated Markovian product on $\Isom(H)$ and $(M_n)$ the Markovian product induced by the Markov chain $(z_n)$ on the state space $\cup_{j=1}^m E_j$ and with transition kernel $P$ defined in the natural way from the $P_{e,j}$'s.

Fix $j=1,\ldots,m$ and $x \in E_j$. One checks exactly as in the same way as Lemma \ref{lem:(star)} that the semigroup $T_x(x)<\Isom(H)$ is non-elementary. This implies that the Markovian product $(M_n^j)$ in $\Isom(H)$ is non-elementary and has positive drift (see Remark \ref{rk.positive.gromov}). Moreover, being defined over an irreducible Markov chain with finite state space, it clearly satisfies Conditions $(A_1)$ and $(A_2')$. In view of Remark \ref{rk.markov.hyperbolic.case}, we deduce that there exists a constant $\ell_{\Lambda_j}>0$ such that for every $\epsilon>0$, there exists $\alpha>0$ and $C>0$ such that for every $x \in E_j$, $\xi \in \overline{H}^h$ and $n \in \N$, we have
\begin{equation}\label{eq.isometry.analogue}
\P_x(|\sigma(M_n^j,\xi)-n\ell_{\Lambda_j}| \geq n\epsilon)\leq Ce^{-\alpha n}.
\end{equation}
Specializing to $\xi=o$, the previous inequality shows by Borel--Cantelli (or directly by the subadditive ergodic theorem) that $\P_x$-a.s.~ $\frac{1}{n} d(M_n^j \cdot o,o) \to \ell_{\Lambda_j}$. Now, the proof of Proposition \ref{prop:mandv} goes through and shows that on one hand the constant $\ell_{\Lambda_j}$ does not depend on the maximal component $B_j$ for $j=1,\ldots,m$ (and hence we denote this constant by $\ell_\Lambda$), and on the other hand, it coincides with the constant $\Lambda$ given by Theorem \ref{thm:rayconvergence} applied with the subadditive function $\phi(\gamma)=d(\gamma \cdot o,o)$ (as Lemma \ref{lemma.ps.lyapunovs}). The former fact together with \eqref{eq.isometry.analogue} gives the following analogue of Corollary \ref{corol.limthms.fullautomat} (the result corresponding to large deviations, i.e.~ Proposition \ref{prop.ld.maximal}):

\begin{proposition}\label{prop.gromov.ld.on.cannon} There exists a constant $\ell_\Lambda>0$ such that for every $\epsilon>0$ and $x \in \cup_{j=1}^m E_j$, we have
$$
\limsup_{n \to \infty} \frac{1}{n} \log \P_x\left(\left|\frac{1}{n}d(M_n \cdot o,o)-\ell_\Lambda\right|>\epsilon\right) <0.
$$ \qed
\end{proposition}

\section{Large deviation theorems} \label{section:ldt}

In \S \ref{subsec.ld.rep} we first prove our counting large devation theorem (Theorem \ref{thm:ldt}) assuming Theorem \ref{thm:ldtboundary}. We subsequently prove our boundary large deviation result (Theorem \ref{thm:ldtboundary}) with respect to the Patterson--Sullivan measure, obtaining a quantitative version of our boundary strong law of large numbers (Theorem \ref{thm:rayconvergence}) in the current setting. 
In \S \ref{subsec.ld.hyp}, we indicate the proofs of the analogous results in the case of isometries of Gromov-hyperbolic spaces.

\subsection{Large deviations for representations}\label{subsec.ld.rep}



We now give the proof of Theorem \ref{thm:ldt} using Theorem \ref{thm:ldtboundary} which will be proven subsequently.

\begin{proof}[Proof of Theorem \ref{thm:ldt}]
Let $\Lambda$ be given by Theorem \ref{thm:ldtboundary} and let $\epsilon>0$ be fixed. Recall from \eqref{eq.key.to.lln} that for any  sufficiently large $R>0$, there exist positive constants $C_1$ and $C_2$ such that for every $n \in \N$,
\begin{equation*}
\frac{\#(S_n \cap A_\epsilon)}{\#S_n} \le C_1 \sum_{x \in S_n \cap A_\epsilon} \nu(O(x,R)) \le C_2 \  \nu\left(\bigcup_{x \in S_n \cap A_\epsilon} O(x,R) \right),
\end{equation*}
where 
\[
A_\epsilon = \left\{ g \in \Gamma: \left| \frac{\log\|\rho(g)\|}{|g|_S} - \Lambda \right| > \epsilon \right\}.
\]
As in the proof of Theorem \ref{thm:lln}, if $\xi \in \bigcup_{x \in S_n \cap A_\epsilon} O(x,R)$, then there is a representative geodesic ray $(\xi_m)$ with $\xi_0 = \id$ and $C>0$ such that for every $n \geq 1$
\begin{equation} \label{eq:geobound}
\left|\frac{\log\|\rho(\xi_n)\|}{n} - \Lambda\right|  \geq \epsilon - \frac{C}{n}.
\end{equation}
Using hyperbolicity, (the fact that geodesic rays with the same end point remain within bounded distance) by enlarging $C$ we can assume that (\ref{eq:geobound}) holds for all geodesic representatives of $\xi \in \bigcup_{x \in S_n \cap A_\epsilon} O(x,R).$
In particular for all sufficiently large $n$,
\begin{align*}
\frac{\#(S_n \cap A_\epsilon)}{\#S_n} &\le C_2 \  \nu\left(\bigcup_{x \in S_n \cap A_\epsilon} O(x,R) \right)\\
&\le C_2 \ \nu\left( \xi \in \partial \Gamma: \text{for all $\xi_m \to \xi$ with $\xi_0=\id$, }\left| \frac{\log\|\rho(\xi_n)\|}{n} - \Lambda \right| > \frac{\epsilon}{2} \right).
\end{align*}
The result now follows from Theorem \ref{thm:ldtboundary}.
\end{proof}

The rest of this section is devoted to the proof of Theorem \ref{thm:ldtboundary} which will make key use of Bougerol's Theorem \ref{thm.ld.matrix} in the form of Proposition \ref{prop.ld.maximal} (and Corollary \ref{corol.limthms.fullautomat}).

\begin{proof}[Proof of Theorem \ref{thm:ldtboundary}]
Let $\Lambda$ be the constant given by Theorem \ref{thm:rayconvergence} applied with the function $\phi(\gamma)=\log \|\rho(\gamma)\|$, where $\|\cdot\|$ is the operator norm induced by the Euclidean norm on $\R^d$ (in particular, it is invariant under the transpose). For any $\epsilon > 0$ and $n \ge 1$, we define the sets
\[
U_n(\epsilon) =  \left\{ \xi \in \partial \Gamma: \text{for all $\xi_m \to \xi$ with $\xi_0=\id$, }\left| \frac{\log\|\rho(\xi_n)\|}{n} - \Lambda \right| > \epsilon \right\}
\]
and
\[
E_n(\epsilon) = \left\{ (x_m)_{m=0}^\infty \in Y_\infty :  \left| \frac{\log\|\rho(\lambda(\ast,x_1) \ldots \lambda(x_{n-1},x_n))\|}{n} - \Lambda \right| > \epsilon \right\}.
\]
Note that $\Psi^{-1}(U_n(\epsilon)) \subseteq E_n(\epsilon)$ and consequently $\nu(U_n(\epsilon)) \le \widehat{\nu}(E_n(\epsilon))$. Therefore to prove Theorem \ref{thm:ldtboundary} it suffices to show that for every $\epsilon >0$, $\widehat{\nu}(E_n(\epsilon)) \to 0$ exponentially quickly as $n\to\infty$, which is what we shall prove in the sequel.\\

To proceed, for every integer $i \geq 1$, we define
\begin{equation}\label{eq.def.ai}
A_i=\left(\sigma^{-i}\left(\bigcup_{j=1}^m \Sigma_{B_j}\right) \backslash \bigcup_{k=0}^{i-1} \sigma^{-k}\left(\bigcup_{j=1}^m \Sigma_{B_j}\right)\right) \cap Y
\end{equation}
where, as before, $B_j$ for $j=1, \ldots , m$ denote the maximal components. Intuitively, each $A_i$ consists of elements in $Y$ that correspond to a path in $\mathcal{G}$ that starts at $\ast$, enters a maximal component exactly  on its $i$th step and then never leaves this component.
For each $n \in \N$, we let  $\widetilde{\nu}_n$ to be the measure on $Y$ given by the restriction of $\widehat{\nu}$ on $\bigcup_{i=1}^n A_i$, i.e.\ for every Borel set $R \subseteq Y$,
\begin{equation}\label{eq.def.nun}
\widetilde{\nu}_n(R) = \widehat{\nu}\left(R \cap \bigcup_{i=1}^n A_i\right).
\end{equation}
We then have the following.
\begin{lemma} [Lemma 4.8 \cite{cantrell.georays}] \label{lem:TV2}
There exists $0<\theta<1$ such that $ \left\| \widetilde{\nu}_n -  \widehat{\nu} \right\|_{TV} = O(\theta^n),$
as $n\to\infty$.
\end{lemma}

It follows from this lemma that for any $\epsilon' >0$ there exist constants $0< \theta = \theta(\epsilon') <1$ and $C_0>0$ such that
\begin{equation}\label{eq.theta'}
  \widehat{\nu}(E_n(\epsilon)) \leq  \widetilde{\nu}_{\epsilon'n}(E_n(\epsilon)) + C_0 \theta^n   \le \sum_{k=1}^{n\epsilon'} \widehat{\nu}(E_n(\epsilon) \cap A_k) + C_0 \theta^n
\end{equation}
for every $\epsilon>0$ and $n \in \N$. Here and throughout the rest of this section we write $\epsilon'n$ instead of $\lfloor \epsilon' n\rfloor$ to simplify notation. We now turn our attention to studying each $\widehat{\nu}(E_n(\epsilon) \cap A_k)$.
\begin{lemma}\label{lem:inequ3}
For every $\epsilon >0$, there exist positive constants $\epsilon'$ and $C_1$ such that for all $n \in \N$, $n \epsilon' \geq k \geq 1$, we have
\[
\widehat{\nu}(E_n(\epsilon) \cap A_k) \le C_1 \ \mu \left\{x \in \cup_j \Sigma_{B_j}: \left| \frac{\log\|\rho(\lambda(x_0,x_1) \ldots \lambda(x_{n-1},x_n))\| }{n} - \Lambda \right| > \frac{\epsilon}{2} \right\}.
\]
\end{lemma}

\begin{proof}
Given $\epsilon>0$, fix $\epsilon'>0$ so that $2 \epsilon ' \max_{s\in S} \{\log\|\rho(s)\|\} < \epsilon/2$.
Then, for each $n\epsilon' \geq k \geq 1$, we have that $E_n(\epsilon) \cap A_k$ is given by

\begin{equation}\label{eq.the.inclusion}
\begin{aligned}
       &\left\{ x \in Y_\infty: \left| \frac{\log\|\rho(\lambda(\ast,x_1) \ldots \lambda(x_{n-1},x_n))\|}{n} - \Lambda\right| > \epsilon \ , \ \sigma^kx \in \cup_j \Sigma_{B_j}, \sigma^{k-1}x \notin \cup_j \Sigma_{B_j}\right\} \\
    &\subseteq \left\{ x \in Y_\infty: \left| \frac{\log\|\rho(\lambda(x_k,x_{k+1}) \ldots \lambda(x_{k+n-1},x_{k+n}))\|}{n} - \Lambda\right| > \frac{\epsilon}{2} \ , \ \sigma^kx \in \cup_j \Sigma_{B_j}\right\}\\
    &= Y_\infty \cap \sigma^{-k} \left\{x \in \cup_j \Sigma_{B_j}: \left| \frac{\log\|\rho(\lambda(x_0,x_1) \ldots \lambda(x_{n-1},x_n))\| }{n} - \Lambda \right| > \frac{\epsilon}{2} \right\}
\end{aligned}
\end{equation}
The inclusion above follows from Lemma \ref{lem:Lipschitz} (due to the submultiplicativity of the operator norm) and the choice of $\epsilon'.$ Letting $V_{\max}$ denote the collection of vertices belonging to a maximal component, it follows that
\begin{align*}
    \widehat{\nu}(E_n(\epsilon) \cap A_k) &\le \sigma_\ast^k\widehat{\nu} \left\{x \in \cup_j \Sigma_{B_j}: \left| \frac{\log\|\rho(\lambda(x_0,x_1) \ldots \lambda(x_{n-1},x_n))\| }{n} - \Lambda \right| > \frac{\epsilon}{2} \right\}\\
    &\hspace{-5mm}= \sum_{v \in V_{\max}} \sigma_\ast^k \widehat{\nu}|_{[v]} \left\{x \in \cup_j \Sigma_{B_j}: \left| \frac{\log\|\rho(\lambda(x_0,x_1) \ldots \lambda(x_{n-1},x_n))\| }{n} - \Lambda \right| > \frac{\epsilon}{2} \right\}\\
    &\hspace{-5mm}= \sum_{v \in V_{\max}} \alpha_v^k \mu|_{[v]} \left\{x \in \cup_j \Sigma_{B_j}: \left| \frac{\log\|\rho(\lambda(x_0,x_1) \ldots \lambda(x_{n-1},x_n))\| }{n} - \Lambda \right| > \frac{\epsilon}{2} \right\}
\end{align*}
where $\alpha_v^k$ are the constants from Lemma \ref{lem:alpha}. We recall now (from the construction of the Parry measure $\mu$) that a vertex  $v$ belongs to $V_{\max}$ if and only if $\mu[v] > 0$. In particular, for $v \in V_{\max}$,
\[
\alpha_v^k = \frac{\widehat{\nu}(\sigma^{-k}[v])}{\mu([v])} \le \max_{v \in V_{\max}} \frac{1}{\mu([v])} < \infty 
\]
and so we deduce that there exists $C_1>0$ such that
\[
 \widehat{\nu}(E_n(\epsilon) \cap A_k) \le C_1 \ \mu \left\{x \in \cup_j \Sigma_{B_j}: \left| \frac{\log\|\rho(\lambda(x_0,x_1) \ldots \lambda(x_{n-1},x_n))\| }{n} - \Lambda \right| > \frac{\epsilon}{2} \right\}
\]
as required. 
\end{proof}

We now complete the proof of Theorem \ref{thm:ldtboundary}. Fix $\epsilon >0$ and let $\epsilon'>0$ be as in Lemma \ref{lem:inequ3}.
By \eqref{eq.theta'}, there exist constants $0<\theta<1$ and $C_0>0$ such that
\[ 
  \widehat{\nu}(E_n(\epsilon)) \le \sum_{k=1}^{n\epsilon'} \widehat{\nu}(E_n(\epsilon) \cap A_k) + C_0 \theta^n
\]
and so by Lemma \ref{lem:inequ3} there is $C_1>0$ such that
\begin{equation}\label{eq.last.before.ld}
\widehat{\nu}(E_n(\epsilon)) \le C_1 n\epsilon' \  \mu\left\{x \in \cup_j \Sigma_{B_j}: \left| \frac{\log\|\rho(\lambda(x_0,x_1) \ldots \lambda(x_{n-1},x_n))\| }{n} - \Lambda \right| > \frac{\epsilon}{2} \right\}  + C_0 \theta^n.
\end{equation}
Recall that by Lemma \ref{lemma.ps.lyapunovs} (and Corollary \ref{corol.limthms.fullautomat}), the constant $\Lambda$ is also the top Lyapunov exponent of the Markovian product $(M_n)$. We now apply Corollary \ref{corol.limthms.fullautomat} (statement corresponding to Proposition \ref{prop.ld.maximal}) which says precisely that the $\mu$-measure of the set in the first term of the right-hand-side of \eqref{eq.last.before.ld} decays exponentially fast in $n$, concluding the proof.
\end{proof}

\begin{remark}
It is also possible to prove Theorem \ref{thm:ldt} using an approximation argument in which one compares the Markov measures on $\mathcal{G}$ to the counting measures on $S_n$. This method, which would avoid proving Theorem \ref{thm:ldtboundary}, is used in Section \ref{section:clt} to prove our counting central limit theorem. We presented the above proof instead as we believe Theorem \ref{thm:ldtboundary} is interesting in its own right.
\end{remark}

\subsection{Large deviations for isometries}\label{subsec.ld.hyp}
This section is devoted to the proof of Theorem \ref{thm.gromov.counting.ld}. 
As in the proof of Theorem \ref{thm:ldt}, we deduce Theorem \ref{thm.gromov.counting.ld} from a boundary large deviation result:  Theorem \ref{thm.gromov.ray.ld}.

The proof of Theorem \ref{thm.gromov.counting.ld} (resp.~ Theorem \ref{thm.gromov.ray.ld}) follows a very a similar line as the proof of Theorem \ref{thm:ldt} (resp.~ Theorem \ref{thm:ldtboundary}). Therefore, for brevity, we will only point out the needed modifications in the proofs. Let us start with the boundary version.

\begin{theorem}\label{thm.gromov.ray.ld}
Let $\Gamma$ be a Gromov-hyperbolic group, $S$ a generating set of $\Gamma$ and $(H,d)$ a geodesic Gromov-hyperbolic space and $o \in H$ a basepoint. Suppose that $\Gamma$ acts on $H$ by isometries and that the action is non-elementary. Let $\nu$ be a Patterson--Sullivan measure on $\partial \Gamma$ for the $S$ word metric. Then there exists a constant $\Lambda>0$ such that for any $\epsilon >0$,
\[
\limsup_{n\to\infty} \frac{1}{n} \log \nu\left( \xi \in \partial \Gamma: \text{for all $\xi_m \to \xi$ with $\xi_0=\id$, }\left| \frac{d(g \cdot o,o)}{n} - \Lambda \right| > \epsilon \right) <0
\]
\end{theorem}

This result implies Theorem \ref{thm.gromov.counting.ld}. The proof of this implication is precisely as in the proof of Theorem \ref{thm:ldt}, one only needs to replace the occurrences of $\log \|\rho(\star)\|$ by $d(\star \cdot o,o)$.

For Theorem \ref{thm.gromov.ray.ld}, similarly, the proof of Theorem \ref{thm:ldtboundary} goes through until the point at the end where we applied Corollary \ref{corol.limthms.fullautomat}. One only has to replace this result by Proposition \ref{prop.gromov.ld.on.cannon}: the analogous Markovian limit law but for the isometric actions (instead of representations) that we now consider. This completes the proof.


\section{Wiener process and the law of the iterated logarithm on the boundary}

The goal of this section is to prove Theorem \ref{thm:boundarywiener}: convergence to the Wiener process and the functional law of iterated logarithm.

\bigskip

Before starting the proof, we recall the notion of tightness that will be used therein. For $t \in [0,1]$, let $E_t: C([0,1]) \to \R$ denote the map that evaluates a function at $t$. We say that a sequence of probability measures $\eta_n$ on $C([0,1])$ is tight if
\begin{enumerate}
    \item $\sup_{n \in \N}E_0{}_\ast \eta_n(\R \setminus [-\lambda,\lambda]) \to 0$ as $\lambda \to +\infty$; and,
    \item $\lim_{\delta \to 0} \sup_{n \in \N} \eta_n (\sup_{|t-s| \leq \delta}|X(t)-X(s)| \geq \epsilon) = 0 $ for every $\epsilon>0$, where $X$ denotes a random variable with distribution $\eta_n$ and $t,s$ range over $[0,1]$. 
\end{enumerate}

In the proof below, the distribution $\eta_n$ will correspond to the pushforward of the Patterson--Sullivan measure $\nu$ by the map $S_n$ defined in \eqref{eq.defn.snxi}.

\begin{proof}[Proof of Theorem \ref{thm:boundarywiener}]
1. To prove the first statement we need to show that the laws of the sequence $(S_n)$ is a tight family and also that finite dimensional distributions of this sequence converge to the finite dimensional distributions of the Wiener measure on $C([0,1])$ (see e.g.~ \cite[Theorem 4.15]{karatzas-shreve}). Without loss of generality, 
we can assume that the operator norm in the definition of $(S_n)$ is induced by the Euclidean norm.

Let us start by showing that the distributions of $(S_n)$ constitute a tight family of measures on $C([0,1])$. Notice that we only need to check the second condition in the definition of tightness above, since by construction $S_n\xi(0)=0$ for $\nu$-a.e. $\xi \in \partial \Gamma$.  Fix $\epsilon>0$. For every $\delta>0$, define 
$$U_n(\epsilon,\delta)=\left\{\xi \in \partial \Gamma : \sup_{|t-s|<\delta} |S_n\xi(t)-S_n\xi(s)|>\epsilon\right\}.$$
For $x \in \Sigma_A$, let $\widehat{S}_n x$ denote the element of $C([0,1])$ defined in the same way as in \eqref{eq.defn.snxi} where for $k \in \N$, $\rho(\xi_k)$ is replaced by $\rho(\lambda(x_0,x_1),\ldots,\lambda(x_{k-1},x_k))$. Let us also similarly define
\begin{equation}\label{eq.en.eps.delta}
E_n(\epsilon,\delta)=\left\{x \in Y_\infty :  \sup_{|t-s|<\delta} |\widehat{S}_n x(t)-\widehat{S}_nx (s)|>\epsilon\right\}.
\end{equation}

Since any two representatives $\xi_m$ and $\xi_m'$ of an element $\xi \in \partial \Gamma$ stay at bounded $S$-distance, it is easy to check that there exists $n_0=n_0(\epsilon) \in \N$ such that for every $n \geq n_0$ and $\delta>0$, we have $\Psi^{-1}(U_n(\epsilon,\delta)) \subseteq E_n(\epsilon/2,\delta)$. Consequently, for every $n \geq n_0$ and $\delta>0$, we have 
$\nu(U_n(\epsilon,\delta)) \leq \widehat{\nu}(E_n(\epsilon/2,\delta))$.    
Therefore, to show that the set of distributions of $S_n$ is tight, it suffices to prove that $\lim_{\delta \to 0}\sup_{n \in \N}\widehat{\nu}(E_n(\epsilon/2,\delta))=0$. 

We use the following strategy to complete the proof of tightness: we show that for large $n \in \N$, the distributions of $\widehat{S}_n$ (under $\widehat{\nu}$) are approximated by that of the Markovian products in Proposition \ref{prop.wiener.maximal} (or more generally Corollary \ref{corol.limthms.fullautomat}) which themselves constitute a tight family (since they converge to the Wiener measure) and for small $n \in \N$ we exploit the fact that jumps of $S_n(\xi)(t)$ are bounded (for $n$ bounded) since they are normalized matrix norms of bounded-length products of elements of the finite set $S$.

It follows from Lemma \ref{lem:TV2} that there exist constants $C_0>0$ and $\theta \in (0,1)$ such that for every $m \in \N$, we have
\begin{equation}\label{eq.theta''}
 \widehat{\nu}(E_n(\epsilon/2,\delta)) \leq  \widetilde{\nu}_{m}(E_n(\epsilon/2,\delta)) + C_0 \theta^m   = \sum_{k=1}^{m} \widehat{\nu}(E_n(\epsilon/2,\delta) \cap A_k) + C_0 \theta^m,
\end{equation}
where the measures $\widetilde{\nu}_m$ and sets $A_k$ are as defined in \eqref{eq.def.ai} and \eqref{eq.def.nun}. We will now require the following observation which is an analogue of Lemma \ref{lem:inequ3}.

\begin{lemma}\label{lemma.inequ4}
There exist constants $c>0$ and $C_1>0$ such that
for every $n \in \N$, $cn^{1/2} \geq k \geq 1$ and $\delta>0$, we have
\[
\widehat{\nu}(E_n(\epsilon/2,\delta) \cap A_k) \le C_1 \ \mu \left\{x \in \cup_j \Sigma_{B_j}:  \sup_{|t-s|<\delta} |\widehat{S}_nx(t)-\widehat{S}_nx (s)|>\epsilon/4 \right\}.
\]
\end{lemma}

\begin{proof}
Fix $c>0$ so that $2 c \max_{s\in S} \{\log\|\rho(s)\|\} < \epsilon/4$.
Then, for each $cn^{1/2} \geq k \geq 1$, the set $E_n(\epsilon/2,\delta) \cap A_k$ satisfies
\begin{align*}
    &  \left\{ x \in Y_\infty: \sup_{|t-s|<\delta} |\widehat{S}_nx(t)-\widehat{S}_nx (s)|>\epsilon/2 \ , \ \sigma^kx \in \cup_j \Sigma_{B_j}, \sigma^{k-1}x \notin \cup_j \Sigma_{B_j}\right\} \\
    &\subseteq \left\{ x \in Y_\infty: \sup_{|t-s|<\delta} |\widehat{S}_nx(t)-\widehat{S}_nx (s)|>\epsilon/4 \ , \ \sigma^kx \in \cup_j \Sigma_{B_j}\right\}\\
    &= Y_\infty \cap \sigma^{-k} \left\{x \in \cup_j \Sigma_{B_j}: \sup_{|t-s|<\delta} |\widehat{S}_nx(t)-\widehat{S}_nx (s)|>\epsilon/4 \right\}.
\end{align*}
The inclusion in the second line above follows from Lemma \ref{lem:Lipschitz} (due to the submultiplicativity of the operator norm) and the choice of $c$. From this point on, the proof follows the same lines as the proof of Lemma \ref{lem:inequ3}. We omit it to avoid repetition.
\end{proof}

From the previous lemma, we deduce the analogue of \eqref{eq.last.before.ld} which reads as follows: for every $n \in \N$, $\delta>0$, $cn^{1/2} \geq m \geq 1$, we have
\begin{equation}\label{eq.last.before.tight}
\widehat{\nu}(E_n(\epsilon/2,\delta)) \le C_1 m \  \mu\left\{x \in \cup_j \Sigma_{B_j}: \sup_{|t-s|<\delta} |\widehat{S}_nx(t)-\widehat{S}_nx (s)|>\epsilon/4 \right\}  + C_0 \theta^m.
\end{equation}

Let $\eta>0$ be arbitrary.  Fix $m \in \N$ large enough so that $C_0 \theta^m<\eta/2$. Now, by Corollary \ref{corol.limthms.fullautomat} (since the operator norm is invariant under the tranpose) the pushforward of $\mu$ by $\widehat{S}_n$ converges to the Wiener measure. These pushforwards are tight and hence we can choose $\delta_1>0$ small enough so that for every $n \geq 1$, the $\mu$-measure on the right-hand-side of \eqref{eq.last.before.tight} is less than $\frac{\eta}{2C_1 m}$ for every $n \geq (m/c)^2$. Now observe from the definition \eqref{eq.en.eps.delta} of $E_n(\epsilon/2,\delta)$ that for every $n \in \N$ such that $\epsilon/2 > \delta n^{1/2} (\frac{3M_0+\Lambda}{\sigma})$, we have $E_n(\epsilon/2,\delta)=\emptyset$, where $M_0=\max_{s \in S}\log \|\rho(s)\|$. Therefore, up to reducing $\delta_1>0$ to $\delta_0>0$ so that any $n \leq  (m/c)^2$ satisfies $\epsilon/2>\delta_0 n^{1/2}(\frac{3M_0+\Lambda}{\sigma})$, we get that for every $\delta \in (0,\delta_0)$, $n \in \N$, we have $\widehat{\nu}(E_n(\epsilon/2,\delta)) \leq \eta$, proving that the laws of $S_n$ constitute a tight family.

We now turn to proving that the finite dimensional distributions of $(S_n)$ converge to those of the Wiener measure. Fix $0=t_0 < t_1< t_2 < \ldots < t_d \leq 1$. Let $F_{n,t_1,\ldots,t_d}(x)$ for $x=(x_1, \ldots,x_d) \in \mathbb{R}^d$ denote the cumulative distribution function
\[
F_{n,t_1,\ldots, t_d}(x) = \nu\left( \xi \in \partial \Gamma: S_n\xi(t_1,\ldots,t_d) \in \prod_{i=1}^{d} (-\infty, x_i] \right),
\]
where 
\[
S_n \xi(t_1, \ldots, t_d) := \left(S_n\xi(t_1), S_n\xi(t_2) - S_n\xi(t_1),  \ldots, S_n\xi(t_d) - S_n\xi(t_{d})\right).
\]
We would like to prove that $F_{n,t_1,\ldots, t_d}(x)$ converges as $n\to\infty$ to the cumulative distribution function $F_{t_1,\ldots,t_d}(x)$ of the multidimensional normal distribution $N(0,\omega)$ with $d\times d$ diagonal covariance matrix $\omega$ with entries $\omega_{ii}=t_i - t_{i-1}$, \cite[\S 1]{billingsley}.
Recall that the Patterson--Sullivan measure $\nu$ on $\partial\Gamma$ is given by $\Psi_\ast \widehat{\nu}$, where $\Psi: Y_\infty \to \partial \Gamma$ is continuous, surjective and finite-to-one. Moreover, using the fact that any two geodesic ray representing $\xi$ stays at bounded distance depending only on the hyperbolicity constant, it follows that there exists a sequence $\eta_n$ converging to zero as $n \to \infty$ such that for every $\xi \in \partial \Gamma$, for any $y \in \Psi^{-1}(\xi)$, we have $\|S_n\xi(t_1,\ldots,t_d)-\widehat{S}_n y(t_1,\ldots,t_d)\|_\infty \leq \eta_n$. As a consequence, 
it suffices to show that for every $x \in \R^{d}$
\[
\widehat{F}_{n,t_1, \ldots, t_d}(x) = \widehat{\nu}\left( y \in Y: \widehat{S}_ny(t_1,\ldots,t_d) \in \prod_{i=1}^{d-1} (-\infty, x_i] \right)
\]
converges to $F_{t_1,\ldots,t_d}(x)$ as $n\to\infty$, where $\widehat{S}_ny(t_1, \ldots, t_d)$ is defined analogously to $S_n \xi(t_1, \ldots, t_d).$

We define $E_{n,t_1, \ldots, t_d}(x)= \left\{ y \in \cup_j \Sigma_{B_j} : \widehat{S}_ny(t_1,\ldots,t_d) \in \prod_{i=1}^{d} (-\infty, x_i] \right\} \subset \Sigma_A$.
Recall from \eqref{eq.parry.to.ps} that the Cesar\'{o} averages of $\widehat{\nu}$ under the shift map converges to the Parry-like measure $\mu$ in the total variation distance. 
It follows from \eqref{eq.parry.to.ps} and Corollary \ref{corol.limthms.fullautomat} (statement corresponding to Proposition \ref{prop.wiener.maximal}) that for every $x \in \R^d$
\begin{align*}
\lim_{n\to\infty} \frac{1}{n^{1/4}} \sum_{k=0}^{n^{1/4}} \sigma_\ast^k \widehat{\nu}(E_{n,t_1, \ldots, t_d}(x)) &= \lim_{n\to\infty} \mu\left( y \in \Sigma_A: S_ny(t_1,\ldots,t_d) \in \prod_{i=1}^{d} (-\infty, x_i] \right) \\
&= F_{t_1,\ldots,t_d}(x).
\end{align*}
Defining
\[
C_{n,t_1, \ldots, t_d}^{\pm}(x) = E_{n,t_1, \ldots, t_d}(x \pm Cn^{-1/4}(1,1, \ldots, 1))
\]
where $C>0$ is some positive constant, we see that
\begin{equation}\label{eq.Cstar}
\lim_{n\to\infty} \frac{1}{n^{1/4}} \sum_{k=0}^{n^{1/4}} \sigma_\ast^k \widehat{\nu}(C_{n,t_1, \ldots, t_d}^{\star}(x)) =F_{t_1,\ldots,t_d}(x),
\end{equation}
for each $x \in \R^d$ and $\star \in \{+,-\}$.
Similarly to \eqref{eq.the.inclusion} if $C>0$ is taken sufficiently large (depending only on $\max_{s\in S}\log\|\rho(s)\|$ and the variance $\sigma^2>0$), by inclusion of the corresponding sets, we have
\begin{equation}\label{eq.inclusion.again}
\sigma_\ast^k\widehat{\nu}(C^-_{n,t_1, \ldots, t_d}(x)) \le \widetilde{\nu}_k(E_{n,t_1, \ldots, t_d}(x)) \le \sigma_\ast^k \widehat{\nu}(C^+_{n,t_1, \ldots, t_d}(x)) 
\end{equation}
for all integers $n \geq 1$ and $n^{1/4} \geq k \geq 1$. We deduce from \eqref{eq.Cstar} and \eqref{eq.inclusion.again} that
\[
\frac{1}{n^{1/4}} \sum_{k=0}^{n^{1/4}} \widetilde{\nu}_k(E_{n,t_1, \ldots, t_d}(x)) = F_{t_1,\ldots,t_d}(x).
\]
Finally, by Lemma \ref{lem:TV2}, this implies that $\widehat{F}_{n,t_1, \ldots, t_d}(x)$  also converges to $F_{t_1,\ldots,t_d}(x)$ as $n\to\infty$. From our above discussion, this concludes the proof of 1.

\bigskip
2. We need to show that the set $U$ of $\xi \in \partial \Gamma$ such that the conclusion of 2.~ holds has full $\nu$-measure. To this end, let $E$ be the set of $y \in Y_\infty$ such that the conclusion holds when $S_n\xi$ is replaced by $\widehat{S}_ny$ and $B$ be the set of $x \in \cup_{j} \Sigma_{B_j}$ such that the same conclusion again holds with $\widehat{S}_nx$. Note that the set $U$ is well-defined since its defining property does not depend on the choice of the representing geodesic ray and all these sets are Borel measurable. Since, given $\xi \in U$, we have that any $\xi'$ with the property $\xi_m=\xi'_{m+k}$ for certain $k \in \Z$ and every $m \in \N$ large enough also belongs to $U$,  the set $U$ is $\Gamma$-invariant. By $\Gamma$-ergodicity of $\nu$, all we need to show is $\nu(U)>0$. As $E \subseteq \Psi^{-1}(U)$ and $\Psi_\ast \widehat{\nu}=\nu$, it suffices to show that $\widehat{\nu}(E)>0$. Let, as before, $V_{\max}$ denote the set of vertices belonging to a maximal component and $v \in V_{\max}$. Let $k \in \N$ be such that there exists a path of length $k$ from $\ast$ to $v$. By Lemma \ref{lem:alpha}, there exists $\alpha_v^k>0$ such that 
\begin{equation}\label{eq.alpha.in.lil}
\sigma_\ast^k\widehat{\nu}|_{[v]} = \alpha_v^k \mu|_{[v]}.
\end{equation}
Now thanks to Corollary \ref{corol.limthms.fullautomat} (statement corresponding to Proposition \ref{prop.lil.maximal}), the set $B$ has full $\mu$ measure. Therefore, by \eqref{eq.alpha.in.lil}, we have $\sigma^k_\ast\widehat{\nu}(B)=\widehat{\nu}(\sigma^{-k}(B))>0$. But since for any $k \in \N$, we have $\sigma^{-k}(B) \cap Y_\infty \subseteq E$, we obtain $\widehat{\nu}(E)>0$, as desired.
\end{proof}

As an immediate consequence, we record the following more classical results, namely the central limit theorem (CLT) and law of iterated logarithm (LIL). The latter one provides a refinement of Theorem \ref{thm:rayconvergence} in the current setting.

\begin{corollary}[Boundary CLT and LIL]\label{corol.boundary.CLT.LIL}
Let $\rho: \Gamma \to \GL_d(\mathbb{R})$ be a strongly irreducible proximal representation of a hyperbolic group $\Gamma$. Equip $\Gamma$ with a finite generating set $S$ and let $\nu$ be the Patterson--Sullivan measure defined in $(\ref{PSdef})$. Let $\Lambda, \sigma^2 >0$ be the mean and variance from Theorem \ref{thm:boundarywiener}. Then,\\[3pt]
1. for each $n \ge 1$ and $ x \in \mathbb{R}$, denoting
\[
 \mathcal{A}_n(x):=\left\{ \xi \in \partial \Gamma: \text{for any representative $\xi_m \to \xi$ with $\xi_0 = o$, } \frac{\log\|\rho(\xi_n)\| - \Lambda n}{\sqrt{n}} \le x \right\}
\]
we have 
\begin{equation}\label{eq.boundary.classical.clt}
\lim_{n \to \infty}\nu(\mathcal{A}_n(x)) = \frac{1}{\sqrt{2\pi} \sigma} \int_{-\infty}^x e^{-t^2/2\sigma^2} \ dt \ \text{; and,}
\end{equation}
2. for $\nu$ almost every $\xi \in \partial \Gamma$ and for any representative $\xi_m \to \xi$,
$$
\liminf_{n\to\infty} \frac{\log\|\rho(\xi_n)\|-n\Lambda}{\sqrt{2\sigma n \log \log n}}=-1 \ \text{ and } \ \limsup_{n\to\infty} \frac{\log\|\rho(\xi_n)\|-n\Lambda}{\sqrt{2\sigma n \log \log n}} = 1.
$$ \qed
\end{corollary}

\begin{remark}[Speed and uniformity in boundary CLT]
Using Theorem \ref{thm.bougerol.berryess}, it is possible to give a speed estimate 
in  \eqref{eq.boundary.classical.clt} that is uniform over $x \in \R$. However, we will not pursue this direction as this would be a (somewhat lengthy and technical) diversion from the main goals in the article (see \cite[\S 4 and \S 5]{cantrell.georays}).
\end{remark}


\section{Counting central limit theorem and error terms} \label{section:clt}

In this section, after briefly commenting on our approach, we prove Theorem \ref{thm:clt}.

\bigskip

Similar to the schemes we followed in the proofs of weak law of large numbers (\S \ref{sec.lln}) and large deviation results for counting (\S \ref{section:ldt}), one could try to obtain a corresponding counting CLT directly from 1.~ of Corollary \ref{corol.boundary.CLT.LIL}. However there are difficulties in implementing that approach for the CLT. The main issue stems from the fact that when we compare the asymptotic density of sets with the Patterson-Sullivan measure of certain boundary sets (e.g.~ as in the proof of Theorem \ref{thm:ldt}), we do so up to a bounded multiplicative constant. Such a constant is inconsequential when proving large deviation type results, however it destroys the precise limiting behaviour that we need for a CLT to hold. To overcome this issue (and, importantly, to prove a CLT with the Berry--Esseen type error term) we will directly compare the uniform counting measures on $S_n$ with the Markov measures on $\Sigma$. This method will make use of a quantified version of an argument from a recent work of Gehktman--Taylor--Tiozzo \cite{GTTCLT}.

More precisely, using the strongly Markov structure, up to a periodicity issue, we will consider a geodesic factorization of an element $g \in \Gamma$ chosen uniformly from the sphere of length $n$ as $g_0g_1g_2$ with $g_0$ and $g_2$ of size approximately $\log n$. It will then suffice to show a CLT with error term for the middle factor $g_1$. We will show (Lemma \ref{lem:approx2}) that the distribution of this middle factor $g_1$ is approximated (with speed) by the length $\sim(n-2\log n)$ path distribution of a Markov chain. We can then associate a Markovian random matrix product to this chain and bring back (Lemma \ref{lem:periodic.berry}) the relevant result (Theorem \ref{thm.bougerol.berryess}) from the Markovian matrix products to establish a counting CLT with error term (Proposition \ref{prop:cltvertex}). The proof is then completed by resolving the periodicity issue. 

\bigskip

We now start collecting the necessary ingredients for the proof of Theorem \ref{thm:clt}. We will heavily use the constructions from Section \ref{sec.hyp}. Fix a non-elementary Gromov-hyperbolic group $\Gamma$, a generating set $S \subset \Gamma$ and strongly Markov structure $\mathcal{G}$. Let $A$ be the transition matrix as introduced in \S \ref{subsub.shift.space.ofmarkov}.  Let $p$ be an integer that is divisible by the periods of each maximal component of $A$ so that the non-negative matrix $A^p$ has a unique (necessarily real) eigenvalue $\lambda^p$ of maximal modulus. To deduce Theorem \ref{thm:clt} we will first study the convergence of our counting distributions along the subsequence $np$.

For a positive $k \in \N$, we will define a Markov chain on the state space
$$\Omega_{kp}:=\{(w_0,\ldots,w_{kp}) : w_i \in \mathcal{G}, \; A_{w_i,w_{i+1}}=1\}$$
of length-$kp$ paths in the strongly Markov structure $\mathcal{G}$. To define a transition kernel and a stationary measure on $\Omega_{kp}$, we lset
\begin{equation}\label{eq.defn.p.u}
p_i = \lim_{n\to\infty} \frac{e_i^T A^{np} 1}{ \lambda^{np}} \ \ \text{ and } \ \ u_i = \lim_{n\to\infty} \frac{e_\ast^T A^{np} e_i}{\lambda^{np}}
\end{equation}
where $e_i$ and $e_\ast$ correspond to the vectors that have the entry $1$ in the index corresponding to the vertices $v_i$ and $\ast$ respectively, and $0$ elsewhere. 

\begin{remark} \label{remark:exp}
Before proceeding further, we remark that, by our choice of $p$, the limits above defining each $p_i$ and $u_i$ converge exponentially quickly. This is because the matrix $A^p$ exhibits a spectral gap from its leading (real positive) eigenvalue to the rest of the spectrum.
\end{remark}
Now let $\pi_{kp}$ be the measure on $\Omega_{kp}$, defined by $$\pi_{kp}(w_0,\ldots,w_{kp})=\frac{u_{w_{0}} p_{w_{kp}}}{\lambda^{kp} p_\ast}.$$
It is readily checked that $\pi_{kp}$ defines a probability measure. Let $N_{kp}$ be the transition kernel defined by
$$
N_{kp}((w_0,\ldots,w_{kp}),(w'_0,\ldots,w'_{kp}))=\left\{
\begin{array}{ll}
\frac{p_{w'_{kp}}}{\lambda^{kp} p_{w'_0}} & \text{if} \ \; p_{w_0'}>0 \ \text{ and } w_{kp} = w_0'\\
0 & \, \textrm{otherwise.} \\
\end{array}
\right. 
$$
Unfolding the definitions, one also readily checks that $N_{kp}$ is a stochastic matrix and $\pi_{kp}$ is $N_{kp}$-stationary (i.e.~ a left eigenvector with eigenvalue one). Let $\widetilde{\Omega}_{kp} \subseteq \Omega_{kp}^\N$ be the subshift associated to this Markov chain and $\widetilde{\P}_{kp}$ be the associated Markovian measure on $\widetilde{\Omega}_{kp}$. 
Finally, for $k \geq 1$ and $z_1,\ldots,z_k \in \Omega_p$, let $[z_1,\ldots,z_k]$ be the associated cylinder set in $\widetilde{\Omega}_p$ and  $(z_1,\ldots,z_k)$ be the corresponding element of $\Omega_{kp}$. Observe that by an easy calculation using the definitions of $\pi_{kp}$'s and $N_{p}$, we have
\begin{equation}\label{eq.pi.as.iteration}
\pi_{kp}((z_1,\ldots,z_k))=\widetilde{\P}_{p}([z_1,\ldots,z_k]).
\end{equation}

\bigskip

The non-negative matrix $A^p$ is not necessarily irreducible and hence we decompose it into connected components (as we did to obtain $A''$ from $A$ in \S \ref{subsub.shift.space.ofmarkov}). Some of these components will have spectrum with simple eigenvalue $\lambda^p$. We label these finitely many $A^p$ maximal components $C_1, \ldots, C_{m_0}$. 
Note that each of the vertex sets for $C_1, \ldots, C_{m_0}$ are subsets of the vertex sets of the maximal components of $A$. Notice from definitions of the constants $p_i$ and $u_i$'s in \eqref{eq.defn.p.u} and that of the stationary measure $\pi_{kp}$ that $\pi_{kp}(w_0,\ldots,w_{kp})>0$ if any only of $w_0$ and $w_{kp}$ belong to the same maximal component of $A^p$. Moreover, the transition kernel $N_{kp}$ sends a path $(w_0,\ldots,w_{kp})$ in a maximal component $C_i$ to a path in $C_i$. Therefore, the Markov chain defined above is not ergodic if $m_0 \geq 2$. Its ergodic components are simply given by the maximal components $C_i$ for $i=1,\ldots,m_0$: the restriction of the transition kernel $N_{kp}$ to the set $\Omega^i_{kp}$ paths of length $kp$ with initial and end vertex belonging to a single $C_i$ (together with the normalized restriction of $\pi_{kp}$ to $\Omega^i_{kp}$) gives an ergodic Markov chain. Moreover, by the choice of $p$ (a common multiple of the periods of maximal component of $A$), these Markov chains are aperiodic. 

We now proceed precisely as in \S \ref{sec.5} to deduce a CLT with Berry--Esseen bounds along periodic products from $\Sigma_A$. Since the procedure is the same, we only outline the steps:
\begin{enumerate}
    \item As in \S \ref{subsub.construction.markov} We associate a Markovian random matrix product $(M_n^i)$ to the aperiodic finite state Markov chains on $\Omega^i_{p}$.
    \item As in Lemma \ref{lem:(star)}, we check the $1$-contracting and strong irreducibility assumptions for these Markovian products.
    \item By applying Theorem \ref{thm.bougerol.berryess}, we deduce a CLT with mean $\Lambda_i$ and variance $\sigma_i^2>0$ and with Berry--Esseen error term of order $O(\frac{\log n}{\sqrt{n}})$.
    \item We check exactly as in Proposition \ref{prop:mandv} that the means $\Lambda_i$ and variances $\sigma_i^2$ do not depend on $i=1,\ldots,m_0$. Set $\Lambda=\Lambda_i$ and $\sigma^2=\sigma^2_i$.
\end{enumerate}
From these, analogous to Corollary \ref{corol.limthms.fullautomat}, we deduce the following.

\begin{lemma}\label{lem:periodic.berry}
There exists a constant $D>0$ such that for every $n \geq 1$ and $t \in \R$
\[
\left| \widetilde{\mathbb{P}}_p\left( (z_1,\ldots ) \in \widetilde{\Omega}_p : \frac{\log\|\rho(\lambda(z_1)) \ldots \rho(\lambda(z_n))\| - np \Lambda}{\sqrt{np}} \le t\right)- \frac{1}{\sigma \sqrt{2\pi}} \int_{-\infty}^t e^{-\frac{s^2}{2\sigma^2}} ds \right|
\]
is bounded above by $\frac{D \log n}{\sqrt{n}}$ where for $z=(w_0,\ldots,w_p)$ belonging to $\Omega_p$, we write $\lambda(z)=\lambda(w_0,w_1) \ldots \lambda(w_{p-1},w_p)$.
\end{lemma}

\begin{remark}\label{rk.Pp.to.pi}
Keeping the notation of the previous lemma, notice that in view of \eqref{eq.pi.as.iteration}, the first term in the previous lemma is equal to 
$$
\pi_{np}\left\{(z_1,\ldots,z_n):\frac{\log\|\rho(\lambda(z_1)) \ldots \rho(\lambda(z_n))\| - np \Lambda}{\sqrt{np}} \le t \right\}
$$
\end{remark}

Having obtained Lemma \ref{lem:periodic.berry}, to prove Theorem \ref{thm:clt}, we now follow the ideas used in Sections 6-7 of \cite{GTTCLT}. However we need to quantify various rates of convergence to obtain the error term in Theorem \ref{thm:clt}. 

\bigskip

We start by defining a probability measure $\mu_q$ on each sphere $S_q$ that will help us deal with the periodicity issue at the end. Fix an integer $0 \le r \le p-1$. We define a measure on the set $S_r$ (or equivalently the set of paths of length $r$ in $\mathcal{G}$ starting from the vertex $\ast$) in the following way. For $g \in S_r$ we set
\[
\mu_r (g) = \frac{e_i^T A_\infty 1}{e_\ast^T A^r A_\infty 1} = \lim_{n\to\infty} \frac{e_i^T A^{np} 1 }{e_\ast^T A^r A^{np} 1}
\]
where $A_\infty = \lim_{n\to\infty} A^{np}/\lambda^{np}$ and $i$ is the end vertex of the path in $\mathcal{G}$ starting at $v_\ast$ corresponding to $g$. Here the limit defining $A_\infty$ exists since by choice of $p \in \N$ so that $\lambda^p$ is the unique eigenvalue of maximal modulus of $A^p$. For the same reason, the limit defining $\mu_r(g)$ converges exponentially quickly. One easily checks that $\sum_{|g|=r} \mu_r(g) = 1$. We extend the definition of $\mu_q$ on $S_q$ for $q \geq p$ as follows: given integers $n \geq 1$ and $q=np+r$ with $0 \leq r \leq p-1$, we define $\mu_{np + r}$ on $S_{np+r}$ (equivalently, on the set of paths of length $np+r$ in $\mathcal{G}$ starting from the vertex $\ast$) as follows: given $g \in S_{np+r}$, let $(\ast,w_1,\ldots,w_{np+r})$ be the unique path in $\mathcal{G}$ such that $g=\lambda(\ast, w_1) \ldots \lambda(w_{np+r-1},w_{np+r})$. Set $h=\lambda(\ast, w_1) \ldots \lambda(w_{r-1},w_r)$ and let $\mu_{np+r}(g)=\mu_r(h)\frac{1}{e_{w_r}^T A^{np}1}$. Note that the denominator in the last expressions is the number of length $np$-paths starting at the vertex $w_r$. 
Let $\tau_n$ denote the uniform probability measure on the sphere $S_n$. We have the following

\begin{lemma} \label{lem:TV}
For each $r=0, \ldots, p-1$ we have that $\|\tau_{np+r} - \mu_{np+r}\|_{TV} = O(\theta^n)$ for some $0< \theta < 1$ as $n \to\infty$.
\end{lemma}

\begin{proof}
Take a set $R \subset \Gamma$, let $R_{np+r} = R \cap S_{np+r}$ and write $R^+_{np+r} = R_{np+r} \cap \bigcup_{\substack{ g \in S_r^+} }[g]$ where $[g]$ denotes all group elements that have corresponding path in $\mathcal{G}$ that start with $g$ and $S_r^+ = S_r \cap \{ g \in S_r: \mu_r(g) >0\}$. From the definition of $\mu_{np+r}$ we see that if $\mu_{np+r}(R) = 0$ then $R^+_{np+r} = \emptyset$ and $\tau_{np+r}(R)$ decays to $0$ exponentially quickly, independently of $R$.
Otherwise, $\mu_{np+r}(R) \neq 0$ and denoting by $v_g$ the last vertex in $\mathcal{G}$ of the path from $\ast$ corresponding to $g$, we have
\begin{align*}
    \mu_{np+r}(R) &= \sum_{g \in S_r^+}\frac{ \mu_r(g) \#(R_{np+r} \cap [g])}{e_{v_g}^T A^{np} 1 }\\
    &= \frac{1}{e_\ast^T A^{np+r}1 }\left(\sum_{g \in S_r^+}\frac{ \mu_r(g) \#(R_{np+r} \cap [g])}{e_{v_g}^T A^{np} 1/e_\ast^T A^{np+r}1 }\right) \\
    &= \left( \frac{1}{e_\ast^T A^{np+r} 1} \ \sum_{g \in S_r^+}  \#(R_{np+r} \cap [g])\right) + O(\theta^n)\\
    &=\tau_{np+r}(R^+_{np+r}) + O(\theta^n)
\end{align*}
for some $0<\theta<1$ independent of $R$. In the penultimate line we used the fact that $e_{v_g}^T A^{np} 1/e_\ast^T A^{np +r} 1$ converges to $\mu_r(g)$ exponentially quickly as $n \to \infty$. To conclude the proof we note that, from the construction of $\mu_r$,  $|\tau_{np+r}(R) - \tau_{np+r}(R^+_{np+r})|$ converges to $0$ exponentially quickly and that this rate of convergence is independent of $R$.
\end{proof}

We now, following \cite{GTTCLT}, define probability measures that will determine the law of the middle factor $g_1$ of a $n$-long product $g$ written in geodesic factorization $g_0g_1g_2$ where $g_0$ and $g_2$ are of logarithmic length. Consequently we show that these measures can be approximated by the path distribution of a Markov chain. Let $c>0$ be a positive constant. For a path $\gamma$ in $\mathcal{G}$ of length $np - 2p\lfloor c \log n \rfloor$ 
starting at $v_i$ ending at $v_j$, we set
\[
 \widetilde{\tau}_{np}^c(\gamma) =   \frac{e_\ast^T A^{p \lfloor c \log n \rfloor} e_i \ e_j^T A^{p \lfloor c \log n \rfloor}1}{e_\ast^T A^{np} 1}.
 \]
Intuitively $ \widetilde{\tau}_{np}^c$ assigns a path $\gamma$ probability $s$ if the proportion of length $np$ paths starting at $\ast$ that have $\gamma$ as a sub-path from the $p\lfloor c \log n\rfloor$ to the $np - p\lfloor c \log n\rfloor$ vertex is $s$. 

\begin{lemma} \label{lem:approx2}
For every fixed $c > 0$ sufficiently large, we have
\[
\| \pi_{pn - 2p\lfloor c\log n \rfloor} - \widetilde{\tau}_{np}^c \|_{TV} = O\left(\frac{1}{\sqrt{n}} \right)
\]
as $n\to\infty$.
\end{lemma}

\begin{proof}
By Remark \ref{remark:exp}, there exists $\delta >0$ such that for every vertex $v_i$ in $\mathcal{G}$ we have
\[
p_i = \frac{e_i^T A^{np} 1}{\lambda^{np}} + O\left( \lambda^{-\delta n}\right) \ \text{ and } \ u_i = \frac{e_\ast^T A^{np}e_i}{\lambda^{np}} + O\left( \lambda^{-\delta n}\right)
\]
as $n \to \infty$. It follows that for any $c > (2\delta \log \lambda )^{-1}$, 
we have
\begin{equation}\label{eq.quad.est}
p_i = \frac{e_i^T A^{\lfloor c \log n \rfloor p} 1}{\lambda^{\lfloor c \log n \rfloor p}} + O\left( n^{-1/2}\right) \ \text{ and } \ u_i = \frac{e_\ast^T A^{\lfloor c \log n \rfloor p}e_i}{\lambda^{\lfloor c \log n \rfloor p}} + O\left(n^{-1/2}\right)
\end{equation}
as $n\to\infty$.
Fix such a constant $c>0$. Let $v_i$ and $v_j$ be two vertices in $\mathcal{G}$ that belong to the same maximal component of $A^p$ and such that  $u_i>0$. Let $\gamma$ be a path of length $n' = np - 2p\lfloor c \log n \rfloor$ from $v_i$ to $v_j$. Then,
\begin{align*}
\frac{\pi_{n'}(\gamma)}{\widetilde{\tau}_{np}^c(\gamma)} &= \frac{u_i p_j}{\lambda^{n'} p_\ast}  \frac{e_\ast^T A^{np} 1}{e_\ast^T A^{p \lfloor c \log n\rfloor} e_i \ e_j^T A^{p \lfloor c \log n \rfloor}1}\\
&= \frac{u_i p_j}{p_\ast} \frac{\lambda^{p\lfloor c \log n \rfloor}}{e_\ast^T A^{p \lfloor c \log n \rfloor}e_i} \frac{\lambda^{p\lfloor c \log n \rfloor}}{e_j^T A^{p \lfloor c \log n \rfloor}1} \frac{e_\ast^T A^{np} 1}{\lambda^{np}}.
\end{align*}
Now by the estimates \eqref{eq.quad.est} and Remark \ref{remark:exp} we see that this quotient is equal to
\begin{equation}\label{eq.quad.est2}
\frac{u_i p_j}{p_\ast} \cdot \left(\frac{1}{u_i} +O(n^{-1/2}) \right) \cdot \left(\frac{1}{p_j} +O(n^{-1/2})\right) \cdot \left( p_\ast +O(\theta^n)\right) = 1 +O(n^{-1/2})
\end{equation}
for some $0<\theta<1.$ Here we have used that $u_i>0$ and $p_j>0$ (the former is assumed, the latter follows since $v_j$ is assumed to belong to maximal component of $A^p$). Since there are only finitely many vertices in $\mathcal{G}$ and the left-hand-side of \eqref{eq.quad.est2} only depends on vertices of $\mathcal{G}$, we deduce that
\begin{equation}\label{eq.quad.est3}
\sup_{\gamma} \left| \frac{\pi_{n'}(\gamma)}{\widetilde{\tau}_{np}^c(\gamma)} - 1 \right| = O(n^{-1/2})
\end{equation}
where the supremum is taken over all paths of length $n'$ that lie entirely in a single $A^p$ maximal component. Now note that, given arbitrary two vertices $v_i$ and $v_j$, if $S_{n'}^{ij}$ denotes the set of paths of length $n'$ from $v_i$ to $v_j$ then
\begin{equation}\label{eq.above.discussion}
\begin{aligned}
    \widetilde{\tau}_{np}^c(S_{n'}^{ij}) &= \frac{(e_i^T A^{n'} e_j )(e_\ast^T A^{p \lfloor c \log n \rfloor} e_i) ( e_j^T A^{p \lfloor c \log n\rfloor}1)}{e_\ast^T A^{np} 1}\\
    &\le \frac{(e_i^T A^{n'} 1 )(e_\ast^T A^{p \lfloor c \log n \rfloor} e_i) ( e_j^T A^{p \lfloor c \log n \rfloor}1)}{e_\ast^T A^{np} 1}=p_iu_ip_j/p_\ast + O(n^{-1/2}),
\end{aligned}
\end{equation}
as $n \to \infty$, where the last equality follows (as in \eqref{eq.quad.est2}) by our estimates \eqref{eq.quad.est} and Remark \ref{remark:exp}. The limit $p_iu_ip_j/p_\ast$ in \eqref{eq.above.discussion} is equal to $0$ unless both $v_i,v_j$ belong to the same $A^p$ maximal component and $u_i>0$. Letting $L_n$ denote all paths of length $n'$ that lie entirely in an $A^p$ maximal component and start at any vertex $v_i$ with $u_i >0$. We have, for any set $R$ consisting of length $n'$ paths,
\begin{equation*} 
|\pi_{n'}(R) - \widetilde{\tau}_{np}^c(R)| \le \sum_{\gamma \in R \cap L_n} |\pi_{n'}(\gamma) - \widetilde{\tau}_{np}^c(\gamma)| + \widetilde{\tau}_{np}^c(R \backslash L_n).
\end{equation*}
Here we have used that $\pi_{n'}(R \backslash L_n) = 0$ which we can see holds from the definition of $\pi_{n'}$. To conclude the proof we note that, from \eqref{eq.above.discussion}, $\widetilde{\tau}_{np}^c(R \backslash L_n) =O(n^{-1/2})$ and 
\begin{align*}
\sum_{\gamma \in R \cap L_n} |\pi_{n'}(\gamma) - \widetilde{\tau}_{np}^c(\gamma)| &= \hspace{-2mm} \sum_{\gamma \in R \cap L_n} \left|\frac{\pi_{n'}(\gamma)}{\widetilde{\tau}_{np}^c(\gamma)} \ \widetilde{\tau}_{np}^c(\gamma) - \widetilde{\tau}_{np}^c(\gamma)\right| \\
&\le \sup_{\gamma \in R \cap L_n} \left| \frac{\pi_{n'}(\gamma)}{\widetilde{\tau}_n^\epsilon(\gamma)} - 1 \right| = O(n^{-1/2}),
\end{align*} where we used \eqref{eq.quad.est3} in the last equality and implied error term constants are independent of the $R$. This completes the proof.
\end{proof}

\begin{remark} \label{remark.newbasepoint}
So far this section has been concerned with comparing the measures $\pi_{np}$, $\tau_{np}$ and $\widetilde{\tau}_{np}^c$.  Each of these measures are constructed with the $\ast$ vertex as their `base point', i.e.~ $\pi_{np}$ is constructed using the $e_\ast$ vector and $\tau_{np}$, $\widetilde{\tau}^c_{np}$ can be seen as counting measures on the paths in $\mathcal{G}$ starting at $\ast$ (as indicated in their constructions). If we replace the $\ast$ vertex with any other vertex $v_0$ of large growth, that is a vertex $v_0$ such that $e_{v_0}^T A^n 1 \ge C \lambda^n$ for some $C>0$ and all $n\ge 1$, then we can construct measures analogous to $\pi_{np}$, $\tau_{np}$ and $\widetilde{\tau}^c_{np}$ but with $v_0$ being the new `base point'. To do this, one replaces $e_\ast$ with $e_v$ in the construction of $\pi_{np}$ and alters $\tau_{np}$ and $\widetilde{\tau}^c_{np}$ so that they count with respect to paths starting at $v_0$ instead of $\ast$. This new construction will yield different measures however all of the results that we have seen so far in this section will also hold for these measures. 
\end{remark}

We can now prove a counting CLT with error term for the sequence of spheres $(S_{np})_{n \in \N}$.

\begin{proposition} \label{prop:clt}
There exists $\Lambda, \sigma^2 > 0$ such that
\[
\tau_{np}\left( g \in \Gamma: \frac{\log\|\rho(g)\| - \Lambda|g|}{\sqrt{|g|}} \le t\right) = N(t,\sigma) + O\left( \frac{\log n}{\sqrt{n}}\right)
\]
as $n\to\infty$.
\end{proposition}

\begin{proof}
Let $\Omega$ denote the set of finite paths in $\mathcal{G}$ and for $g \in \Omega$ let $\overline{g}$ denote the group element corresponding to $g$ via the labeling map. For $t\in\mathbb{R}$, let $E(t)$ and $\widehat{E}(t)$ be the sets
\[
\left\{ g \in \Gamma: \frac{\log\|\rho(g)\| - \Lambda |g|}{\sqrt{|g|}} \le t \right\} \ \text{ and } \  \left\{ g \in \Omega: \frac{\log\|\rho(\overline{g})\| - \Lambda |\overline{g}|}{\sqrt{|\overline{g}|}} \le t \right\}
\]
respectively and let $c>0$ be a constant given by Lemma \ref{lem:approx2}. For each $n \in \N$, we factorise each path (or element) $g$ of length $np$ as a concatenation (resp.~ product) $g_0g_1g_2$ where $g_0$, $g_1$ and $g_2$ are the sub-paths (resp.~ factors) of $g$ of length $p\lfloor c \log n\rfloor$, $np-2p\lfloor c \log n \rfloor$ and $p\lfloor c \log n\rfloor$ respectively.  Writing $\tau_{np}(E(t)) = \tau_{np}(g = g_0g_1g_2 \in E(t))$ and using submultiplicativity of the matrix norm $\|\cdot\|$ we deduce that there exists $C>0$ such that $\tau_{np}(E(t))$ is bounded above and below by
\[
\tau_{np}(g=g_0g_1g_2: g_1 \in E(t + Cn^{-1/2} \log n))  \text{ and }  \tau_{np}(g=g_0g_1g_2: g_1 \in E(t - Cn^{-1/2} \log n))
\]
respectively. Now note that by the definition of $\widetilde{\tau}_{np}^c$ and by Lemma \ref{lem:approx2}
\begin{equation*}
\begin{aligned}
\tau_{np}\{g=g_0g_1g_2: g_1 \in E(t \pm Cn^{-1/2} \log n) \}&= \widetilde{\tau}_{np}^c(\widehat{E}(t \pm Cn^{-1/2}\log n))\\  &\hspace{-1cm}= \pi_{pn-2p\lfloor c\log n \rfloor}(\widehat{E}(t \pm Cn^{-1/2}\log n)) +O( n^{-1/2}).
\end{aligned}
\end{equation*}

On the other hand, by Lemma \ref{lem:periodic.berry} and Remark \ref{rk.Pp.to.pi}, we get that
\begin{align*}
\pi_{pn-2p\lfloor c\log n \rfloor}(\widehat{E}(t \pm Cn^{-1/2}\log n)) &=  N(t\pm Cn^{-1/2}\log n,\sigma) + O\left( \frac{\log n}{\sqrt{n}}\right)\\
&=N(t,\sigma) + O\left( \frac{\log n}{\sqrt{n}}\right)
\end{align*}
as $n\to\infty$ uniformly in $t \in \mathbb{R}$. The last line follows from the fact that the normal distribution has uniformly bounded derivative. The proof is completed by combining the last two displayed equations.
\end{proof}

Using the same ideas we can also prove the following. Given a vertex $v$ in $\mathcal{G}$ recall that we say that $v$ is of large growth if the number of length $n$ paths in $\mathcal{G}$ starting at $v$ grows at least like $C\lambda^{n}$ for some $C>0$, i.e. $e_v^T A^n 1 \ge C \lambda^n$.

\begin{proposition} \label{prop:cltvertex}
Suppose $v$ is a vertex of large growth. Suppose $\tau_{np}^{v}$ is the uniform counting measure on the paths in $\mathcal{G}$ of length $np$ starting at $v$. Let $\Lambda, \sigma^2 >0$ be the constants in Proposition \ref{prop:clt}. Then
\[
\tau_{np}^{v}\left( g \in \Omega :  \frac{\log\|\rho(\overline{g})\| - \Lambda|\overline{g}|}{\sqrt{|\overline{g}|}} \le x \right) = N(x,\sigma) + O\left( \frac{\log n}{\sqrt{n}}\right)
\]
as $n\to\infty$ where $\Omega$ represents the set of finite paths in $\mathcal{G}$ and for $g \in \Omega,$ $\overline{g} \in \Gamma$ is the group element corresponding to multiplying the edge labelings in $g$.
\end{proposition}

\begin{proof}
When $v = \ast$ this proposition is precisely Proposition \ref{prop:clt}. The proof of this more general result follows the same method used to prove Proposition \ref{prop:clt} but we consider the `initial vertex' to be $v$ instead of $\ast$. We define the counting measures $\widetilde{\tau}_{np}^c$ and $\pi_{kp}$ as before, but we replace the vector $e_\ast$ with the vector $e_v$ in their definitions, see Remark \ref{remark.newbasepoint}. We can then prove analogous results, such as Lemma \ref{lem:approx2} for these measures and then carry out the same proof. 
\end{proof}

Finally, we are in a position to prove our central limit theorem.

\begin{proof}[Proof of Theorem \ref{thm:clt}]
For $t \in \mathbb{R}$, let $E(t)$ denote the set defined in Proposition \ref{prop:clt}. Fix $r \in \{0,\ldots,p-1\}$. For each $g_0 \in \Gamma$ with $|g_0|_S = r$ let $t(g_0)$ denote the terminal vertex in the path corresponding to $g_0$ which begins with $\ast$ in $\mathcal{G}$. It follows from the definition of $\mu_r$ that if $t(g_0)$ is not a vertex of large growth then $\mu_r(g_0)=0$. Then, by definition of the measure $\mu$ and Proposition \ref{prop:cltvertex}
\begin{align*}
\mu_{np+r} (E(t)) &= \sum_{|g_0| = r } \mu_r(g_0) \tau_{np}^{t(g_0)}(g_1: g_0g_1 \in E(t))\\
&= \sum_{|g_0| = r } \mu_r(g_0) \left(N(t,\sigma) +O\left( \frac{\log n}{\sqrt{n}}\right)\right)\\
&= N(t,\sigma) +O\left( \frac{\log n}{\sqrt{n}}\right)
\end{align*}
$n\to\infty$ and where the implied constant is independent of $t \in \mathbb{R}$. It then follows from Lemma \ref{lem:TV} that $\tau_{np+r}(E(t)) = N(t,\sigma) +O\left( n^{-1/2}\log n\right)$. Since this holds for each $r=0,\ldots,p-1$, our theorem follows.
\end{proof}



\section{On a question of Kaimanovich--Kapovich--Schupp}\label{sec.last}

Here we briefly discuss some consequences of our results which pertain to the growth indicator functions, and make a connection between these and a result of Lubotzky--Mozes--Raghunathan \cite{LMR1,LMR2}. These consequences provide an affirmative answer to a question of Kaimanovich--Kapovich--Schupp \cite{kaimanovich-kapovich-schupp} that was also raised in our precise setting in Sert's thesis \cite{sert.thesis}.

\subsection{Unique maximum of growth indicator}\label{subsec.unique.max}

We will formulate the consequences using the language of reductive real linear algebraic groups. For definitions of the notions and objects we use, we refer the reader to  \cite{bq.book}. The reader is invited to consider the case $G=\SL_d(\R)$ or $\GL_d(\R)$ in which case we will specify the relevant objects.

Let $G$ be the group of real points of a connected reductive affine algebraic group defined over $\R$ (we will shortly refer to such a group as a real reductive Lie group). Let $\mathfrak{a}^+$ be a Weyl chamber in a Cartan subspace of the Lie algebra of $G$ and $\overrightarrow{\kappa}: G \to \mathfrak{a}^+$ the associated Cartan projection. For the case of $G=\GL_d(\R)$ or $\SL_d(\R)$, one can define for $g \in G$, 
$$
\overrightarrow{\kappa}(g)=(\log \sigma_1(g),\ldots,\log \sigma_d(g)),
$$
where $\sigma_i(g)$'s are the singular values of $g$ in decreasing order and $\mathfrak{a}^+$ to be the  cone in $\R^d$ given by $x_1 \geq \ldots \geq x_d$ in the case of $\GL_d(\R)$ and the intersection of this cone with the subspace $x_1+\ldots+x_d=0$ in the case of $\SL_d(\R)$. Denote by $\mathfrak{a}^{++}$ the interior of $\mathfrak{a}^+$. We clarify that these definitions only differ by a linear change of coordinates from the more standard definitions in \cite{bq.book} and the results discussed below are independent of the choice of coordinates up to affine transformations.

A direct corollary of  Theorem \ref{thm:lln} is the following.
\begin{corollary}\label{corol.cartan.weak.law}
Let $G$ be a real reductive Lie group, $\overrightarrow{\kappa}:G \to \mathfrak{a}^+$ a Cartan projection of $G$. Let $\Gamma$ be a Gromov-hyperbolic group, $\rho:\Gamma \to G$ a representation with Zariski-dense image. For every finite symmetric generating set $S$ of $\Gamma$, there exists $\overrightarrow{\Lambda} \in \mathfrak{a}^{++}$ such that
$$
\frac{1}{n} \sum_{\gamma \in S_n} \frac{1}{\# S_n} \overrightarrow{\kappa}(\rho(\gamma)) \underset{n \to \infty}{\longrightarrow} \overrightarrow{\Lambda}.
$$
\end{corollary}

\begin{proof}
The convergence is a straightforward consequence of \cite[Lemma 8.15 $\&$ Lemma 8.17]{bq.book} together with Theorem \ref{thm:lln}. The fact that $\overrightarrow{\Lambda} \in \mathfrak{a}^{++}$ is obtained using additionally Proposition \ref{prop.positivity}. The details are standard and omitted. 
\end{proof}

It might be possible to prove the above convergence under the same assumptions when we replace the Cartan projection $\overrightarrow{\kappa}$ with the Jordan projection $\overrightarrow{\lambda}$. However, even in the case of Markovian random matrix products, the law of large numbers for the spectral radius may fail (see \cite{aoun-sert.lln}) and one has to deal with this difficulty. On the other hand, in ongoing work with Cipriano and Dougall \cite{CCDS}, we show that the above convergence holds for both $\overrightarrow{\kappa}$ and $\overrightarrow{\lambda}$ with a speed estimate under the assumption that the representation $\rho$ is Anosov (with respect to an appropriate sense parabolic subgroup). Finally, in the previous result, one may prove the stronger statement that $\overrightarrow{\Lambda}$ belongs to the interior of the joint spectrum $J(S)$ of $S$ (see \cite{breuillard-sert}). We will however content with the above version for brevity.

\bigskip

We now turn to a consequence of our large deviation estimate Theorem \ref{thm:ldt}, its connection to the uniqueness of the maximum of the growth indicator function and the connection between the latter and a question of Kaimanovich--Kapovich--Schupp \cite[Problem 9.3]{kaimanovich-kapovich-schupp}. In the latter, the authors proved (see also an earlier related consideration in \cite{kapovich-schupp-shpilrain}) that if in Theorem \ref{thm:lln}, one considers $\Gamma$ to be a free group with a free generating set $S$ and having fixed an automorphism $\phi:F \to F$, one takes $\varphi:F \to \R$ to be the function $w \mapsto |\phi(w)|_S$, then the convergence in Theorem \ref{thm:lln} is exponential. In \cite[Problem 9.3]{kaimanovich-kapovich-schupp}, the authors ask the question of whether there are other examples where this convergence is exponential for a map on a free group. Theorems \ref{thm:ldt} and \ref{thm.gromov.counting.ld} clearly provide positive answers in a more general (both for underlying groups and generating sets) setting. The existence of this kind of phenomenon was also asked in \cite[Introduction 7.4.3]{sert.thesis} with the language of growth indicator of a finite set, a notion that was introduced therein (see also \cite{sert.ongoing}). We now briefly recall this notion and formulate the consequence of our counting large deviation result.

Let $G$ be a real reductive Lie group, $\mathfrak{a}^+$ a Weyl chamber of $G$, $\Gamma<G$ a finitely generated subgroup and $S$ be a finite generating set $\Gamma$. We define the growth indicator of $S$ as:
\begin{equation*}
\begin{aligned}
\varphi_S: \mathfrak{a}^+ & \to [0,\infty) \cup \{-\infty\} \\
 \alpha & \mapsto \inf_{\alpha \in O} \limsup_{n \to \infty} \frac{1}{n}\log \#\left\{g \in S_n \; | \; \frac{1}{n}\overrightarrow{\kappa}(g) \in O\right\},
\end{aligned}
\end{equation*}
where $O$ ranges over neighborhoods of $\alpha$ in the Weyl chamber $\mathfrak{a}^+$. If $\Gamma$ is Zariski-dense, the closure of the locus of points $x \in \mathfrak{a}^+$ on which $\varphi_S$ takes values in $[0,\infty)$ is contained in the joint spectrum of $S$ (\cite{breuillard-sert}), which is a convex body in $\mathfrak{a}^+$.
On the other hand, denoting by $\lambda_S>1$ the exponential growth rate of the cardinality of $S_n$, the function $\varphi_S$ is bounded above by $\log \lambda_S$. Moreover, it is not hard to see that the value $\log \lambda_S$ is always attained by $\varphi_S$. In this general setting, the locus of maxima, i.e.~ the description of the set $\varphi_S^{-1}(\{\log \lambda_S\})$ remains to be studied. Thanks to our Theorem \ref{thm:ldt}, we can describe it in the setting of Corollary \ref{corol.cartan.weak.law}. Indeed, the conclusion of the latter implies that $\varphi_S(\overrightarrow{\Lambda})=\log \lambda_S$ where $\overrightarrow{\Lambda} \in \mathfrak{a}^{++}$ is given by that corollary and the following consequence of Theorem \ref{thm:ldt} says that $\varphi_S$ attains its maximum only on $\overrightarrow{\Lambda}$  which is precisely the aforementioned positive answer to \cite[Problem 9.3]{kaimanovich-kapovich-schupp}.

\begin{corollary}\label{corol.cartan.unique.max}
Let $G$ be a real reductive Lie group, $\overrightarrow{\kappa}:G \to \mathfrak{a}^+$ a Cartan projection, $\Gamma<G$ a Zariski-dense Gromov-hyperbolic subgroup and $S$ a finite symmetric generating set of $\Gamma$. Let $\varphi_S:\mathfrak{a}^+ \to [0,\infty) \cup \{-\infty\}$ be the growth indicator of $S$. Then, the Weyl chamber element $\overrightarrow{\Lambda} \in \mathfrak{a}^{++}$ given by Corollary \ref{corol.cartan.weak.law} is the unique point where $\varphi_S$ reaches its maximum value $\log \lambda_S $. \qed
\end{corollary}

\subsection{A connection to the work of Lubotzky--Mozes--Raghunathan}\label{subsec.LMR}

Here we let $G$ be a connected semisimple real Lie group and $\Gamma<G$ a finitely generated Zariski-dense subgroup, endowed with a finite symmetric generating set $S$. Let $K<G$ be a maximal compact subgroup and $d_G$ a left-$G$-invariant and bi-$K$-invariant Riemannian metric on $G$ induced by the Killing form. If $\Gamma$ is a uniform lattice in $G$, then it is not hard to see that the word-metric $d_S$ is Lipschitz equivalent $d_G$ (see e.g.~ \cite[Proposition 3.2]{LMR2}). The situation is much less clear for non-uniform lattices. Confirming a conjecture of Kazhdan (see \cite{gromov}), Lubotzky--Mozes--Raghunathan \cite{LMR2} have shown that if $G$ has $\R$-rank at least two and $\Gamma$ is an irreducible lattice in $G$, then $d_S$ and $d_G$ are Lipschitz equivalent. In other words, there is a constant $C>1$ such that for every $n \in \N$ and $g \in S_n$, we have
\begin{equation*}
    C^{-1}n \leq d_G(g,\id) \leq Cn.
\end{equation*}

This equivalence breaks down for rank-one simple Lie groups in which case the word-metric $d_S$ of a (non-uniform) lattice can be exponentially distorted in the terminology of \cite[\S 3]{gromov}. This is for example the case for $\SL_2(\Z)<\SL_2(\R)$. When $\Gamma$ is only required to be Zariski-dense, the connection between $d_S$ and $d_G$ is much less clear. In many cases (e.g.~ if $\Gamma$ is not discrete), it does not make sense to ask for Lipschitz equivalence \textit{of every element of $\Gamma$} as there can be elements of $\Gamma$ with arbitrarily large $d_S$-length but small $d_G$-length. 

One way to study the connection between $d_S$ and $d_G$, despite the fact they can not be Lipschitz equivalent, is to ask whether there is an equivalence $d_S \sim d_G$ for \textit{most of the elements of $\Gamma$}. Our counting large deviation Theorem \ref{thm:ldt} (and Corollary \ref{corol.cartan.unique.max}) then have the following consequence which establishes such a statistical relation between $d_S$ and $d_G$ for Gromov-hyperbolic groups.

\begin{corollary}\label{corol.LMR}
Let $\Gamma$ be a Zariski-dense, non-elementary Gromov-hyperbolic subgroup of a real semisimple Lie group $G$. Then, for every finite symmetric generating set $S$ of $\Gamma$ and constant $\epsilon>0$, there exists a subset $T_\epsilon$ of $\Gamma$ with the property that
\begin{equation}\label{eq.exp.generic}
\frac{\# (S_n \setminus (T_\epsilon \cap S_n))}{\# S_n}=O(e^{-\alpha n})
\end{equation}
for some $\alpha>0$, and there exists a constant $\Lambda=\Lambda(S) >0$ such that for every $n \in \N$ and $g \in S_n \cap T_\epsilon$, we have
\begin{equation}\label{eq.last.last}
n(\Lambda-\epsilon) \leq d_G(g,\id) \leq n(\Lambda +\epsilon).
\end{equation}
\end{corollary}

A subset of $\Gamma$ satisfying \eqref{eq.exp.generic} can be called $S$-\textit{exponentially generic} in $\Gamma$ in the terminology of \cite{kaimanovich-kapovich-schupp}. 

\begin{proof}
It suffices to work with the symmetric space $G/K$ and the $G$-invariant metric $d_{G/K}$ induced by the Killing form. Let $\mathfrak{a}^+$ be a Weyl chamber in a Cartan subspace $\mathfrak{a}$ of the Lie algebra $\mathfrak{g}$ of $G$ such that we have the Cartan decomposition $K\exp(\mathfrak{a}^+)K$. Denoting by $\|\cdot\|$ the norm induced by the Killing form on $\mathfrak{a}$, by \cite[\S 6.7.4]{bq.book}, for any $g \in G$, we have $\|\overrightarrow{\kappa}(g)\|=d_{G/K}(g \cdot o,o)$. The result now follows from Corollary \ref{corol.cartan.unique.max}.
\end{proof}

\begin{remark}
By replacing the use of Theorem \ref{thm:ldt} (in the form of Corollary \ref{corol.cartan.unique.max}) by Theorem \ref{thm:lln}, one can obtain a version of Corollary \ref{corol.LMR} valid for any left-$G$-invariant Riemannian metric $d$ on $G$ but $T_\epsilon$ being only $S$-\textit{generic} for $\Gamma$ instead of $S$-exponentially generic. Here, by $S$-generic for $\Gamma$, we understand a subset satisfying \eqref{eq.exp.generic} with $O(e^{-\alpha n})$ replaced by $o(1)$).
\end{remark}

\end{document}